\documentclass[reqno, a4paper]{amsart}
\usepackage[T1]{fontenc}
\usepackage{amssymb, mathdots}
\usepackage{mathabx}
\usepackage{mathrsfs}
\usepackage[utf8]{inputenc}
\usepackage{amsmath,  graphicx, tensor}
\usepackage{tikz}
\usetikzlibrary{matrix, arrows.meta}
\usetikzlibrary{cd}

\usepackage{enumerate}
\usepackage{multirow}
\usepackage{hyperref}
\DeclareSymbolFont{bbold}{U}{bbold}{m}{n}
\DeclareSymbolFontAlphabet{\mathbbold}{bbold}
\def\qmod#1#2{{\hbox{}^{\displaystyle{#1}}}\!\big/\!\hbox{}_{
\displaystyle{#2}}}

 \def\psp#1#2%
  {\mathop{}%
   \mathopen{\vphantom{#2}}^{#1}%
   \kern-\scriptspace%
    \hskip -0.3mm{#2}} 
 \def\psb#1#2%
  {\mathop{}%
   \mathopen{\vphantom{#2}}_{#1}%
   \kern-\scriptspace%
    \hskip -0.0mm{#2}} 
 \def\pscr#1#2#3%
  {\mathop{}%
   \mathopen{\vphantom{#3}}^{#1}_{#2}%
   \kern-\scriptspace%
    \hskip -0.3mm{#3}} 

\def\A{{\mathbb A}}

\def\C{{\mathbb C}}

\def\G{{\mathbb G}}

\def\N{{\mathbb N}}

\def\Q{{\mathbb Q}}
\def\R{{\mathbb R}}

\def\Z{{\mathbb Z}}

\def\map{\longrightarrow}
\def\textmap#1{\mathop{\vbox{\ialign{
                                  ##\crcr
      ${\scriptstyle\hfil\;\;#1\;\;\hfil}$\crcr
      \noalign{\kern 1pt\nointerlineskip}
      \rightarrowfill\crcr}}\;}}
\def\bigtextmap#1{\mathop{\vbox{\ialign{
                                  ##\crcr
      ${\hfil\;\;#1\;\;\hfil}$\crcr
      \noalign{\kern 1pt\nointerlineskip}
      \rightarrowfill\crcr}}\;}}
      
\newcommand{\cal}{\mathcal}
\def\textlmap#1{\mathop{\vbox{\ialign{
                                  ##\crcr
      ${\scriptstyle\hfil\;\;#1\;\;\hfil}$\crcr
      \noalign{\kern-1pt\nointerlineskip}
      \leftarrowfill\crcr}}\;}}

\def\cg{{\mathfrak c}}

\def\lg{{\mathfrak l}}
\def\mg{{\mathfrak m}}
\def\ng{{\mathfrak n}}

\def\pg{{\mathfrak p}}

\def\tg{{\mathfrak t}}

\def\Ag{{\mathfrak A}}
\def\Bg{{\mathfrak B}}
\def\Cg{{\mathfrak C}}

\def\Mg{{\mathfrak M}}
\def\Ng{{\mathfrak N}}

\def\Pg{{\mathfrak P}}

\def\Tg{{\mathfrak T}}

\theoremstyle{remark}
\newtheorem{ex}{Example}[section]
\newtheorem*{pb}{Problem}

\newtheorem{sz}{Satz}[section]
\theoremstyle{remark}

\theoremstyle{plain}
\newtheorem{thry}[sz]{Theorem}
\newtheorem{pr}[sz]{Proposition}
\newtheorem{co}[sz]{Corollary}
\newtheorem{dt}[sz]{Definition}
\newtheorem{lm}[sz]{Lemma}

\newtheorem{re}[sz]{Remark}

\def\tr{\mathrm {Tr}}
\def\End{\mathrm {End}}
\def\Aut{\mathrm {Aut}}

\def\GL{\mathrm {GL}}
\def\PGL{\mathrm {PGL}}

\def\SL{\mathrm {SL}}
\def\PSL{\mathrm {PSL}}

\def\Pic{\mathrm {Pic}}

\def\Hom{\mathrm{Hom}}

\def\Tors{\mathrm{Tors}}
\def\T{\mathrm{T}}

\def\id{ \mathrm{id}}

\def\p{\mathrm{p}}

\def\Ext{\mathrm{Ext}}

\def\Spec{\mathrm{Spec}}
\newcommand\smvee{{\hskip -0.1ex \raise 0.2ex\hbox{$\scriptscriptstyle\vee$}}\hskip -0,3ex}

\def\Aff{{\rm Aff}}

\def\bpm{\begin{pmatrix}}
\def\epm{\end{pmatrix}}	
\def\edf{\coloneq}

\def\trp#1{\tensor[^{\mathrm{t}}]{\hskip -0.3ex#1}{}}

\def\cof{\mathrm{cof}}
\def\trcof{\trp{\,\cof}}

\begin{document}

\title{On the classification of Inoue surfaces}
\author{Zahraa Khaled \and Andrei Teleman}
\thanks{We are  grateful to Matei Toma for his useful suggestions concerning the formalism used in the construction of type I Inoue surfaces, and we  thank the referee for pointing  out the important reference \cite{In2} }

\address{Aix Marseille Univ, CNRS, I2M,  Marseille, France, email: zahraa.k.khaled@gmail.com, andrei.teleman@univ-amu.fr}

\begin{abstract} We prove that any Inoue surface admits a unique holomorphic connection and we infer that two Inoue surfaces $S=H\times\mathbb{C}/G$, $S'=H\times\mathbb{C}/G'$ are biholomorphic if and only if   $G$, $G'$ are conjugate in the group of affine transformations of $H\times\mathbb{C}$. This result allows us to prove explicit classification theorems for Inoue surfaces: Let $\mathcal{M}$ be the set of $\mathrm{SL}(3,\mathbb{Z})$-matrices $M$ with a real eigenvalue $\alpha>1$ and two non-real eigenvalues, and let $\mathcal{N}^\pm $ the set of $\mathrm{GL}(2,\mathbb{Z})$-matrices $N$ with a real eigenvalue $\alpha>1$  and $\det(N)=\pm 1$.
	
We prove that: 
\begin{enumerate}
	\item For any $\mathrm{GL}(3,\mathbb{Z})$-similarity class $\mathfrak{M}\in \mathcal{M}/\sim$, there exists exactly two biholomorphism classes of type I Inoue surfaces. 
\item \begin{itemize}
	\item[(+)] For any $\mathrm{GL}(2,\mathbb{Z})$ similarity class $\mathfrak{N}=[N]\in \mathcal{N}^+/\sim$ and positive integer $r\in\mathbb{N}^*$, we have a finite set of deformation classes of type II Inoue surfaces. This set is parameterised by the quotient of $\mathbb{Z}^2/(I_2-N)\mathbb{Z}^2+r\mathbb{Z}^2$ by an action of the ``positive centraliser" $Z^+_{\mathrm{GL}(2,\mathbb{Z})}(N)$ of $N$ in $\mathrm{GL}(2,\mathbb{Z})$. The set of biholomorphism types corresponding to a deformation class, endowed with its  natural topology, can be identified with either $\mathbb{C}^*$ or $\mathbb{C}$.
	\item[(-)]   For any $\mathrm{GL}(2,\mathbb{Z})$-similarity class $\mathfrak{N}=[N]\in \mathcal{N}^-/\sim$ and positive integer $r\in\mathbb{N}^*$, we have a finite set of biholomorphism classes of type III Inoue surfaces. This set is parameterised by the quotient of $\mathbb{Z}^2/(I_2+N)\mathbb{Z}^2+r\mathbb{Z}^2$ by an action of   $Z^+_{\mathrm{GL}(2,\mathbb{Z})}(N)$.	 \end{itemize} 
	 In both cases $Z^+_{\mathrm{GL}(2,\mathbb{Z})}(N)$-can be identified with the stabiliser of $N$ in $\mathrm{PGL}(2,\mathbb{Z})$ and  is an infinite cyclic group (see section \ref{Z+-section}). 
\end{enumerate}
\vspace{1mm}
Taking into account the Latimer-MacDuffee theorem and a classical finiteness theorem for ideal classes in orders, it follows that:
\begin{enumerate}
\item 	For  any  polynomial $\chi\in \chi(\mathcal{M})$ we have only finitely many biholomorphism classes of type I surfaces.
\item For any  pair  $(\chi,r)\in   \chi(\mathcal{N}^+)\times\mathbb{N}^*$ we have only finitely many deformation classes of type II Inoue surfaces.
\item For any pair  $(\chi,r)\in \chi(\mathcal{N}^-)\times\mathbb{N}^*$ we have only finitely many biholomorphism classes of type   III  Inoue surfaces.
\end{enumerate}

\end{abstract}


\maketitle

\tableofcontents

\section{Introduction} \label{intro}

In his renowned article \cite{In}, Inoue introduced three  types  of non-Kählerian surfaces obtained as quotients of $U\edf H\times\C$ (where $H\edf\{w\in \C|\ \Im(w)>0\}$) by certain groups of affine transformations   acting properly discontinuously on $U$. Moreover Inoue proves that any  surface $S$ belonging to one of his three families is a class VII surface with $b_2=0$ (\cite{Ko1}-\cite{Ko3}, \cite{BHPV}, \cite{Na}, \cite{Te-survey}) and   has the following remarkable geometric properties:
\begin{enumerate}[(A)]
\item 	$S$ admits a holomorphic foliation, or, equivalently, the tangent sheaf admits an invertible subsheaf (see \cite[Condition (C) p. 280]{In} and section \ref{tangent-bundle-sect} in this article).
\item $S$ contains no curve, or, equivalently, for any nontrivial invertible sheaf ${\cal N}$ on $S$, one has $H^0(S,{\cal N})=0$ (see \cite{In} and Theorem \ref{ThInoue} in this article).
\end{enumerate}

The fondamental theorem \cite[Theorem p. 280]{In} states that,  conversely, any class VII surface  with $b_2=0$ having these two properties is biholomorphic to a an Inoue surface, i.e. to a surface belonging to one of Inoue's three families. Note  that in fact,  by the main result of \cite{Te-Bo}, any class VII surface with $b_2=0$ satisfying (B) is biholomorphic to an Inoue surface.  Taking into account Kodaira's classification of class VII surfaces with $b_2=0$ admitting curves \cite[Theorem 34, p. 699]{Ko2}, this result shows that the classification of class VII surfaces with $b_2=0$ reduces to a detailed classification (i.e. an explicit parametrisation of the set of biholomorphism classes) of Hopf surfaces and Inoue surfaces. For primary Hopf surfaces such a detailed classification can be found in \cite{We}. 
\vspace{1mm}

 Inoue surfaces  have many interesting properties, for instance:
 \begin{itemize}
	\item[-] Any Inoue surface admits a geometric structure compatible with its complex structure \cite{Wa}. More precisely  a type I Inoue surface admits a compatible geometric structure of type $\mathrm{Sol}^4_0$,  a type II Inoue surface with parameter $t\in \R$ and  a type III Inoue surface admits a compatible geometric structure of type $\mathrm{Sol}^4_1$, whereas a type II Inoue surface with parameter $t\in \C\setminus \R$ admits a compatible geometric structure of type $\mathrm{Sol}'^4_1$ (see \cite[Proposition 9.1]{Wa}).
	\item[-] The Bott-Chern class $c_1^{\rm BC}({\cal K}_S)$ of the canonical line bundle  ${\cal K}_S$  of any Inoue surface $S$ is non-trivial and pseudo-effective (see \cite[Remark 4.2]{Te-cone}). This might look surprising taking into account that   $\mathrm{kod}(S)=-\infty$.  
	\item[-] An Inoue surface $S$ comes with a tautological affine structure, which, by a result of Klingler, is its unique affine structure \cite[Lemma 4.3]{Kl}.  In other words, $S$ has a unique holomorphic connection $\nabla_S$ with vanishing torsion and curvature. A stronger result of Dumitrescu states that $\nabla_S$ is the unique holomorphic connection with vanishing torsion  \cite[Proposition 3.1]{Du} of $S$, whereas our Theorem \ref{HolConn} proved in this article states that $\nabla_S$ is its unique holomorphic connection  (without any restriction).  Compact complex surfaces admitting a holomorphic connections have been classified in \cite{IKO}.
\end{itemize}

We  also mention that  the  constructions which yield Inoue surfaces have been recently generalised in higher dimensions by Oeljeklaus-Toma  \cite{OT} and Miebach-Oeljeklaus \cite{MO} using number theoretical methods. The resulting higher dimensional generalisations of Inoue surfaces are called Oeljeklaus-Toma manifolds and generate already significant interest in the literature.

The class of Inoue surfaces plays a fundamental role in the theory of surfaces (see for instance \cite[Section V.19]{BHPV}) so, naturally, they are abundantly studied in the literature.  However, to our surprise, we did not find in the literature an explicit classification theorem for Inoue surfaces, as we have   for primary Hopf surfaces \cite{We}. 
 
Also to our surprise we realised that the problem is more difficult and more interesting than we expected. The  first difficulty  we encountered comes from the traditional notations used in the literature for these surfaces. For instance, the ``traditional'' notation used for  Inoue surfaces of the first type is $S_M$, where $M\in\SL(3,\Z)$ is a matrix with $\Spec_\C(M)\supsetneq  \Spec_\R(M)\subset ]1,+\infty[$ (see \cite[section 2]{In}). This notation suggests that for any such matrix $M$ one has a well defined complex surface $S_M$. This is not true. In fact, in order to construct an Inoue surface of the first type, one also has to fix a non-real eigenvalue $\beta\in  \Spec_\C(M)\setminus\R$ and a pair $(a,b)\in\R^3\times\C^3$ consisting of a real eigenvector for the real eigenvalue $\alpha=|\beta|^{-2}\in \Spec_\R(M)$ and an eigenvector for $\beta$. Having the triple  $(\beta,a,b)$ one can recover $M$, but not vice-versa. Therefore the  notation  which is appropriate for classification purposes  is $S^{\beta}_{a,b}$, not $S_M$. 
For the first type of surfaces our final result (see Theorem \ref{Class-Type-I-similitude}, Remark \ref{Class-Type-I-similitude-rem}) shows that biholomorphism classes of first type Inoue  surfaces correspond bijectively to pairs $(\Mg,\beta)$ consisting of a $\GL(3,\Z)$-similarity class of $\SL(3,\Z)$-matrices of the considered type and a non-real eigenvalue $\beta\in\Spec(\Mg)\setminus\R$. Therefore to any similarity class $\Mg$ of such matrices correspond {\it two} biholomorphism types of type 1 Inoue surfaces.   The fact that the data of a matrix $M$ as above is not sufficient to determine a biholomorphism class has already been  noticed by Inoue in \cite{In2}, who introduced the notations $S_M^\pm$  in order to emphasise that, replacing the eigenvalue $\beta$ by its conjugate, will change the isomorphism type of the obtained surface. We are grateful to the referee for pointing us out this  reference. Unfortunately this important remark has not been taken into account in the subsequent mathematical literature in English, so the traditional notation for these surfaces remained $S_M$ (see for instance \cite[p.  145-146]{Wa}).

 
 A concrete consequence of our classification result: for any type I Inoue surface $S$ we have $S\not\simeq\bar S$ (see Corollary \ref{S-non-isom-barS}), in particular $S$ does not admit any Real structure. This should be compared to the classification results for Real structures on other classes of non-Kählerian surfaces \cite{Fr}, \cite{Kh}. 

Note that, by a renowned Theorem of Latimer and MacDuffee (see \cite{LMD}, \cite{Ta}), for any  degree $n$ polynomial  $\chi\in\Z[X]$ which in irreducible in $\Q[X]$,  the set of $\GL(n,\Z)$-similarity classes of matrices with integer coefficients and  characteristic polynomial $\chi$ correspond bijectively to the set of ideal classes of the order $\Z[\alpha]$, where $\alpha$ is a root of $\chi$. Therefore, by a fundamental theorem in Number Theory \cite[Theorem 3, p. 128]{BoSh}, this set is  finite\footnote{This finiteness theorem is usually stated  for a maximal order, i.e. for the ring of integers $O_K$ of an algebraic number field $K$. }. 

Therefore, for any   polynomial $X^3-\theta_2 X^2+\theta_1X-1\in\Z[X]$ with a root   $\alpha\in ]1,+\infty[$ and two non-real roots, we have finitely many biholomorphism classes of type I Inoue surfaces, namely two biholomorphism classes for any equivalence class of ideals of the order $\Z[\alpha]$. We give an explicit example in which the order $\Z[\alpha]$ is not maximal (see Example \ref{Louboutin}).

\vspace{2mm}

For surfaces of the second type, the situation is even more complicated: the ``traditional'' notation for a type II Inoue surface is $S_{N,p,q,r,t}^+$, where $N\in\SL(2,\Z)$ has $\Spec(N)\subset (0,+\infty)\setminus\{1\}$ and $(p,q,r,t)\in\Z\times\Z\times\Z^*\times\C$. This notation suggests that any 5-tuple $(N,p,q,r,t)$ as above yields a well defined surface  denoted $S_{N,p,q,r,t}^+$. This again is not true; the construction also needs a pair $(a,b)\in\R^2\times\R^2$, where $a$   is an eigenvector for the eigenvalue $\alpha\in\Spec(N)\cap]1,+\infty[$ and $b$ is   an eigenvector for $\alpha^{-1}$.

For this class of surfaces we will adopt the notation 
$S_{a,b,c,t}^{\alpha,r}$, where:
\begin{itemize} 
	\item $t\in\C$, $r\in \N^*$ and $\alpha\in ]1,+\infty[$ is {\it $S^+$-admissible}, i.e. $\alpha+\alpha^{-1}\in\N$.
	\item $(a,b)\in \R^2\times\R^2$ is a linearly independent pair which is $\alpha$-{\it compatible}, i.e.
$$
N(\alpha,a,b)\coloneq \begin{pmatrix}
a_1 &b_1\\
a_2& b_2	
\end{pmatrix}\begin{pmatrix}
\alpha &0\\
0&\alpha^{-1}	
\end{pmatrix}\begin{pmatrix}
a_1 &b_1\\
a_2& b_2	
\end{pmatrix}^{-1}\in M_2(\Z).
$$
\item $c\in\R^2$ is $(a,b,r)$-compatible in the sense of the compatibility condition (\ref{CompCond}).
\end{itemize} 
The set of $(a,b,r)$-compatible vectors is identified with $\Z^2$ via a bijection which depends on $(a,b,r)$. Therefore we will replace Inoue's pair $(p,q)\in\Z^2$  by the pair of coefficients $c=(c_1,c_2)$ which intervenes effectively in the construction of the surface as a quotient \cite[p. 276]{In}.

For any $S^+$-admissible $\alpha\in ]1,+\infty[$, let ${\cal N}_\alpha\subset \SL(2,\Z)$ be the set of $\SL(2,\Z)$ matrices with $\alpha\in\Spec(N)$. By the Latimer and MacDuffee Theorem, the set of similarity classes ${\cal N}_\alpha/\sim$ is finite.

The pair $(\alpha,r)$ is a biholomorphism invariant. We will identify the set of biholomorphism classes of type II Inoue surfaces $S_{a,b,c,t}^{\alpha,r}$ with $N(\alpha,a,b)\in{\cal N}_\alpha$ with a quotient set $\mathscr{Q}_{\alpha,r}$ of the set of parameters  $(a,b,c,t)$,  where $(a,b,c)$ satisfies the above compatibility conditions (see Theorem \ref{Classif-Theorem-II}). We have a natural map 
$$
\Psi_{\alpha,r}:\mathscr{Q}_{\alpha,r}\to {\cal N}_\alpha/\sim, \ [S_{a,b,c,t}^{\alpha,r}]\mapsto [N(\alpha,a,b)].
$$ 

Let $\Ng\in  {\cal N}_\alpha/\sim$,  choose $N=N(\alpha,a,b)\in {\cal N}_\alpha$, and let $Z^+_{\GL(2,\Z)}(N)$ be ``the positive centraliser of $N$'' in $\GL(2,\Z)$, which is an infinite cyclic group,  see the definition formula (\ref{PosCentrDef}) and section \ref{Z+-section}. 
We will prove (see Theorem \ref{ConnCompOfFibre}) that the fibre $\Psi_{\alpha,r}^{-1}(\Ng)$ over   $\Ng$ comes with a natural topology and that the set of connected components of this fibre can be naturally identified   with the quotient of 
$$\Z_{N,r}\edf \Z^2/(I_2-N)\Z^2+r\Z^2$$
 by the  group  $Z^+_{\GL(2,\Z)}(N)$ acting on $\Z_{N,r}$ by $K*[p]\edf [\varepsilon_K Kp]$.
 
 Moreover, we will prove that any connected component of the fibre $\Psi_{\alpha,r}^{-1}(\Ng)$ can be identified to either $\C^*$ or $\C$, the latter case occurring when condition $(C)$ stated in  Proposition \ref{TheConnComp} is satisfied. 
 Using Inoue’s \cite[Proposition 4 p. 278]{In}
  one can prove easily that two type II Inoue surfaces $S_{a,b,c,t}^{\alpha,r}$, $S_{a',b',c',t'}^{\alpha',r'}$ are deformation equivalent 
 if and only if $(\alpha,r)=(\alpha',r')$ and the images of the the 4-tuples $(a,b,c,t)$, $(a',b',c',t')$ in the quotient space $\mathscr{Q}_{\alpha,r}$ belong to the same connected component. 
 
 Therefore, for any pair $(\alpha,r)$ as above with $\alpha$ admissible, we have only finitely many deformation classes of type II Inoue surfaces, namely a finite number of similarity classes of $\SL(2,\Z)$-matrices with eigenvalue  $\alpha$,  and a finite number of deformation classes for  each such similarity class.
 
 At the end of section \ref{ClassTypeII-section} we treat in detail the cases $\theta=3$ ($\alpha=\frac{3+\sqrt{5}}{2}$) and $\theta=4$ ($\alpha= 2+\sqrt{3}$) specifying the number of deformation classes corresponding to any $r\in\N^*$ and the isomorphism type of the connected components of the space of isomorphism classes. The main tools are: the algorithmic version of the Latimer-MacDuffee theorem of \cite{BVdM} and our Theorem \ref{Z+Calcul-eng} which specifies a generator of the positive centraliser  $Z^+_{\GL(2,\Z)}(N)$ for a large class of $\SL(2,\Z)$ matrices.

\vspace{3mm}

We will use similar methods for the classification of type III surfaces, i.e. of  the surfaces denoted traditionally by $S_{N,p,q,r}^-$. For this surfaces we will adopt the notation $S_{a,b,c}^{\alpha,r}$, where
\begin{itemize} 
	\item   $r\in \N^*$ and $\alpha\in ]1,+\infty[$ is {\it $S^-$-admissible}, i.e. $\alpha-\alpha^{-1}\in\N$.
	\item $(a,b)\in \R^2\times\R^2$ is a linearly independent pair which is $\alpha$-{\it compatible}, i.e.
$$
N(\alpha,a,b)\coloneq \begin{pmatrix}
a_1 &b_1\\
a_2& b_2	
\end{pmatrix}\begin{pmatrix}
\alpha &0\\
0&-\alpha^{-1}	
\end{pmatrix}\begin{pmatrix}
a_1 &b_1\\
a_2& b_2	
\end{pmatrix}^{-1}\in M_2(\Z).
$$
\item $c\in\R^2$ is $(a,b,r)$-compatible in the sense of  (\ref{CompCond-}).
\end{itemize}

We define ${\cal N}_\alpha^-\edf \{N\in\GL(2,\Z)|\ \det(N)=-1,\ \alpha\in\Spec(N)\}$ and we let $\mathscr{Q}_{\alpha,r}$ denote the set of biholomorphism classes of type III Inoue surfaces  $S^{\alpha,r}_{a,b,c}$ with $N(\alpha,a,b)\in {\cal N}^-_\alpha$. Our main result for this class of surfaces is Theorem \ref{FibreOverSimilClass-}, which states that the fibre of the natural map $\mathscr{Q}_{\alpha,r}\to {\cal N}^-_\alpha$ over a similarity class $\Ng=[N(\alpha,a,b)]\in {\cal N}^-_\alpha/\sim$ can be identified with  the quotient of 
$$\Z_{N,r}\edf \Z^2/(I_2+N)\Z^2+r\Z^2$$
 by the  group  $Z^+_{\GL(2,\Z)}(N)$ acting on $\Z_{N,r}$ by $K*[p]\edf [\varepsilon_K Kp]$. The set ${\cal N}_\alpha^-/\sim$ of similarity classes associated with $\alpha$ is again finite, so, for any $S^-$-admissible $\alpha$ and $r\in\N^*$ one obtains a finite number of biholomorphism classes of type III Inoue surfaces.
 \vspace{1mm}
 
 Note that  ``positive centraliser''  $Z^+_{\GL(2,\Z)}(N)$ associated with a matrix $N\in {\cal N}_\alpha$ ($N\in {\cal N}_\alpha^-$), where $\alpha$ is $S^+$- (respectively $S^-$-)admissible, is always infinite cyclic (see section \ref{Z+-section}, where it is shown that specifying a generator of this group is an interesting number theoretical problem).

\vspace{2mm}

We will begin our article with a new presentation of Inoue's constructions; for each one of the three types of Inoue surfaces we will precise  the entire data system needed for the construction of such a surface and we  will use the notation $S_\pg$ for the complex surface associated with the data system (parameter) $\pg$ varying in a set $\Pg$ which we will describe explicitly. Noting that surfaces of different types cannot be biholomorphic (see Corollary   \ref{LineSubBdls}), it follows that the  classification problem  for Inoue surfaces reduces to giving --  for every fixed type -- an explicit description  of the quotient set $\Pg/R$, where $R$ is the equivalence relation 
$$R\edf\{(\pg',\pg'')\in\Pg\times\Pg|\ S_{\pg'}\simeq S_{\pg''}\}.$$
The main difficulty is to understand in detail the equivalence relation $R$.

Let  $\Aff(\C^2)$, $\mathrm{T}(\C^2)$ be the groups of affine transformations, respectively translations, of $\C^2$, $U\edf H\times\C$ and let $\Aff(U)$, $\mathrm{T}(U)$ be the subgroups  
$$
\Aff(U)\edf\{f\in\Aff(\C^2)|\ f(U)=U\},\ \mathrm{T}(U)\edf\{f\in\mathrm{T}(\C^2)|\ f(U)=U\}
$$
consisting of those affine transformations (respectively translations) which fix $U$.

An elementary computation shows that
\begin{re}
We have 
$$
\mathrm{T}(U)=\left\{\begin{pmatrix}
w\\z	
\end{pmatrix}\stackrel{f}{\map}\begin{pmatrix}
w\\z 	
\end{pmatrix}+ \begin{pmatrix}
u\\ \zeta 	
\end{pmatrix}\vline \  u\in\R,\ \zeta\in\C\right\}.
$$	
$$
\Aff(U)=\left\{\begin{pmatrix}
w\\z	
\end{pmatrix}\stackrel{f}{\map}\begin{pmatrix}
\mu &0\\
\lambda&\nu	
\end{pmatrix}\begin{pmatrix}
w\\z 	
\end{pmatrix}+ \begin{pmatrix}
u\\ \zeta 	
\end{pmatrix}\vline \   \mu\in \R_{>0},\ \nu\in\C^*,\ u\in\R,\ (\lambda,\zeta)\in\C^2\right\}.
$$	
\end{re}

Any Inoue surface is the quotient $S=U/\Gamma$, where $\Gamma$ is a subgroup of the group $\Aff(U)$ acting properly discontinuously on $U$. Our classification results are based on the following theorem: 
\newtheorem*{th1}{Theorem \ref{tilde f is affine}}
\begin{th1}
	Let	$S'=U/\Gamma'$, $S''=U/\Gamma''$ be Inoue surfaces, $f:S'\to S''$ be a biholomorphism and  $\tilde f:U\to U$ be a lift of $f$. Then $\tilde f \in \Aff(U)$.
\end{th1}

This  theorem is a consequence of our uniqueness theorem: the tangent bundle of any Inoue surface  admits a  unique  holomorphic connection  (see Theorem \ref{HolConn} in this article), but can also be proved using Klingler's  or Dumitrescu's  uniqueness result \cite[Lemma 4.3]{Kl}, \cite[Proposition 3.1]{Du},  which concerns  holomorphic connections with vanishing curvature and torsion, respectively holomorphic connections with vanishing torsion.   

Theorem \ref{tilde f is affine} gives a purely algebraic interpretation of the set of isomorphism classes of Inoue surfaces: 

\newtheorem*{coro}{Corollary  \ref{ClassifCoro}}
\begin{coro}
Two Inoue surfaces 	 $S'=U/\Gamma'$, $S''=U/\Gamma''$ are biholomorphic if  and only if the subgroups $\Gamma'$, $\Gamma''$ of the group $\Aff(U)$ belong to the same conjugacy class. 
\end{coro}

Using this result we will obtain, for each type of Inoue surfaces,  a purely algebraic interpretation of the equivalence relation $R$ on the parameter set $\Pg$, which will allow us to describe the quotient set $\Pg/R$.

\section{The construction of Inoue surfaces}

\subsection{Type I Inoue surfaces}\label{Intro-typeI}

We begin with the following definition:
\begin{dt}
A complex number $\beta$ will be called admissible if $\beta\in  \C^*\setminus \R$, $|\beta|<1$,  and $|\beta|^{-2}$, $\beta$ and $\overline \beta$ are  the roots of a polynomial $P(X)\in \Z[X]$ of the form
 $$P(X)=X^3-\theta _2X^2+\theta_1X-1 \hbox{ with } \theta_1,\theta_2\in \Z.$$
Equivalently, $\beta\in  \C\setminus \R$ with $|\beta|<1$ is admissible if  
$$\begin{cases}
|\beta|^{-2}+2\Re(\beta)\in \Z \\
 2|\beta|^{-2} \Re(\beta)+\vert\beta\vert^2\in \Z.
\end{cases}$$

\end{dt}

\begin{re}\label{P(X)-is-irred}
The polynomial $P(X)=X^3-\theta _2X^2+\theta_1X-1\in\Z[X]$ associated with an admissible complex number $\beta$ is irreducible in $\Q[X]$.	
\end{re}
\begin{proof}
It suffices to prove that $\alpha\edf |\beta|^{-2}$ is not rational.  Suppose that $\alpha\in\Q$. $P(X)$ is a monic polynomial with integer coefficients, so its root $\alpha$ is integer over $\Z$. Since $\Z$ is integrally closed in $\Q$, it follows that $\alpha\in \Z$, so $P(X)$ decomposes as $P(X)=(X-\alpha)\big(X^2+ (\alpha-\theta_2)X+\alpha(\alpha-\theta_2)+\theta_1 \big)$. This implies $\alpha(\alpha(\alpha-\theta_2)+\theta_1) =1$, so $\alpha^{-1}=|\beta|^2\in\Z$, which contradicts the assumption $|\beta|<1$.
\end{proof}

\begin{dt}
Let $\beta$ be an admissible complex number. A pair $(a,b)\in \R^3\times \C^3$ will be called $\beta$-compatible if  $a$, $b$ and $\overline b$ are linearly independent over  $\C$  and  
	$$M(\beta, a,b)\edf {\begin{pmatrix} a_1 & b_1 & \overline b_1 \\ a_2& b_2 &\overline b_2 \\ a_3 & b_3 & \overline b_3 \end{pmatrix}}{\begin{pmatrix} |\beta|^{-2} & 0 & 0\\ 0& \beta &0 \\ 0 &0 & \overline \beta \end{pmatrix}}{\begin{pmatrix} a_1 & b_1 & \overline b_1 \\ a_2& b_2 &\overline b_2 \\ a_3 & b_3 & \overline b_3 \end{pmatrix}}^{-1}\in M_3(\Z).$$
We put:
 $${\cal P}_{\beta}=\big\{(a,b)\in \R^3\times \C^3| \ (a,b)\hbox{ is }\beta\hbox{-compatible}\big\}.$$
\end{dt}
 Note that

\begin{re}

Let $(a,b)$ be a  $\beta$-compatible pair. Then
\begin{enumerate}
\item $M(\beta, a,b)\in\SL(3,\Z)$.
\item $\Spec(M(\beta,a,b))=\{|\beta|^{-2},\beta, \bar \beta\}$.
\item The eigenspaces of  $M(\beta,a,b)$ in $\C^3$ are:
	 $$E_{|\beta|^{-2}}=\C a\,, \ E_{\beta}=\C b\,, \   E_{\bar \beta}=\C \bar b.$$

\end{enumerate}

\end{re}

For a pair $(a,b)\in \R^3\times \C^3$ we define the affine transformations $g_i(a,b)\in\Aff(U)$, $1\leq i\leq 3$,  by
$$
g_i(a,b):(w,z)\to\ (w+a_i,z+b_i).
$$

\begin{dt}\label{G(alpha,beta,a,b)I)}
	Let $\beta$ be an admissible complex number and $(a,b)\in{\cal P}_{\beta}$. We define   $g_0(\beta)\in \Aff(U)$  by
	$$ 
g_0(\beta):(w,z)\to\ (|\beta|^{-2} w, \beta  z)
$$	
and we define $G(\beta,a,b)$ to be the group generated by $g_0(\beta)$ and  $(g_i(a,b))_{1\leq i\leq 3}$.
\end{dt}

Note that $G(\beta,a,b)$ acts properly discontinuously on $U$ \cite{In}. The argument uses essentially the compatibility property of the pair $(a,b)$. 

\begin{dt}
	Let $\beta$ be an admissible complex number and $(a,b)\in{\cal P}_{\beta}$. We define the (first type)   Inoue surface $S^{\beta}_{a,b}$ by
	$$S^{\beta}_{a,b}=\qmod{U}{G(\beta,a,b)}$$
	where $G(\beta,a,b)$ is the subgroup of $\Aff(U)$ generated by the affine automorphisms  	$g_0(\beta)$, $(g_i(a,b))_{1\leq i\leq 3}$. 
\end{dt}

Therefore, for type I Inoue surfaces, the space of parameters $\Pg$ considered in the introduction is
$$
\Pg=\{(\beta,a,b)\in\C\times \R^3\times\C^3|\  \beta \hbox{ is admissible, } (a,b)\hbox{ is } \beta \hbox{-compatible}\}.	
$$
Note that 
\begin{re}\label{comm-rel-I}\cite[p. 274]{In}
The generators $g_0\edf g_0(\beta)$, $g_i\coloneq g_i(a,b)$ of $G(\beta,a,b)$ satisfy the commutation relations
$$
g_ig_j=g_jg_i \ \ \  i,j=1,2,3,
$$
$$g_0g_ig_0^{-1}=g_1^{m_{i1}}g_2^{m_{i2}}g_3^{m_{i3}} \ \ \ i=1,2,3,$$
where 	$m_{ij}$ are the entries of the matrix $M\coloneq M(\beta,a,b)$.
\end{re}
Using Remark \ref{comm-rel-I} we obtain the following simple
\begin{re} \label{ginG}
	 With the notations introduced above we have:
\begin{enumerate}
\item Any element $g\in G(\beta,a,b)$ can be written 	in a unique way as 
$$g=g_3(a,b)^{k_3}g_2(a,b)^{k_2}g_1(a,b)^{k_1}g_0(\beta)^{k_0}\hbox{ with }k_i\in\Z.$$  
\item We have 
$$G(\beta,a,b)\cap \T(U)= \langle g_1(a,b),g_2(a,b),g_3(a,b)\rangle\simeq\Z^3.$$

\end{enumerate}
\end{re}
\begin{proof}
	The first statement follows using the fact that the subgroup 
	$$\langle g_1(a,b),g_2(a,b),g_3(a,b)\rangle \subset G( \beta,a,b)$$
	 is abelian and normal in $G(\beta,a,b)$. The second statement follows from the first taking into account that $|\beta|<1$. 
\end{proof}

\subsection{Type II Inoue surfaces}\label{Intro-typeII}

In this section we will need the following subgroups of the group $\Aff(U)$ of affine transformations of $U\coloneq H\times\C$:
$$
\Aff_1(U)\coloneq\bigg\{\begin{pmatrix}
w\\
z 	
\end{pmatrix}\textmap{g} \begin{pmatrix}
\mu  &0\\
\lambda	&1
\end{pmatrix}\begin{pmatrix}
w\\
z 	
\end{pmatrix}+\begin{pmatrix}
u\\ \zeta	
\end{pmatrix}\vline\ \mu\in\R_{>0},\  \ u\in\R,\ \lambda\in\C, \ \zeta\in\C
\bigg\},
$$
$$
\Aff^1_1(U)\coloneq\bigg\{\begin{pmatrix}
w\\
z 	
\end{pmatrix}\textmap{g} \begin{pmatrix}
1  &0\\
\lambda	&1
\end{pmatrix}\begin{pmatrix}
w\\
z 	
\end{pmatrix}+\begin{pmatrix}
u\\ \zeta	
\end{pmatrix}\vline \  \ u\in\R,\ \lambda\in\C, \ \zeta\in\C
\bigg\},
$$
$$
\T^0(U)\coloneq  \bigg\{\begin{pmatrix}
w\\
z 	
\end{pmatrix}\map  \begin{pmatrix}
w\\
z 	
\end{pmatrix}+\begin{pmatrix}
0\\ \zeta	
\end{pmatrix}\vline  \ \zeta\in\C
\bigg\}. 
$$
Note that
\begin{re}
\begin{enumerate}

\item  $\T^0(U)$, $\T(U)$, $\Aff_1(U)$, $\Aff^1_1(U)$ are normal subgroups of $\Aff(U)$.
\item $\T^0(U)$ is central in $\Aff_1(U)$, so also in $\Aff^1_1(U)$.
\end{enumerate}
	
\end{re}

We shall see that the subgroups which intervene  in the construction of type II Inoue surfaces  are all contained in $\Aff_1(U)$.
 
\begin{dt}\label{def-S^+-admissible}
A real number $\alpha$ will be called $S^+$-admissible if  $\alpha=\frac{\theta+\sqrt{\theta^2-4}}{2}$, where $\theta\in\N_{\geq 3}$.	
\end{dt}

Therefore $\alpha$ is $S^+$-admissible if and only if $\alpha\in]1,+\infty[$ and $\alpha+\alpha^{-1}\in\N$, and if and only it is the larger root of a quadratic equation $t^2-\theta t+1$ with $\theta\in\N_{\geq 3}$.  
\begin{dt}
Let $\alpha\in\R$ be $S^+$-admissible. A linearly independent pair $(a,b)\in\R^2\times\R^2$   will be called  $\alpha$-compatible if 
\begin{equation}\label{Zcondition}
N(\alpha,a,b)\coloneq \begin{pmatrix}
a_1 &b_1\\
a_2& b_2	
\end{pmatrix}\begin{pmatrix}
\alpha &0\\
0&\alpha^{-1}	
\end{pmatrix}\begin{pmatrix}
a_1 &b_1\\
a_2& b_2	
\end{pmatrix}^{-1}\in M_2(\Z).
\end{equation}
\end{dt}

\begin{re}

Let $\alpha$  be as above and $(a,b)$ be  an $\alpha$-compatible pair. Then
\begin{enumerate}
\item $N(\alpha, a,b)\in\SL(2,\Z)$.
\item $\Spec(N(\alpha,a,b))=\{\alpha,\alpha^{-1}\}$.
\item The eigenspaces of   $N(\alpha,a,b)$ in $\R^2$ are  $E_{\alpha}=\R a$, $ E_{\alpha^{-1}}=\R  b$.	
\end{enumerate}	
\end{re}
We put:
$${\cal P}_{\alpha}=\big\{(a,b)\in \R^2\times \R^2| \ (a,b)\hbox{ is }\alpha\hbox{-compatible}\big\},$$
$${\cal N}_\alpha\coloneq\{N\in\SL(2,\Z)|\ \alpha\in\Spec(N)\},$$
and we denote by $\eta_\alpha:{\cal P}_{\alpha}\to {\cal N}_\alpha$ the map given by $\eta_\alpha(a,b)\edf N(\alpha,a,b)$.

\begin{dt}\label{Def-alpha-r-compat} Let $(a,b)$ be an $\alpha$-compatible pair and $r\in\N^*$.
\begin{enumerate}
	\item  A vector $c\in\R^2$ will be called $(a,b,r)$-compatible if, putting $N\coloneq N(\alpha,a,b)$, we have 
\begin{equation}\label{CompCond}
(I_2-N)\bigg(c-\frac{1}{2}\bpm a_1b_1\\ a_2b_2\epm\bigg)-\frac{b\wedge a}{2}\bpm n_{11}n_{12}\\ n_{21}n_{22}\epm\in \frac{b\wedge a}{r}\Z^2.  
\end{equation}

If this is the case, we define  $p(a,b,c,r)\in\Z^2$ by 
\begin{equation}\label{c->p-formula}
(I_2-N)\bigg(c-\frac{1}{2}\bpm a_1b_1\\ a_2b_2\epm\bigg)-\frac{b\wedge a}{2}\bpm n_{11}n_{12}\\ n_{21}n_{22}\epm= \frac{b\wedge a}{r}p(a,b,c,r).	
\end{equation}
	
The set of 	$(a,b,r)$-compatible vectors will be denoted by ${\cal C}_{a,b,r}$.

\item If $(a,b)$ is $\alpha$-compatible and $c$ is $(a,b,r)$-compatible we will also say that the triple $(a,b,c)$ is $(\alpha,r)$-compatible. Therefore, $(a,b,c)$  is $(\alpha,r)$-compatible if and only if 
\begin{enumerate}
\item $N(\alpha,a,b)\coloneq \begin{pmatrix}
a_1 &b_1\\
a_2& b_2	
\end{pmatrix}\begin{pmatrix}
\alpha &0\\
0&\alpha^{-1}	
\end{pmatrix}\begin{pmatrix}
a_1 &b_1\\
a_2& b_2	
\end{pmatrix}^{-1}\in M_2(\Z)$,	
\item Putting $N=(n_{ij})_{i,j}=N(\alpha,a,b)$, we have
$$
(I_2-N)\bigg(c-\frac{1}{2}\bpm a_1b_1\\ a_2b_2\epm\bigg)-\frac{b\wedge a}{2}\bpm n_{11}n_{12}\\ n_{21}n_{22}\epm\in \frac{b\wedge a}{r}\Z^2.$$	
\end{enumerate}
\end{enumerate}
\end{dt}
\begin{re}
Our compatibility condition (\ref{CompCond})  is equivalent to Inoue's compatibility condition \cite[formula (17) p. 276]{In}, which reads
$$(I_2-N)c-e\in \frac{b\wedge a}{r}\Z^2,$$
where
$$
\begin{pmatrix}
e_1\\e_2	
\end{pmatrix}=
\frac{1}{2}\begin{pmatrix}
n_{11}(n_{11}-1) a_1b_1+ n_{12}(n_{12}-1)a_2b_2\\
n_{21}(n_{21}-1) a_1b_1+n_{22}(n_{22}-1) a_2 b_2	
\end{pmatrix}+ 
\begin{pmatrix}
 n_{11}n_{12}\\
 n_{21}n_{22}	
 \end{pmatrix}b_1 a_2.
$$
Using the identities $Na=\alpha a$, $Nb=\alpha^{-1}b$, one obtains easily:
$$
\begin{pmatrix}
e_1\\e_2	
\end{pmatrix}=\frac{1}{2}\bigg[\bpm a_1b_1\\
 a_2 b_2
\epm - N\bpm a_1b_1\\ a_2b_2\epm+ (b\wedge a)\begin{pmatrix}
 n_{11}n_{12}\\
 n_{21}n_{22}	
 \end{pmatrix}\bigg],
 $$
 which proves the claim.
 	
\end{re}

The Inoue surfaces of type II are quotients of $U$ by a  group of affine  transformations constructed using an $(\alpha,r)$-compatible triple and a complex parameter $t\in\C$:
\begin{re}\label{c->p-remark}
Let $(a,b)\in {\cal P}_\alpha$ and $r\in \N^*$. 
\begin{enumerate}
\item The set ${\cal C}_{a,b,r}$ is a $\frac{b\wedge a}{r}(I_2-N)^{-1}(\Z^2)$-torsor.
\item 	The map $\pi_{a,b,r}:{\cal C}_{a,b,r}\to \Z^2$, $\pi_{a,b,r}(c)\edf 	p(a,b,c,r)\in\Z^2$
is bijective and satisfies the identity 
$$
\pi_{a,b,r}\bigg(c+\frac{b\wedge a}{r}(I_2-N)^{-1}\bpm s_1\\s_2\epm \bigg)=\pi_{a,b,r}(c)+\bpm s_1\\s_2\epm.
$$


%
\end{enumerate}
\end{re}

Note that the two components of  	$p(a,b,c,r)$ correspond to the integers $p$, $q$ in Inoue's notation.\\

Let $(a,b,c)\in \R^2\times\R^2\times\R^2$ and $r\in\N^*$. We define the affine transformations $g_i=g_i(a,b,c)$  ($1\leq i\leq 2$), $g_3=g_3(a,b,r)\in\Aff(U)$,    by 

\begin{equation}\label{Def-gi}
g_i\begin{pmatrix}
w\\z	
\end{pmatrix}=\begin{pmatrix}
w+a_i\\z+b_iw+c_i	
\end{pmatrix}=\begin{pmatrix}
1 &0\\
b_i& 1	
\end{pmatrix}\begin{pmatrix}
w\\z	
\end{pmatrix}+\begin{pmatrix}
a_i\\c_i	
\end{pmatrix}
,\ g_3\begin{pmatrix}
w\\z	
\end{pmatrix}=\begin{pmatrix}
w\\z+\frac{b\wedge a}{r}	
\end{pmatrix}.	
\end{equation}

\begin{dt}\label{G(rabct)} 
Let $\alpha$, $r$ be as above and $(a,b,c)$ be  an $(\alpha,r)$-compatible triple. We define
$G(a,b,c,r)$ to be the subgroup of $\Aff(U)$ generated by the affine transformations $g_1=g_1(a,b,c)$, $g_2=g_2(a,b,c)$, $g_3=g_3(a,b,r)$.

For $t\in\C$ we also define $g_0=g_0(\alpha,t)\in \Aff(U)$ by
$$
g_0(\alpha,t)\begin{pmatrix}
w\\z	
\end{pmatrix}=\begin{pmatrix}
\alpha w\\z+t	
\end{pmatrix}
$$
and the subgroup $G(\alpha,a,b,c,r,t)\subset\Aff(U)$ by 
$$G(\alpha,a,b,c,r,t)\coloneq \langle g_0,g_1,g_2,g_3\rangle.$$
\end{dt}
Recall \cite[p. 276]{In} that $G(\alpha,a,b,c,r,t)$ acts properly discontinuously  on $U$. The argument uses essentially the $(\alpha,r)$-compatibility condition of the triple $(a,b,c)$. We define

\begin{dt}\label{StypeII}
Let $\alpha$, $r$ be as above and $(a,b,c)$ be  an $(\alpha,r)$-compatible triple. The type II Inoue surface associated with the parameters 	$(\alpha,a,b,c,r,t)$ is
$$
S^{\alpha,r}_{a,b,c,t}\coloneq \qmod{U}{G(\alpha,a,b,c,r,t)}.
$$
\end{dt}

Therefore, for type II Inoue surfaces, the space of parameters $\Pg$ considered in the introduction is
\begin{align*}
\Pg=\{(\alpha,a,b,c,r,t)\in\ ]1,+\infty[\times \R^2\times\R^2\times\R^2\times\N^*&\times\C|\ \alpha \hbox{ is $S^+$-admissible, and}\\ 
&(a,b,c)\hbox{ is } (\alpha,r)-\hbox{compatible}\}.	
\end{align*}

\begin{re} \label{com-rel-II} \cite[p. 276]{In} Let $(a,b,c)$ be  an $(\alpha,r)$-compatible triple.
Putting $p\coloneq p(a,b,c,r)\in\Z^2$, we have the commutation relations 
$$g_1^{-1}g_2^{-1}g_1g_2=g_3^r,\ g_ig_3=g_3g_i \hbox{ for } 0\leq i\leq 2,\ g_0g_ig_0^{-1}=g_1^{n_{i1}}g_2^{n_{i2}}g_3^{p_i},$$
where $N=(n_{ij})_{i,j}=N(\alpha,a,b)$.

\end{re}
\begin{re}\label{RelationsSubgroups}
The subgroup $Z(a,b,r)\coloneq\langle g_3	\rangle\subset G(\alpha,a,b,c,r,t)$ is central in $G(\alpha,a,b,c,r,t)$ and $G(a,b,c,r)$ is a normal subgroup of $G(\alpha,a,b,c,r,t)$.

\begin{enumerate}
\item We have 
\begin{equation}\label{inters1}
 Z(a,b,r)=G(a,b,c,r)\cap \T(U)=G(\alpha,a,b,c,r,t)\cap \T(U)	
 \end{equation}
 The proof uses the fact that $b_1$, $b_2$ are linearly independent over $\Q$ (which follows easily taking into account that $b$ is an eigenvector of a rational matrix associated with an irrational eigenvalue).
 
 \item We have
\begin{equation}\label{inters2}
 G(a,b,c,r)=	G(\alpha,a,b,c,r,t)\cap \Aff^1_1(U). 
\end{equation}

\item We have inclusions $G(\alpha,a,b,c,r,t)\subset \Aff_1(U)$, $G(a,b,c,r)\subset \Aff^1_1(U)$, and the induced morphism 
\begin{equation}\label{mor-quot}
\qmod{G(\alpha,a,b,c,r,t)}{G(a,b,c,r)}\to \qmod{\Aff_1(U)}{\Aff^1_1(U)}
\end{equation}
between the corresponding quotient groups is  a monomorphism. Identifying the quotient ${\Aff_1(U)}/{\Aff^1_1(U)}$ with $\R_{>0}$ via the obvious isomorphism 
$${\Aff_1(U)}/{\Aff^1_1(U)}\textmap{\simeq} \R_{>0}$$ induced by  
$
\left(\begin{pmatrix}
w\\
z 	
\end{pmatrix}\textmap{g} \begin{pmatrix}
\mu  &0\\
\lambda	&1
\end{pmatrix}\begin{pmatrix}
w\\
z 	
\end{pmatrix}+\begin{pmatrix}
u\\ \zeta	
\end{pmatrix}\right)\mapsto \mu,$ 
we obtain a group isomorphism  
$$
{G(r,\alpha,a,b,c,t)}/{G(a,b,c,r)}=\langle [g_0]\rangle\textmap{\simeq} \langle \alpha \rangle. 
$$
 onto the cyclic group $\langle \alpha\rangle\subset\R_{>0}$ which maps the  class $[g_0]_{G(a,b,c,r)}$ to $\alpha$.

\end{enumerate}

\end{re}
Using Remark \ref{com-rel-II}, we obtain:
\begin{re}\label{[GG(a,b,c,r),GG(a,b,c,r)]}
One has $[G(a,b,c,r),G(a,b,c,r)]=\langle g_3^r\rangle\subset \langle g_3\rangle$	, in particular one has an isomorphism
$$
{\langle g_3\rangle}/{[G(a,b,c,r),G(a,b,c,r)]}\simeq\Z_r.
$$
\end{re}
\begin{re}\label{gen-form-g-in-G(a,b,c,r-t)}
Any element $g\in G(a,b,c,r)$ can be written in a unique way in the form
\begin{equation}\label{g-in-G(a,b,c,r)}
g=g_1^{n_1}g_2^{n_2}g_3^k
\end{equation} 
with $n_1$, $n_2$, $k\in\Z$; any element $g\in G(\alpha,a,b,c,r,t)$ can be written in a unique way in the form
\begin{equation}\label{g-in-G(a,b,c,r,t)}
g=g_0^lg_1^{n_1}g_2^{n_2}g_3^k 
\end{equation}  
with $n_1$, $n_2$, $k$, $l\in\Z$.

\end{re}

\subsection{Type III Inoue surfaces}\label{Intro-typeIII}

We start with the analogue of Definition \ref{def-S^+-admissible} for type III surfaces:

\begin{dt}\label{def-S^--admissible}
A real number $\alpha\in ]1,+\infty[$ will be called $S^-$-admissible if  $\alpha=\frac{\theta+\sqrt{\theta^2+4}}{2}$, where $\theta\in\N^*$.	
\end{dt}

In other words $\alpha$ is  $S^-$-admissible if it coincides with the positive root of a quadratic equation of the form $t^2-\theta t- 1=0$ with $\theta\in\N^*$.
Let $\alpha$ be $S^-$-admissible.

\begin{dt}\label{DefCompat-}
A linearly independent pair $(a,b)\in\R^2\times\R^2$   will be called  $\alpha$-compatible if, putting 
$P_{ab}=\begin{pmatrix}
a_1 &b_1\\
a_2& b_2	
\end{pmatrix}$, we have 
\begin{equation}\label{Zcondition-}
N(\alpha,a,b)\coloneq P_{ab}\begin{pmatrix}
\alpha &0\\
0&-\alpha^{-1}	
\end{pmatrix}P_{ab}^{-1}\in M_2(\Z).
\end{equation}
\end{dt}
The space of $\alpha$-compatible pairs in the sense of Definition \ref{DefCompat-} will be denoted by ${\cal P}_\alpha^-$.  For $(a,b)\in {\cal P}_\alpha^-$ we have
$$N(\alpha,a,b)\in \GL_-(2,\Z)\edf \{N\in\GL(2,\Z)|\ \det(N)=-1\},\ \tr(N(\alpha,a,b))=\theta,$$
and the eigenspaces of $N(\alpha,a,b)$ are $E_\alpha=\R a$, $E_{-\alpha^{-1}}=\R b$. For an $S^-$-admissible $\alpha\in ]1,+\infty[$ we will use the notation:
$$
{\cal N}^-_\alpha\edf \{N\in \GL_-(2,\Z)|\ \alpha\in \Spec(N)\}.
$$

\begin{dt} Let $(a,b)$ be an $\alpha$-compatible pair and $r\in\N^*$.
An element $c\in\R^2$ is called $(a,b,r)$-compatible if, putting $N\coloneq N(\alpha,a,b)$, we have 
\begin{equation}\label{CompCond-}
(I_2+N)\bigg(c-\frac{1}{2}\bpm a_1b_1\\ a_2b_2\epm\bigg)-\frac{a\wedge b}{2}\bpm n_{11}n_{12}\\ n_{21}n_{22}\epm\in \frac{a\wedge b}{r}\Z^2.
\end{equation}
If this is the case, we define  $p(a,b,c,r)\in\Z^2$ by
$$
(I_2+N)\bigg(c-\frac{1}{2}\bpm a_1b_1\\ a_2b_2\epm\bigg)-\frac{a\wedge b}{2}\bpm n_{11}n_{12}\\ n_{21}n_{22}\epm=\frac{a\wedge b}{r}\bpm p_1\\p_2\epm, 
$$
 and we'll also say that $(a,b,c)$ is an $(\alpha,r)$-compatible triple. We will denote by ${\cal C}_{a,b,r}$ the space of $(a,b,r)$-compatible vectors and by ${\cal T}_{\alpha,r}$ the space of $(\alpha,r)$-compatible triples.
\end{dt}

Note that, as mentioned for type II surfaces, our compatitibility condition is equivalent to Inoue's  formula \cite[(20), p. 279]{In}.

\begin{dt}\label{G(rabc-)} 
Let $\alpha$, $r$ be as above and $(a,b,c)$ be  an $(\alpha,r)$-compatible triple. We define
$G(a,b,c,r)$ to be the subgroup of $\Aff(U)$ generated by the affine transformations $g_1=g_1(a,b,c)$, $g_2=g_2(a,b,c)$, $g_3=g_3(a,b,r)$ defined in (\ref{Def-gi}). 
We also define $g_0=g_0(\alpha)\in \Aff(U)$ by
$$
g_0(\alpha)\begin{pmatrix}
w\\z	
\end{pmatrix}=\begin{pmatrix}
\alpha w\\-z	
\end{pmatrix}
$$
and the subgroup $G(\alpha,a,b,c,r)\subset\Aff(U)$ by 
$$G(\alpha,a,b,c,r)\coloneq \langle g_0,g_1,g_2,g_3\rangle.$$
\end{dt}

Recall \cite[p. 279]{In} that $G(\alpha,a,b,c,r)$ acts properly discontinuously  on $U$.   We define

\begin{dt}\label{StypeIII}
Let $\alpha$, $r$ be as above and $(a,b,c)$ be  an $(\alpha,r)$-compatible triple. The type III Inoue surface associated with the parameters 	$(\alpha,a,b,c,r)$ is
$$
S^{\alpha,r}_{a,b,c}\coloneq \qmod{U}{G(\alpha,a,b,c,r)}.
$$
\end{dt}

\begin{re} \label{com-rel-III} \cite[p. 279]{In} Let $(a,b,c)$ be  an $(\alpha,r)$-compatible triple.
Putting $p\coloneq p(a,b,c,r)\in\Z^2$, we have the commutation relations 
$$g_1^{-1}g_2^{-1}g_1g_2=g_3^r,\ g_ig_3=g_3g_i \hbox{ for } 1\leq i\leq 2,\ g_0g_3g_0^{-1}=g_3^{-1},\ g_0g_ig_0^{-1}=g_1^{n_{i1}}g_2^{n_{i2}}g_3^{p_i}$$
where $N=(n_{ij})_{i,j}=N(\alpha,a,b)$.

\end{re}

The statements of Remark \ref{RelationsSubgroups}  extend to type III surfaces with one notable exception: the cyclic subgroup $\langle g_3\rangle $  is still central in the group $G(a,b,c,r)$, but not in $G(\alpha,a,b,c,r)$, because $g_3$ does not commute with $g_0$. The obvious analogues of Remarks \ref{[GG(a,b,c,r),GG(a,b,c,r)]}, \ref{gen-form-g-in-G(a,b,c,r-t)} hold for type III surfaces.

\section{Fundamental properties of Inoue surfaces}

\subsection{Line bundles on Inoue surfaces}

Let $X$ be a class VII surface, $x_0\in X$ and $\pi_1(X,x_0)$ be the fundamental group of the pair $(X,x_0)$. Since $b_1(X)=1$, the torsion free quotient $H_1(X,\Z)/\Tors(H_1(X,\Z))$ is always isomorphic to $\Z$.  If $b_2(X)=0$, the natural map 
$$\Hom(H_1(X,\Z),\C^*)=\Hom(\pi_1(X,x_0),\C^*)\to \Pic(X)$$
is an isomorphism, and the image of the  natural monomorphism
$$
 \Hom\left(\qmod{H_1(X,\Z)}{\Tors(H_1(X,\Z))},\C^*\right)\to \Hom(H_1(X,\Z),\C^*)=\Pic(X)
$$
is precisely the identity component $\Pic^0(X)$ of $\Pic(X)$ (see for instance \cite{Te-survey}). Fixing an isomorphism   $H_1(X,\Z)/\Tors(H_1(X,\Z))\textmap{\simeq}\Z$ gives an isomorphism 
$$\Hom(\Z,\C^*)=\C^*\textmap{\simeq}  \Hom\left(\qmod{H_1(X,\Z)}{\Tors(H_1(X,\Z))},\C^*\right)= \Pic^0(X).$$ 
\vspace{1mm}

Any Inoue surface $S$ is the quotient $U/\Gamma$ where $\Gamma$ is a group of affine transformations acting properly discontinuously on  $U\edf H\times\C\subset\C^2$. We choose $u_0\in U$, we put $x_0=[u_0]\in S$ and   we  identify $\pi_1(S,x_0)$ with $\Gamma$ in the obvious way. The canonical isomorphism $\Hom(\pi_1(S,x_0),\C^*)=\Hom(\Gamma,\C^*)\to \Pic(S)$ is given explicitly by
$$
\Hom(\Gamma,\C^*)\ni \rho\mapsto [L_\rho]\in\Pic(S),
$$
where  
$$
L_\rho\edf U\times\C/_\rho\,\Gamma.
$$
Here $U\times\C/_\rho\,\Gamma$ stands for the quotient of the trivial line bundle $U\times\C$ over $U$ by the group $\Gamma$ acting by
$$
(g,((w,z),\zeta))\mapsto (g(w,z), \rho(g)\zeta).
$$

For an Inoue surface $S$ we also have a natural  choice of  an isomorphism 
$$\qmod{H_1(S,\Z)}{\Tors(H_1(S,\Z))}\to\Z$$
 obtained using Inoue's description of the group $H_1(S,\Z)$ (see \cite[p. 274, 276, 279]{In}):

Identifying $H_1(S,\Z)$ with the abelianization  ${\Gamma}/{[\Gamma,\Gamma]}$ of $\Gamma$, we see that, for all three types of Inoue  surfaces, the classes $[g_i]$, $1\leq i\leq 3$, generate  $\Tors(H_1(S,\Z))$, whereas the map 
$$
H_1(S,\Z)=\qmod{\Gamma}{[\Gamma,\Gamma]}\ni [g_0^{k_0}][g_1^{k_1}][g_2^{k_2}][g_3^{k_3}] \longmapsto k_0\in\Z 
$$
is an epimorphism with kernel $\Tors(H_1(S,\Z))$, so it induces   induces an isomorphism 
$$\kappa_0:\qmod{H_1(S,\Z)}{\Tors(H_1(S,\Z))}\to\Z.$$
For an Inoue surface $S$ we will always use the isomorphism $\kappa_0$ defined above to identify the torsion free quotient $H_1(S,\Z)/\Tors(H_1(S,\Z))$ with $\Z$. The corresponding isomorphism 
$$\C^*=\Hom(\Z,\C^*)\textmap{\simeq}  \Hom\left(\qmod{H_1(S,\Z)}{\Tors(H_1(S,\Z))},\C^*\right)=\Pic^0(S) $$
is given explicitly by 
$$
\zeta\mapsto [L_{\rho_\zeta}],
$$
where $\rho_\zeta:\Gamma\to \C^*$ is defined by 
$$\rho_\zeta(g_i)=1\hbox{ for }1\leq i\leq 3,\ \rho_\zeta(g_0)=\zeta.$$
We will put $L_\zeta\edf L_{\rho_\zeta}$ to save on notations. \\

One of the main results of \cite{In} is the following fundamental:

\begin{thry} \label{ThInoue} (Inoue) 
Let  $S$ be an Inoue surface. Then $S$ has no complex curve.
\end{thry}

In other words an Inoue surface 	has no positive effective divisor. This implies:

\begin{co}\label{CoInoue}
Let $S$ be an Inoue surface and ${\cal L}$ be a non-trivial invertible sheaf on $S$. Then $H^0(S,{\cal L})=0$.	
\end{co}

\subsection{The tangent bundle of an Inoue surface} \label{tangent-bundle-sect}

 Denoting by $(e_w,e_z)$ the canonical basis of $\C^2$, note  that in all cases the line subbundle $U\times\C e_z\subset T_U=U\times\C^2$  is invariant under the tangent map $g_*$ for any $g\in\Gamma$. Therefore we obtain a line subbundle $M$ of $T_S$ defined as the quotient of $U\times\C e_z$ by the induced $\Gamma$-action on the tangent bundle $T_U$. Putting $L\edf T_S/M$, we obtain a short exact sequence
\begin{equation}\label{CES}
0\to M\hookrightarrow T_S\to L\to 0,  
\end{equation}
which will be called {\it the canonical   exact sequence} for the tangent bundle of $S$. 
\vspace{2mm}

Note  that: 
\begin{re}\label{CESI} For Inoue surfaces of type I, the line subbundle $U\times \C e_w\subset T_U$ is also $\Gamma$-invariant and its $\Gamma$-quotient is a complement of $M$ in $T_S$, so it can be identified with $L$. Therefore, in this case, the canonical   exact sequence (\ref{CES}) splits giving a direct sum decomposition  $T_S=M\oplus L$.	
\end{re}

Identifying the line subbundles $U\times\C e_z\subset T_U$, $U\times\C e_w\subset T_U$ with the trivial line bundle $U\times\C$, and taking into account the explicit formulae for the generators $g_i$ of $\Gamma$ in each case, we see that $M$, $L$ can be identified respectively with the quotients 
$$
(U\times\C) /_{\rho_M} \Gamma\,,\,\ (U\times\C) /_{\rho_L} \Gamma
$$
where the group morphisms $\rho_M$, $\rho_L\in\Hom(\Gamma,\C^*)$ are given by the table below:
 \begin{equation}\label{rho}
\begin{array}{|c|c|c|c|c|c|c|c|c|c|c|}
\hline 	
\rm type &\rho&\rho(g_0)&\rho(g_1)&\rho(g_2)&\rho(g_3) \\ \hline 
\multirow{2}{*}{I} &\rho_M&\beta&1&1&1  \\ \cline{2-6}
  &\rho_L&\alpha&1&1&1  \\\hline
 \multirow{2}{*}{II} &\rho_M&1&1&1&1  \\ \cline{2-6}
  &\rho_L&\alpha&1&1&1  \\\hline
  \multirow{2}{*}{III} &\rho_M&-1&1&1&1  \\ \cline{2-6}
  &\rho_L&\alpha&1&1&1  \\\hline 
\end{array}	\ .
 \end{equation}

Using the notation introduced above, this  shows that 
\begin{re}\label{extclass}
For any Inoue surface $S$ we have a natural identification $L=L_\alpha$. For an Inoue surface of type $I$ (respectively II, III), we have 	a natural identification $M=L_\beta$ (respectively $M=L_1=S\times\C$, $M=L_{-1}$).
\end{re}

Note also that (\ref{CES}) gives natural identifications 
\begin{equation}\label{K*}
K_S^*=\det(T_S)=M\otimes L	=\left\{\begin{array}{cc}
L_{\alpha\beta}  &\hbox{ if $S$ is of type I}\\	
L_{\alpha}  &\hbox{ if $S$ is of type II}\\
L_{-\alpha}  &\hbox{ if $S$ is of type III}\\
\end{array}\right..
\end{equation}
\begin{re}\label{CESII-III}
For an Inoue surface of type II, the canonical exact sequence (\ref{CES}) becomes
\begin{equation}\label{CESII}
0\to {\cal O}_S\to {\cal T}_S\to {\cal K}_S^*\to 0,	
\end{equation}
and for an Inoue surface of type III, the canonical exact sequence (\ref{CES}) becomes
\begin{equation}\label{CESIII}
0\to {\cal L}_{-1}\to {\cal T}_S\to {\cal K}_S^*\otimes {\cal L}_{-1}\to 0.	
\end{equation}
In both cases the extension class $h\in \Ext^1({\cal L},{\cal M})= H^1(S,{\cal K}_S)$ is non-trivial. 
\end{re}
\begin{proof}
Suppose first that $S$ is of type II. In this case, taking duals in the short exact 	sequence (\ref{CESII}) we obtain
\begin{equation}\label{CESIID}
0\to {\cal K}_S\to \Omega_S\to {\cal O}_S\to 0,	
\end{equation}
and the extension classes of (\ref{CESII}), (\ref{CESIID}) coincide. But (\ref{CESIID}) does not split because, $S$ being a class VII surface, we have $h^{10}_S\edf \dim(H^0(S,\Omega_S))=0$.\\

Suppose now that $S$ is of type III.  As noted in \cite[p. 279]{In}, $S$ admits a double cover $\tilde S$ which is an Inoue surface of type II. Moreover, the canonical short exact sequence of $\tilde S$ is the pull-back of the canonical short exact sequence of $S$. Since the latter is not split, the canonical short exact sequence of $S$ cannot be split.
\end{proof}
\begin{co} \label{LineSubBdls}
1. A Inoue surface 
\begin{enumerate}
\item is of type I if and only of its tangent bundle has two (different) line subbundles.  
\item is of type II if and only it has 	its tangent bundle has a unique line subbundle and this line subbundle is trivial.
\item is of type III if and only it has 	its tangent bundle has a unique line subbundle and this line subbundle is non-trivial.
\end{enumerate}
2. If two Inoue surfaces are biholomorphic, then they have the same type.
\end{co}
\begin{proof}
The first statements follow by Remarks \ref{CESI},  \ref{CESII-III}. The second statement follows from the first.
\end{proof}
\subsection{Vanishing theorems}

For a vector bundle $E$ (a locally free sheaf ${\cal E}$) we will denote by $\End_0(E)$ (respectively ${\cal E}nd_0({\cal E}$)) the bundle (sheaf) of trace-free endomorphisms of $E$ (respectively ${\cal E}$).
The goal of this section is the following vanishing theorem:
\begin{thry}\label{VanTh}
Let $S$ be an Inoue surface. Then
\begin{enumerate}
\item $
H^0(S, {\cal K}_S\otimes{\cal E}nd({\cal T}_S))=0$.
\item $H^0(S, \Omega_S\otimes {\cal E}nd({\cal T}_S))=0$.
\end{enumerate}
	
\end{thry}

\begin{proof}
If $S$ is of type I,  we have
 $$
 {\cal T}_S\simeq {\cal M}\oplus {\cal L},\ {\cal K}_S\otimes{\cal E}nd({\cal T}_S)\simeq  {\cal M}^{*\otimes 2} \oplus {\cal L}^{*\otimes 2} \oplus {\cal K}_S^{\oplus2}, $$
 $$ \Omega_S\otimes {\cal E}nd({\cal T}_S)=({\cal L}^{*\otimes 2}\otimes{\cal M})\oplus ({\cal M}^{*\otimes 2}\otimes{\cal L)\oplus {\cal M}^{*\oplus 3} \oplus {\cal L}^{*\oplus 3}}. $$
We apply Corollary \ref{CoInoue} to the invertible sheaves ${\cal M}^*$, ${\cal L}^*$ 	${\cal M}^{*\otimes 2}$, ${\cal L}^{*\otimes 2}$, ${\cal L}^{*\otimes 2}\otimes{\cal M}$, ${\cal M}^{*\otimes 2}\otimes{\cal L}$ and ${\cal K}_S$. These sheaves are  non-trivial: this follows using (\ref{rho}) taking into account that $1\not\in \{\alpha^{-1},\beta^{-1}, \alpha^{-2},  \beta^{-2},  \alpha^{-2}\beta,  \beta^{-2}\alpha, \alpha^{-1}\beta^{-1}\}$.\\

Suppose that $S$ is of type II.
\vspace{2mm}\\
(1) Since 
$${\cal K}_S\otimes {\cal E}nd({\cal T}_S)={\cal K}_S\otimes({\cal O}_S\id_{{\cal T}_S}\oplus  {\cal E}nd_0({\cal T}_S))\simeq {\cal K}_S\oplus ({\cal K}_S\otimes  {\cal E}nd_0({\cal T}_S)),$$
and ${\cal K}_S$ is non-trivial, it suffices to prove that $H^0(S,{\cal K}_S\otimes  {\cal E}nd_0({\cal T}_S))=0$. We will make use of Lemma \ref{CommDiag} proved below applied to the exact sequence (\ref{CESII}). Tensorizing by ${\cal K}_S$ the last two lines of the first diagram, we obtain the commutative diagram
\begin{equation}\label{nice-diag}
\begin{tikzcd}[column sep=2em, row sep= 2em]
0\ar[r]&{\cal K}_S\otimes{\cal F}_0\ar[r, hook, "i_{{\cal F}_0}"]\ar[d, "{(u,v)_0}"]&{\cal K}_S\otimes{\cal E}nd_0({\cal T}_S)\ar[r, "(qj)_0"]\ar[d, "\circ j"]&	 {\cal O}_S\ar[d, "\id"] \ar[r]&0\\
0\ar[r]&{\cal K}_S \ar[r, "\circ j" ]&  {\cal K}_S\otimes {\cal T}_S\ar [r, "q\circ "]  & {\cal O}_S\ar[r] & 0 \end{tikzcd}
\end{equation}  
with exact rows. We obtain the following commutative diagram with exact rows
\begin{equation}\label{nice-diag-coh}
\begin{tikzcd}[column sep=1em, row sep= 2em]
0\ar[r]&H^0(S,{\cal K}_S\otimes{\cal F}_0)\ar[r, hook,  ]\ar[d,  ]&H^0(S,{\cal K}_S\otimes{\cal E}nd_0({\cal T}_S))\ar[r,  ]\ar[d,  ]&	 \C\ar[d, "\id"] \ar[r, "\partial'"]&H^1(S,{\cal K}_S\otimes{\cal F}_0)\ar[d]\\
0\ar[r]&H^0(S,{\cal K}_S) \ar[r]&  H^0(S,{\cal K}_S\otimes {\cal T}_S)\ar [r]  & \C\ar[r,  "\partial''"] & H^1(S,{\cal K}_S) \end{tikzcd},
\end{equation} 
where $\partial'$, $\partial''$ are the connecting morphisms associated with the two exact sequences in (\ref{nice-diag}). Since the lower exact sequence in (\ref{nice-diag}) is non-split (it is obtained by tensorizing    (\ref{CESII}) with ${\cal K}_S$), it follows that $\partial''$ is injective, so $\partial '$ is also injective.   Therefore the monomorphism 
$$
H^0(S,{\cal K}_S\otimes{\cal F}_0)\to H^0(S,{\cal K}_S\otimes{\cal E}nd_0({\cal T}_S))
$$ 
is an isomorphism, so it suffices to prove that
\begin{equation}\label{H0KF0}
H^0(S,{\cal K}_S\otimes{\cal F}_0)=0.	
\end{equation}

Tensorizing by ${\cal K}_S$ the left hand column in the first diagram of Lemma \ref{CommDiag}, we obtain the short exact sequence
$$
0\to {\cal K}_S^{\otimes 2}\to {\cal K}_S\otimes {\cal F}_0\to {\cal K}_S\to 0.
$$
Since $H^0(S,{\cal K}_S^{\otimes 2})=H^0(S,{\cal K}_S)=0$, the associated cohomology exact sequence gives $H^0(S,{\cal K}_S\otimes{\cal F}_0)=0$, as claimed. 
\\ \\
(2) Tensorizing the exact sequence (\ref{CESIID}) by ${\cal E}nd({\cal T}_S)$, we obtain the exact sequence
\begin{equation}\label{SESOmegaEnd}
0\to {\cal K}_S\otimes{\cal E}nd({\cal T}_S)\to \Omega_S\otimes {\cal E}nd({\cal T}_S)\to {\cal E}nd({\cal T}_S)\to 0, 	
\end{equation}
which gives the long exact cohomology sequence
\begin{equation}
\begin{split} 
0\to H^0(S, {\cal K}_S\otimes{\cal E}nd({\cal T}_S))&\to    H^0(S,\Omega_S\otimes {\cal E}nd({\cal T}_S))\to H^0(S,{\cal E}nd({\cal T}_S))\textmap{\Delta}  \\
&\textmap{\Delta}  H^1(S, {\cal K}_S\otimes{\cal E}nd({\cal T}_S))\to \dots, 
\end{split}
\end{equation}
where $\Delta$ is the connecting morphism associated with (\ref{SESOmegaEnd}).

Since $ H^0(S, {\cal K}_S\otimes{\cal E}nd({\cal T}_S))=0$ by (1), the vanishing of $H^0(S,\Omega_S\otimes {\cal E}nd({\cal T}_S))$ follows from the following claim which will be proved below:
\vspace{2mm}\\
{\bf Claim:} $\ker(\Delta: H^0(S,{\cal E}nd({\cal T}_S))\to H^1(S, {\cal K}_S\otimes{\cal E}nd({\cal T}_S)))=0$.\vspace{2mm}

For this Claim we  use the following commutative diagram with  exact rows
\begin{equation}\label{2ExactSeqnew}
\begin{tikzcd}[column sep=8mm, row sep=9mm]
0 \ar[r]&{\cal K}_S\otimes{\cal E}nd({\cal T}_S)\ar[r]\ar[d, "\id_{{\cal K}_S}\otimes\tr" ]& \Omega_S\otimes {\cal E}nd({\cal T}_S)\ar[r]\ar[d, "\id_{\Omega_S}\otimes\tr"]&{\cal E}nd({\cal T}_S)\ar[r]\ar[d, "\tr"]& 0\\
0\ar[r]& {\cal K}_S \ar[r]& \Omega_S \ar[r]& {\cal O}_S\ar[r]&0
\end{tikzcd},
\end{equation}
where the vertical morphisms on the left are give by:

$$
(\id_{\Omega_S}\otimes\tr)(\omega\otimes f)=\omega\otimes \tr(f)=\tr(f)\omega.$$
By the functoriality of the connecting morphism, we obtain the commutative diagram

$$
\begin{tikzcd}[column sep=8mm, row sep=9mm]
0 \ar[r]&H^0(S,{\cal E}nd({\cal T}_S))\ar[r, "\Delta"]\ar[d, "H^0(\tr)" ]& H^1(S,{\cal K}_S\otimes{\cal E}nd({\cal T}_S))\ar[d, "H^1(\id_{{\cal K}_S}\otimes\tr)"]\\
0\ar[r]& H^0(S,{\cal O}_S)=\C   \ar[r, "\partial" ]& H^1(S,{\cal K}_S)
\end{tikzcd}\ ,
$$
where $\partial$ is the connecting morphism associated with the lower short exact sequence in (\ref{2ExactSeqnew}). We know that this exact sequence is non-split, so $\partial(1)\ne 0$, which shows that $\partial$ is injective.

Therefore the claim   will be proved if we show that $H^0(\tr)$ is injective.  Using the cohomology long exact sequence associated with the short exact sequence
$$
0\to {\cal E}nd_0({\cal T}_S)\hookrightarrow {\cal E}nd ({\cal T}_S) \textmap{\tr} {\cal O}_S\to 0
$$
we see that
$$
\ker\big(H^0(\tr): H^0(S,{\cal E}nd ({\cal T}_S))\to H^0(S,{\cal O}_S)\big)=H^0(S,{\cal E}nd_0({\cal T}_S)).
$$
Therefore our claim will be proved if we show that $H^0(S,{\cal E}nd_0({\cal T}_S))=0$. Using the central exact row in the first diagram in Lemma \ref{CommDiag} and the vanishing of $H^0(S,{\cal K}_S^*)$, it suffices to prove that $H^0(S,{\cal F}_0)=0$. This follows as in the proof of (1) using the vanishing of $H^0(S,{\cal K}_S)$, the commutative diagram  
\begin{equation}\label{newcommdiag}
\begin{tikzcd}[column sep=2em, row sep= 2em]
0\ar[d]&0\ar[d]\\
 H^0(S,{\cal K}_S)\ar[r, "\id"]\ar[d]& H^0(S,{\cal K}_S)\ar[d]\\
H^0(S,\Omega_S)\ar[r]\ar[d, ]&H^0(S,{\cal F}_0)\ar[d]\\
\C=H^0(S,{\cal O}_S)\ar[r, "\simeq"]\ar[d,"\partial_1" ]&H^0(S,{\cal O}_S)=\C\ar[d, "\partial_2"]\\
 H^1(S,{\cal K}_S)\ar[r, "\id"]& H^1(S,{\cal K}_S)&
\end{tikzcd}
\end{equation} 
associated with the second diagram in Lemma \ref{CommDiag} and the injectivity of $\partial_1$ (which follows by Remark \ref{CESII-III}).
\\

Finally suppose that $S$ is of type III. In this case we can complete the proof in two ways: First method: we use again the canonical exact sequence and the associated commutative diagrams given by Lemma \ref{CommDiag}. Second method: we use the double cover $\sigma:\tilde S\to S$ (with $\tilde S$ of type II) as in the proof of Remark \ref{extclass} and we note that 
$$
\sigma^*({\cal K}_S\otimes{\cal E}nd({\cal T}_S))={\cal K}_{\tilde S}\otimes{\cal E}nd({\cal T}_{\tilde S}),
$$
$$
\sigma^*(\Omega_S\otimes {\cal E}nd({\cal T}_S))=\Omega_{\tilde S}\otimes {\cal E}nd({\cal T}_{\tilde S}).
$$
Since $\sigma$ is surjective, the natural pull-back morphisms
$$
H^0(S,{\cal K}_S\otimes{\cal E}nd({\cal T}_S))\to H^0(\tilde S,\sigma^*({\cal K}_S\otimes{\cal E}nd({\cal T}_S)))=H^0(\tilde S, {\cal K}_{\tilde S}\otimes{\cal E}nd({\cal T}_{\tilde S}))
$$
$$
H^0(S,\Omega_S\otimes{\cal E}nd({\cal T}_S))\to H^0(\tilde S,\sigma^*(\Omega_S\otimes{\cal E}nd({\cal T}_S)))=H^0(\tilde S, \Omega_{\tilde S}\otimes{\cal E}nd({\cal T}_{\tilde S}))
$$
are injective, so the vanishing of $H^0(S,{\cal K}_S\otimes{\cal E}nd({\cal T}_S))$, $H^0(S,\Omega_S\otimes{\cal E}nd({\cal T}_S))$ follow from the vanishing of $H^0(\tilde S, {\cal K}_{\tilde S}\otimes{\cal E}nd({\cal T}_{\tilde S}))$, $H^0(\tilde S, \Omega_{\tilde S}\otimes{\cal E}nd({\cal T}_{\tilde S}))$.

\end{proof}

\begin{lm}\label{CommDiag}
Let $X$ be a complex manifold and let 
\begin{equation}\label{jqES}
0\to {\cal M}\textmap{j} {\cal E}\textmap{q}	{\cal Q}\to 0
\end{equation}
be an exact sequence of locally free sheaves, where ${\cal M}$, ${\cal Q}$ are  of rank 1. Let
$$qj:{\cal E}nd({\cal E})\to {\cal H}om({\cal M},{\cal Q})={\cal M}^*\otimes{\cal Q},\ (qj)_0:{\cal E}nd_0({\cal E})\to {\cal H}om({\cal M},{\cal Q})={\cal M}^*\otimes{\cal Q}$$
be the morphisms defined $\varphi\mapsto q\circ \varphi\circ j$. Put 
$${\cal F}\edf \ker(qj),\ {\cal F}_0={\cal F}\cap{\cal E}nd_0({\cal E})=\ker(qj)_0.$$
\begin{enumerate}
\item 	The sheaf morphism $jq:{\cal H}om({\cal Q},{\cal M})\to {\cal E}nd({\cal E})$ defined by $\varphi\mapsto j\circ\varphi\circ q$ takes values in ${\cal F}_0$, so it defines a morphism $(jq)_0: {\cal H}om({\cal Q},{\cal M})\to {\cal F}_0$.
\item ${\cal F}$ comes with natural morphisms $u:{\cal F}\to {\cal E}nd({\cal M})$, $v:{\cal F}\to {\cal E}nd({\cal Q})$ defined as follows: for any $x\in X$ and $\varphi\in {\cal F}_x$, we have $\varphi({\cal M}_x)\subset {\cal M}_x$, so $\varphi$ induces morphisms $u(\varphi)\in {\cal E}nd({\cal M}_x)$, $v(\varphi)\in {\cal E}nd({\cal E}_x/{\cal M}_x)={\cal E}nd({\cal Q}_x)$.  

\item The restriction $(u,v)_0$ of the obtained morphism 
$$(u,v):{\cal F}\to {\cal E}nd({\cal M})\oplus {\cal E}nd({\cal Q})={\cal O}_X\id_{\cal M}\oplus {\cal O}_X\id_{\cal Q}$$
to ${\cal F}_0$ takes values in the trivial invertible sheaf ${\cal O}_X(\id_{\cal M},-\id_{\cal Q})$.

\item The composition of the morphism $(j\circ):{\cal H}om({\cal E},{\cal M})\to {\cal E}nd({\cal E}) $ with the trace-free morphism ${\cal E}nd({\cal E})\to {\cal E}nd_0({\cal E})$ takes values in ${\cal F}_0$, so it induces a morphism 
$$
(j\circ)_0:{\cal H}om({\cal E},{\cal M})\to {\cal F}_0.
$$
\item We have   commutative diagrams with exact lines and  columns
\begin{equation}\label{big-diag-0}
\begin{tikzcd}[column sep=1.3em, row sep= 3em]
&0\ar[d]&0\ar[d]\\
0\ar[r]&{\cal H}om({\cal Q},{\cal M})\ar[r,"\id"]\ar[d, "(jq)_0"]&{\cal H}om({\cal Q},{\cal M})\ar[d, "i_{{\cal F}_0}\circ (jq)_0"]\ar[r]&0\ar[d]\\
0\ar[r]&{\cal F}_0\ar[r, hook, "i_{{\cal F}_0}"]\ar[d, "{(u,v)_0}"]&{\cal E}nd_0({\cal E})\ar[r, "(qj)_0"]\ar[d, "\circ j"]&	 {\cal H}om({\cal M},{\cal Q})\ar[d, "\id"] \ar[r]&0\\
0\ar[r]&{\cal O}_X(\id_{\cal M},-\id_{\cal Q})\simeq{\cal O}_X\id_{\cal M}\ar[d]\ar[r, "\circ j" ]& {\cal H}om({\cal M},{\cal E})\ar [r, "q\circ "]\ar[d] & {\cal H}om({\cal M},{\cal Q})\ar[r]\ar[d]& 0\\
&0&0&0
\end{tikzcd}
\end{equation}
\begin{equation}
\begin{tikzcd}[column sep=3em, row sep= 3em]
0\ar[d]&0\ar[d]\\
{\cal H}om({\cal Q},{\cal M})\ar[r, "\id"]\ar[d, "\circ q " ]&{\cal H}om({\cal Q},{\cal M})\ar[d, "(jq)_0"]\\
{\cal H}om({\cal E},{\cal M})\ar[r, "(j\circ)_0"]\ar[d, "\frac{1}{2}\circ j "]&{\cal F}_0\ar[d, "{(u,v)_0}"]\\
{\cal H}om({\cal M},{\cal M})={\cal O}_X\id_{\cal M}\ar[r, "\simeq"]\ar[d]&{\cal O}_X(\id_{\cal M},-\id_{\cal Q})\ar[d]\\
0&0&
\end{tikzcd}
\end{equation} 
\end{enumerate}	
\end{lm}
\begin{proof}
	The statement  follows from the analogue statement for an exact sequence of holomorphic bundles, which is a consequence of the analogue statement for an an exact sequence $0\to M\textmap{j} E\textmap{q}	Q\to 0$ of vector spaces. The latter follows easily using a basis $(u,v)$ of $E$ with $u\in M$.
\end{proof}

Using the second statement in Theorem \ref{VanTh}, we obtain

\begin{thry}\label{HolConn}
Let $S=U/\Gamma$ be an Inoue surface. The tangent bundle $T_S$ admits a unique holomorphic connection. The pull-back of this connection to $U$ coincides with the trivial connection on $T_U=U\times\C^2$, in particular this connection is flat and torsion free. 	
\end{thry}

\begin{proof}
For the existence: the trivial connection $\nabla_0$ on $T_U=U\times\C^2$ is holomorphic, flat, torsion free and $\Gamma$-invariant, so it descends to a flat, torsion free holomorphic connection   on the quotient $S=U/\Gamma$.

For the unicity: if non-empty,  the space of holomorphic connections on $T_S$ is an affine space with model vector space $H^0(S,\Omega_S\otimes {\cal E}nd({\cal T}_S))$. But this vector space vanishes by Theorem \ref{VanTh}.
\end{proof}

\section{The classification of Inoue surfaces}

Recall that any Inoue surface is the quotient $S=U/\Gamma$, where $\Gamma$ is a subgroup of the group $\Aff(U)$ acting properly discontinuously on $U$ (see section \ref{intro}).\\

Let $S'=U/\Gamma'$, $S''=U/\Gamma''$ be Inoue surfaces and $p':U\to S'$, $p'':U\to S''$ be the corresponding covering projections.  Let $f:S'\to S''$ be a biholomorphism. Since $p''$ and $f\circ p'$ are both universal covers of $S''$, it follows that there exists a biholomorphism $\tilde f:U\to U$ lifting $f$, i.e. such that $p''\circ \tilde f=f\circ p'$.  This gives the commutative diagram
\begin{equation}\label{USS'}
\begin{tikzcd}[row sep=3em]
U\ar[r,"\tilde f"]\ar[d, "p'"]& U\ar[d, "p''"]\\
S'=U/\Gamma'\ar[r, "f"]&	U/\Gamma''=S''
\end{tikzcd}	
\end{equation}
and shows that, for such a lift $\tilde f$, we have:
\begin{equation}\label{GammConj}
\tilde f\circ \Gamma'\circ  \tilde 	f^{-1} =\Gamma''.\end{equation}

An important role in our following results will be played by the following theorem, which reduces the classification of Inoue surfaces to an algebraic problem:
\begin{thry}\label{tilde f is affine}
Let	$S'=U/\Gamma'$, $S''=U/\Gamma''$ be Inoue surfaces, $f:S'\to S''$ be a biholomorphism and  $\tilde f:U\to U$ be a lift of $f$. Then $\tilde f \in \Aff(U)$.
\end{thry}
\begin{proof} Let $\nabla'$, $\nabla''$ be the (unique) holomorphic connections on the tangent bundles $T_{S'}$,  $T_{S''}$ respectively.  By Theorem \ref{HolConn} we know that $p'^*(\nabla')=p''^*(\nabla'')=\nabla_0$.

The pull-back $f^*(\nabla'')$ is a holomorphic connection on $S'$, so, by our unicity result, it coincides with $\nabla'$.  Therefore 
\begin{equation}
\begin{split}
\nabla_0&=p'^*(\nabla')=p'^*(f^*(\nabla''))=(f\circ p')^*(\nabla'')=\\
&=(p''\circ\tilde f)^*(\nabla'')=\tilde f^*(p''^*(\nabla''))=\tilde f^*(\nabla_0).	
\end{split}	
\end{equation}

Therefore $\tilde f$ is $\nabla_0$-affine. This implies that the tangent map $\tilde f_*$ preserves the space $\left\langle \frac{\partial}{\partial w},\frac{\partial}{\partial z}\right\rangle_\C$ of $\nabla_0$-parallel holomorphic  vector fields  on $U$. Therefore  the Jacobian matrix of $\tilde f$ is constant on $U$, so, since $U$ is connected, $\tilde f$ is the restriction to $U$ of an affine transformation (which fixes $U$).
	
\end{proof}

Taking into account Theorem \ref{tilde f is affine} and formula (\ref{GammConj}), we obtain the following result, which reduces the classification of Inoue surfaces to a purely algebraic problem:

\begin{co}\label{ClassifCoro}
Two Inoue surfaces 	 $S'=U/\Gamma'$, $S''=U/\Gamma''$ are biholomorphic if  and only if the subgroups $\Gamma'$, $\Gamma''$ of the group $\Aff(U)$ belong to the same conjugacy class. 
\end{co}

Our next goal: based on this corollary, we will obtain explicit descriptions of isomorphism classes of Inoue surfaces  of a fixed type.

\subsection{The classification of type I Inoue surfaces}

The goal of this section is Theorem \ref{Class-Type-I-similitude} stated at the end of the section, which reduces the classification of type I Inoue surfaces to a purely algebraic problem.

Recall from section \ref{Intro-typeI} that the space of parameters for type I Inoue surfaces is
$$
\Pg=\{(\beta,a,b)\in \C\times \R^3\times\C^3|\  \beta \hbox{ is  admissible, } 
 (a,b)\hbox{ is }  \beta \hbox{-compatible}\}.	
$$
Taking into account Corollary \ref{ClassifCoro}, our problem  becomes:
\begin{pb} Let $\beta$, $\beta'$ be   admissible complex numbers. Let $(a,b)\in {\cal P}_{\beta}$ be a  $\beta$-compatible pair,  and $(a',b')\in P_{\beta'}$ a   $\beta'$-compatible pair (see section \ref{Intro-typeI}). 
Describe explicitly the quotient of $\Pg$ by the equivalence relation
$$
(\beta,a,b)\sim (\beta',a',b') \hbox{ if $G(\beta,a,b)$, $G(\beta',a',b')$ are conj. in $\Aff(U)$}.
$$
\end{pb}
\begin{re}
	Let $\beta$ be an admissible complex number. The formula
$$
(K,(\mu,\nu))\cdot (a,b)\edf K\cdot((\mu,\nu)\cdot(a,b))\edf(K\mu a,K\nu b)
$$	
defines an action of the product
$
\GL(3,\Z)\times (\R^*_+\times \C^*)  
$
on the space ${\cal P}_{\beta}$. \end{re}

\begin{pr}\label{equiv-sim-orbits}
Let  $(\beta,a,b)$, $(\beta',a',b')\in \Pg$. The following conditions are equivalent:
\begin{enumerate}
	\item $(\beta,a,b)\sim (\beta',a',b')$.
	\item $ \beta= \beta'$ and $(a',b')\in \GL(3,\Z)\times (\R^*_{+}\times \C^*)(a,b)$.
\end{enumerate} 
Therefore the quotient set $	\qmod{\Pg}{\sim}$ is identified with the orbit space 
$$\qmod{\Pg}{\GL(3,\Z)\times (\R^*_{+}\times \C^*)}=\coprod_{\substack{\beta\in\C   \\ \rm{admissible}}}
\qmod{{\cal P}_{\beta}}
{\GL(3,\Z)\times (\R^*_{+}\times \C^*)}.
$$

\end{pr}
\begin{proof}
	Put $G\edf G(\beta,a,b)$, $G'\edf G(\beta',a',b')$,
$$
g_i\edf g_i(a,b),\ g'_i\edf g_i(a',b') \hbox{ for }1\leq i\leq 3,\  g_0\edf g_0(\beta),\ g'_0\edf g_0(\beta').
$$
First, suppose $G$, $G'$ are conjugate in $\Aff(U)$, and let $\theta\in\Aff(U)$,
$$
\theta\begin{pmatrix}
w\\ z	
\end{pmatrix}=\begin{pmatrix}
\mu &0\\
\lambda&\nu	
\end{pmatrix}\begin{pmatrix}
w\\z 	
\end{pmatrix}+ \begin{pmatrix}
u\\ \zeta 	
\end{pmatrix} \hbox { with } \mu\in\R^*_+,\ \nu\in\C^*,\ u\in\R,\ (\lambda,\zeta)\in\C^2
$$
 be such that 
\begin{equation}\label{GG'}
G(\beta',a',b')=\theta G(\beta,a,b)  \theta^{-1}.	
\end{equation}
Formula (\ref{GG'}) shows in particular that
\begin{equation}
g'\edf \theta\circ g_0\circ\theta^{-1}\in G'.
\end{equation}
We can write
$$
g'=g_3'^{k_3}g_2'^{k_2}g_1'^{k_1}g_0'^{k_0}
$$
with $k_i\in\Z$. Comparing the linear parts of $g'$ and $\theta\circ g_0\circ \theta^{-1}$ we obtain
\begin{equation}
\mu^{-1}\lambda(|\beta|^{-2}-\beta)=0,\ |\beta|^{-2}	=|\beta'|^{-2k_0},\ \beta=\beta'^{k_0},
\end{equation}
in particular $\lambda=0$ (i.e. the linear part of $\theta$ is diagonal) and (since $|\beta|<1$, $|\beta'|<1$) $k_0\in\N^*$. Changing the roles of $G$ and $G'$ we obtain a formula of the form $\beta'=\beta^{l_0}$ with $l_0\in\N^*$. It follows $\beta=\beta^{k_0l_0}$, so $k_0=l_0=1$  and $\beta=\beta'$.

Note first that (\ref{GG'}) implies
\begin{equation}
 \theta (G\cap\T(U))\theta^{-1}=(\theta G\theta^{-1}) \cap (\theta\T(U)\theta^{-1}) =(\theta G\theta^{-1})\cap\T(U)=G'\cap\T(U).
\end{equation}
Taking into account Remark \ref{ginG}(2) it follows that $\gamma'_i\edf \theta g_i\theta^{-1}$ belongs to 
$$G'\cap\T(U)=\langle g'_1,g'_2,g'_3\rangle$$
 for $1\leq i\leq 3$, and $(\gamma'_1,\gamma'_2,\gamma'_3)$ is a system of generators of this free abelian group. In other words there exists $K\in\GL(3,\Z)$ such that
\begin{equation}\label{gamma'i1}
\gamma'_i =g_1'^{k_{i1}}g_2'^{k_{i2}}g_3'^{k_{i3}} \hbox{ for }1\leq i\leq 3.
\end{equation}
On the other hand, since $\lambda=0$, we have
\begin{equation}\label{gamma'i2}
\gamma'_i\begin{pmatrix}
w\\z 	
\end{pmatrix}= (\theta g_i\theta^{-1})\begin{pmatrix}
w\\z 	
\end{pmatrix}=\begin{pmatrix}
w+\mu a_i\\z+ \nu b_i 	
\end{pmatrix}.
\end{equation}
Comparing (\ref{gamma'i1}) with (\ref{gamma'i2}) we obtain $ a'=K^{-1}\mu a$, $b'=K^{-1}\nu b$, with $K^{-1}\in\GL(3,\Z)$, which proves the claim. 
\\ \\
2. Conversely, suppose that $\beta=\beta'$ and  $(a',b')\in \GL(3,\Z)\times (\R^*_{+}\times \C^*)(a,b)$. Therefore there exists $(K,(\mu,\nu))\in\GL(3,\Z)\times (\R^*_+\times \C^*)$ such that 

\begin{equation}\label{Kmn1}
	(a',b')\edf K\cdot((\mu,\nu)\cdot(a,b)).
 \end{equation}

The linear transformation $\phi$ defined by the diagonal matrix $\begin{pmatrix} \mu&0\\ 0 &\nu\end{pmatrix}$ belongs to $\Aff(U)$ and (\ref{Kmn1}) shows that
$$
\phi g_i\phi^{-1}=g_1'^{l_{i1}}g_2'^{l_{i2}}g_3'^{l_{i3}} \hbox{ for }1\leq i\leq 3,
$$
where $L=(l_{ij})\coloneq K^{-1}$. Noting that $\phi g_0\phi^{-1}=g_0=g_0'$, it follows that 
$$(\phi g_0\phi^{-1},\phi g_1\phi^{-1}, \phi g_2\phi^{-1}, \phi g_3\phi^{-1})$$
 is a system of generators of $G'$, so $\phi G\phi^{-1}=G'$. 
\end{proof}

\begin{pr}\label{orbits-similitude}
Let $\beta$ be an admissible complex number. Let $(a,b)$, $(a',b') \in P_{\beta}$ be  $\beta$-compatible pairs. The following conditions are equivalent:
\begin{enumerate}
\item $(a',b')\in \big(\GL(3,\Z)\times (\R^*_{+}\times \C^*)\big)(a,b)$
	\item $M(\beta,a,b)\sim  M(\beta,a',b')$,
where $\sim$ denotes the similarity  relation on $\SL(3,\Z)$, i.e. the equivalence relation defined by conjugation with matrices in  $\GL(3,\Z)$.
\end{enumerate}
\end{pr}

\begin{proof}
$(1)\Rightarrow (2)$:	Suppose that $(a',b')\in \GL(3,\Z)\times (\R^*_{+}\times \C^*)(a,b)$. Therefore there exists $(K,(\mu,\nu))\in\GL(3,\Z)\times (\R^*_+\times \C^*)$ such that 
\begin{equation}\label{Kmn2}
	(a',b')\edf K\cdot((\mu,\nu)\cdot(a,b))\edf (K\mu a,K\nu b).
 \end{equation}
Then, %
$$
\begin{pmatrix} a'_1 & b'_1 & \overline b'_1 \\ a'_2& b'_2 &\overline b'_2 \\ a'_3 & b'_3 & \overline b'_3 \end{pmatrix}=K\begin{pmatrix} a_1 & b_1 & \overline b_1 \\ a_2& b_2 &\overline b_2 \\ a_3 & b_3 & \overline b_3 \end{pmatrix}\begin{pmatrix} \mu & 0 & 0 \\ 0&\nu &0 \\0 &0 & \bar\nu\end{pmatrix},$$
so 
$$
M(\beta,a',b')=K M(\beta,a,b)K^{-1} ,
$$
 which proves the claim. 
\\ \\
  $(2)\Rightarrow (1)$: Suppose that  $M(\beta,a,b)\sim M'(\beta,a',b')$. Therefore there exists $K\in\GL(3,\Z)$ such that  $M(\beta,a,b)=K M(\beta,a',b')K^{-1}$, so comparing the eigenspaces of the two matrices, it follows that $Ka'\in\R^*a$, $Kb'\in\C^*b$, so there exists $(\lg,\mg)\in\R^*\times\C^*$ such that
\begin{equation}\label{lgmg}
Ka'=\lg a,\ Kb'=\mg b.
\end{equation}
If $\lg>0$, it follows that $(a',b')\edf K^{-1}\cdot ((\lg,\mg)\cdot (a,b) $, so the claim is verified.
If  $\lg<0$, then $
(-K)a'=(-\lg) a,\ (-K)b'=(-\mg) b$.
Therefore 
$$
(a',b')\edf -K^{-1}\cdot ((-\lg,-\mg)\cdot (a,b),
$$
with $(-K^{-1},((-\lg,-\mg))\in\GL(3,\Z)\times (\R^*_{+}\times \C^*)$.
\end{proof}
Combining Propositions \ref{equiv-sim-orbits}, \ref{orbits-similitude} and taking into account Corollary \ref{ClassifCoro}, we obtain: 
\begin{co} \label{simeq-similitude}
Let $\beta$, $\beta'$ be  admissible complex numbers. Let $(a,b)\in P_{\beta}$ be a  $\beta$-compatible pair,  and let $(a',b')\in P_{\beta'}$ be a   $\beta'$-compatible pair. 

Then  
	$$S^{\beta}_{a,b}\simeq S^{\beta'}_{a',b'}\Longleftrightarrow \begin{cases}
		\beta=\beta' \\ \rm and \\ M(\beta,a,b)\sim  M(\beta',a',b')
	\end{cases},$$
where $\sim$ denotes the similarity  relation on $\SL(3,\Z)$, i.e. the equivalence relation defined by conjugation with matrices in  $\GL(3,\Z)$.
\end{co}

Now put
$${\cal M}\edf \big\{M\in\SL(3,\Z)|\ \Spec_\C(M)\supsetneq \Spec_\R(M)\subset ]1,+\infty[\big\}.$$
For a similarity class $\Mg=[M]\in \qmod{{\cal M}}{\sim}$\,, put $\Spec_\R(\Mg)= \Spec_\R(M)$, $\Spec_\C(\Mg)= \Spec_\C(M)$.

Note that knowing the similarity class of $M(\beta,a,b)$,  $\beta$ {\it is only determined  up to conjugation}. With this remark, Corollary \ref{simeq-similitude} becomes:

\begin{thry}\label{Class-Type-I-similitude}
	Biholomorphism classes of Inoue surfaces of type 1 correspond bijectively to  elements of  the set
$$
\left \{(\Mg,\beta)|\ \Mg\in \qmod{{\cal M}}{\sim},\ \beta\in \Spec_\C(\Mg)\setminus \Spec_\R(\Mg)\right\}=\hspace{-1.5mm}\coprod_{\Mg\in \qmod{{\cal M}}{\sim}} \big (\Spec_\C(\Mg)\setminus \Spec_\R(\Mg)\big ).
$$ 
The pair corresponding to the surface $S^{\beta}_{a,b}$ is $([M(\beta,a,b)],\beta)$.
	\end{thry}
Therefore, 
\begin{re} \label{Class-Type-I-similitude-rem} The set of  biholomorphism classes of Inoue surfaces of type 1 comes with a natural surjective map onto the set ${{\cal M}}/{\sim}$ of similarity class of matrices $M\in {\cal M}$. For any  similarity  class $\Mg$ we have {\it two} biholomorphism classes of type I Inoue surfaces. 
	
\end{re}

\begin{co}\label{S-non-isom-barS}
For any type I Inoue surface $S$, we have	 $S\not  \simeq \bar S$. In particular   a type I Inoue surface does not admit any Real structure\footnote{See \cite{Fr} and \cite{Kh} for classification results for Real structures on other classes of non-Kählerian surfaces.}.
\end{co}
\begin{proof}
Let $c:U\to U$ be the canonical Real structure  on $U$ defined by
$
c(w,z)=(-\bar w,\bar z)$. 
Note that $c$ induces an anti-holomorphic  diffeomorphism   
$$
\tilde c: U/G(\beta,a,b)\to U/c\,G(\beta,a,b)\,c^{-1},
$$
which gives a biholomorphism $
\bar S^{\beta}_{a,b}\simeq  U/c\,G(\beta,a,b)\,c^{-1}$.
Elementary computations give 
$$c\circ g_0(\beta)\circ c^{-1}=g_0(\bar \beta),\ c\circ g_i(a,b)\circ c^{-1}=g_i(-a,\bar b) \hbox{ for }1\leq i\leq 3,$$
so $
c\,G(\beta,a,b)\,c^{-1}=G(\bar\beta,-a,\bar b)$,
so $
\bar S^{\beta}_{a,b}\simeq S^{\bar \beta}_{-a,\bar b}$.  We have
%
$M(\bar \beta,-a,\bar b)=M(\beta,a,b)$. Therefore, via the bijection given by Theorem \ref{Class-Type-I-similitude}, $S^{\beta}_{a,b}$ et $\bar S^{\beta}_{a,b}$ correspond to the  same similarity class $\Mg$, but to different eigenvalues $\beta$.
\end{proof}

As explained in the introduction, for any polynomial $P(X)=X^3-\theta_2 X^2+\theta_1 X-1\in\Z[X]$ with a real root $\alpha\in]1,+\infty[$ and two non-real roots, the set of $\GL(3,\Z)$-similarity classes of matrices $M\in \SL(3,\Z)$ with $\chi_M(X)=P(X)$ correspond bijectively  to equivalence classes of ideals in the order $\Z[\alpha]$. This follows from the Latimer-MacDuffee Theorem taking into account Remark \ref{P(X)-is-irred}. Since the set of equivalence classes of ideals in an order is finite \cite[Theorem 3, p. 128]{BoSh}, we obtain the following finiteness theorem:
\begin{thry}\label{Finiteness-Theorem-TypeI}
To any polynomial 	$P(X)=X^3-\theta_2 X^2+\theta_1 X-1\in\Z[X]$ with a real root $\alpha\in]1,+\infty[$ and two non-real roots, correspond exactly $2 h_{\Z[\alpha]}$ isomorphism classes of type I Inoue surfaces, where $h_{\Z[\alpha]}$ is the number of equivalence classes of ideals in the order $\Z[\alpha]$. 
\end{thry}

Note that in general the order $\Z[\alpha]\subset K\edf \Q(\alpha)$ is not maximal, i.e. it does not necessarily coincide with the ring ${\cal O}_K$ of integers of the algebraic number field $K$. We are grateful to Stéphane Louboutin for explaining us the following example:
\begin{ex}\label{Louboutin}
The polynomial $P_1(X)= X^3-2X^2-2X-1$ has a  single real root $\varepsilon \approx 2.83117$ and its discriminant is $\Delta_1=-83$. Put $K=\Q(\varepsilon)$. Using \cite[formula (8) p. 38]{Mil} and the equality $\mathrm{disc}(1,\varepsilon,\varepsilon^2)=\Delta_1$ (see  \cite[p. 37]{Mil}), we obtain
$$
\Delta_1=[{\cal O}_K:\Z[\varepsilon]]^2\mathrm{disc}({\cal O}_K)	.
$$
Since $\Delta_1$ is a prime number, it follows  $\mathrm{disc}({\cal O}_K)=\Delta_1=-83$ and $[{\cal O}_K:\Z[\varepsilon]]=1$, i.e.  $\Z[\varepsilon]={\cal O}_K$.  Now note that $\alpha\edf \varepsilon^2\approx     8.01556$ is the only real root of the polynomial $P_2(X)=X^3-8X^2-1$ whose discriminant is $\Delta_2=25 \Delta_1$. Noting that $\alpha \varepsilon-8\alpha-1=0$, it follows $\Q(\alpha)=\Q(\varepsilon)=K$, and formula \cite[formula (8) p. 38]{Mil} applied to $\alpha$ gives $\Delta_2=[{\cal O}_K:\Z[\alpha]]^2\mathrm{disc}({\cal O}_K)$, so $[{\cal O}_K:\Z[\alpha]]=5$.

\end{ex}

\subsection{The classification of type II Inoue surfaces}
\label{ClassTypeII-section}

Recall from section \ref{Intro-typeII} that the space of parameters for type II Inoue surfaces is
\begin{align*}
\Pg=\{(\alpha,a,b,c,r,t)\in\ ]1,+\infty[\times \R^2\times\R^2\times\R^2\times\N^*&\times\C|\ \alpha \hbox{ is $S^+$-admissible, and}\\ 
&(a,b,c)\hbox{ is } (\alpha,r)-\hbox{compatible}\}.	
\end{align*}

We begin with the simple
\begin{re}\label{r=r'}
Let $\pg=(\alpha,a,b,c,r,t)$, $\pg'=(\alpha',a',b',c',r',t')\in\Pg$ such that $S_\pg\simeq S_{\pg'}$. Then $\alpha=\alpha'$, $r=r'$.
\end{re}
\begin{proof}

By  Corollary \ref{ClassifCoro}, there exists $\theta\in\Aff(U)$ such that 
$$\theta G(\alpha,a,b,c,r,t)\theta^{-1}=G(\alpha',a',b',c',r',t').$$

In particular $\theta g_0(\alpha,t)\theta^{-1}= g'$ for an element $g'\in G(\alpha',a',b',c',r',t')$. Using Remark 
\ref{gen-form-g-in-G(a,b,c,r-t)} and the explicit computation of the right hand term of formula (\ref{g-in-G(a,b,c,r,t)}), we obtain
$\alpha=\alpha'^l$ with $l\in\Z$. Since $\alpha$, $\alpha'>1$, we have $l\in\N^*$. Changing the roles, we obtain $\alpha'=\alpha^{l'}$ with $l'\in\N^*$.  It follows $ll'=1$, so $l=l'=1$, so $\alpha=\alpha'$.

We have 
$$G(a,b,c,r)= G(\alpha,a,b,c,r,t)\cap \Aff^1_1,\ G(a',b',c',r')= G(\alpha',a',b',c',r',t')\cap \Aff^1_1.$$
Since $\Aff^1_1$ is normal in $\Aff(U)$  and $\theta G(\alpha,a,b,c,r,t)\theta^{-1}=G(\alpha',a',b',c',r',t')$, it follows

\begin{equation}\label{conj-eq}
\theta G(a,b,c,r)\theta^{-1}=G(a',b',c',r').
\end{equation}

This implies 
$$
\theta (G(a,b,c,r)\cap \T(U)) \theta^{-1}=G(a',b',c',r')\cap (\theta  \T(U)\theta^{-1})=G(a',b',c',r')\cap\T(U),
$$
so 
$$
\theta \langle g_3(a,b,r)\rangle \theta^{-1}= \langle g_3(a',b',r')\rangle.
$$
On the other hand formula (\ref{conj-eq}) also implies
$$
\theta [G(a,b,c,r),G(a,b,c,r)]\theta^{-1}=[G(a',b',c',r'),G(a',b',c',r')].
$$
Therefore the isomorphism $\langle g_3(a,b,r)\rangle\to \langle g_3(a',b',r')\rangle$ induced by the interior automorphism $\iota_\theta$ maps the subgroup $ [G(a,b,c,r),G(a,b,c,r)]$ of  $\langle g_3(a,b,r)\rangle$ onto the subgroup $G(a',b',c',r'),G(a',b',c',r')]$ of $\langle g_3(a',b',r')\rangle$. It follows that $\iota_\theta$  induces an isomorphism
$$
\begin{tikzcd}
 \qmod{\langle g_3(a,b,r)\rangle}{ [G(a,b,c,r),G(a,b,c,r)]}\ar[d, "\simeq"]\ar[r, "\simeq"]& \Z_r \\ \qmod{\langle g_3(a',b',r')\rangle}{ [G(a',b',c',r'),G(a',b',c',r')]} \ar[r, "\simeq"]& \Z_{r'}
\end{tikzcd}.
$$
On the right we have used the isomorphisms given by Remark \ref{[GG(a,b,c,r),GG(a,b,c,r)]}. So $r=r'$. 
 \end{proof}

Remark \ref{r=r'} shows that, for any pair $(\alpha,r)\in]1,+\infty[\times \N^*$ with $\alpha$ $S^+$-admissible, the subset $\Pg_{\alpha,r}$ obtained from $\Pg$ by fixing $\alpha$ and $r$ is saturated with respect to the equivalence relation defined by the condition $S_\pg\simeq S_{\pg'}$. 

Therefore we will fix such a pair $(\alpha,r)$ and we will study the induced equivalence relation on the subset 
$\Pg_{\alpha,r}$, which is obviously identified with the product 
$${\cal A}_{\alpha,r}\edf{\cal T}_{\alpha,r}\times\C,$$
 where ${\cal T}_{\alpha,r}$ is the space of $(\alpha,r)$-compatible triples  in the sense of Definition \ref{Def-alpha-r-compat}.
\vspace{2mm}

We will endow ${\cal T}_{\alpha,r}$  with its natural topology, namely with the coarsest topology which makes the maps
$$
(a,b,c)\mapsto N(\alpha,a,b)\in \SL(2,\Z), (a,b,c)\mapsto \pi_{a,b,r}(c)\in\Z^2,\ (a,b,c)\mapsto (a,b)\in \R^2
$$
continuous. With respect to this topology, 
$$
 \big\{\{(\mu,\rho)\mapsto (\mu a, \rho b,\mu\rho c)|\ \mu,\ \rho\in ]1-\varepsilon,1+\varepsilon[\}|\ \varepsilon\in]0,1[\big\}
$$
is a fundamental system of neighbourhoods of $(a,b,c)\in {\cal T}_{\alpha,r}$, and we endow ${\cal A}_{\alpha,r}$ with the product topology. The quotient spaces we define will be endowed with the quotient topologies.

\begin{re}\label{ClassRem(alpha,r)}
Taking into account Corollary \ref {ClassifCoro}, our problem becomes: 
\begin{pb} Let $(\alpha,r)\in]1,+\infty[\times \N^*$ with $\alpha$ $S^+$-admissible.
Describe explicitly the quotient of ${\cal A}_{\alpha,r}\edf{\cal T}_{\alpha,r}	\times\C$ by the equivalence relation
$$
(a,b,c,t)\sim (a',b',c',t') \hbox{ if $G(\alpha,a,b,c,r,t)$, $G(\alpha,a',b',c',r,t')$ are conj. in $\Aff(U)$}.
$$
\end{pb}
\end{re}

\subsubsection{Parameterizing the groups defining type II Inoue surfaces}
We will first study the quotient of ${\cal A}_{\alpha,r}$
 by the finer equivalence relation
$$
(a,b,c,t)\approx  (a',b',c',t') \hbox{ if $G(\alpha,a,b,c,r,t)=G(\alpha,a',b',c',r,t')$}.
$$

The quotient set ${\cal A}_{\alpha,r}/\approx$ is obviously identified with the set of all subgroups $G\subset \Aff(U)$ defining type II Inoue surfaces.
%
%
\begin{re}\label{equiv-for-c'-t'}
Let $(a,b)\in {\cal P}_\alpha$, $c$, $c'\in {\cal C}_{a,b,r}$, $t$, $t'\in \C$.	   We have the equivalences
\begin{equation}\label{s(c)} 
g_i(a,b,c')=g_i(a,b,c) g_3(r,a,b)^{s_i}\ (1\leq i\leq 2)\Leftrightarrow 	c'=c+\frac{b\wedge a}{r}\bpm s_1\\s_2\epm.
\end{equation}
\begin{equation}\label{s(t)} 
g_0(\alpha,t')=g_0(\alpha,t)g_3(r,a,b)^{s_0}\Leftrightarrow t'=t+\frac{b\wedge a}{r}s_0.
\end{equation}
\end{re}

This shows that the  action   
\begin{equation}\label{Z3-action}
\left(\bpm s_0\\s_1\\s_2\epm,(a,b,c,t)\right)\mapsto \left(a,b,c+\frac{b\wedge a}{r}\bpm s_1\\s_2\epm, t+\frac{b\wedge a}{r}s_0  \right)	
\end{equation}
of the group $\Z^3$ on   the parameter space ${\cal A}_{\alpha,r}$   is compatible with the equivalence relation $\approx$, i.e. we have
\begin{equation*}
\forall \bpm s_0\\s_1\\s_2\epm \in\Z^3\ \forall 	(a,b,c,t)\in {\cal A}_{\alpha,r},\  \left(a,b,c+\frac{b\wedge a}{r}\bpm s_1\\s_2\epm, t+\frac{b\wedge a}{r}s_0  \right)  \approx (a,b,c,t).
\end{equation*}
We will also need the action of $\Z^2$ on ${\cal T}_{\alpha,r}$ obtained by omitting the fourth component in  (\ref{Z3-action}).

For a triple $(a,b,c)\in {\cal T}_{\alpha,r}$ and $t\in\C$, put $[c]_{ab}\edf c+\frac{b\wedge a}{r}\Z^2$, $[t]_{ab}\edf t+\frac{b\wedge a}{r}\Z $ and define
\begin{align*}
\mathscr{C}_{a,b,r}&\edf {\cal C}_{a,b,r}\bigg/\frac{b\wedge a}{r}\Z^2,  \\
\mathscr{T}_{\alpha,r}&\edf \left\{(a,b,[c]_{ab})\big|\,\   (a,b)\in {\cal P}_\alpha,\ [c]_{ab}\in \mathscr{C}_{a,b,r}\right\} \simeq  {\mathcal{T}_{\alpha,r}}/\Z^2 ,
\\
\mathscr{A}_{\alpha,r}&\edf   \left\{(a,b,[c]_{ab},[t]_{ab})\big|\,\   (a,b)\in {\cal P}_\alpha,\ [c]\in \mathscr{C}_{a,b,r},\ [t]\in \C\bigg/\frac{b\wedge a}{r}\Z\right\} \simeq  {\mathcal{A}_{\alpha,r}} /\Z^3.
\end{align*}
%

%
%
Using Remark \ref{c->p-remark}, we see that the bijection $\pi_{a,b,r}:{\cal C}_{a,b,r}\to\Z^2$ induces a bijection  
$$\pg_{a,b,r}:\mathscr{C}_{a,b,r}\to \Z^2/(I_2-N)\Z^2\simeq \Z_{e_1}\oplus\Z_{e_2},$$
where $e_1$, $e_2$ are the elementary divisors of the matrix $I_2-N$. 

\begin{re}
The quotient $\mathscr{C}_{a,b,r}= {\cal C}_{a,b,r}\big/\frac{b\wedge a}{r}\Z^2$ is always finite. It reduces to a singleton if and only if $\theta=3$.
\end{re}
\begin{proof}
The second statement follows from $\det(I_2-N)=(1-\alpha	)(1-\alpha^{-1})=2-\theta$ recalling that $\theta\geq 3$.
\end{proof}

 We will regard  $\mathscr{T}_{\alpha,r}$, $\mathscr{A}_{\alpha,r}$ as spaces over the set ${\cal P}_{\alpha}$ of $\alpha$-compatible pairs. The fibre  of $\mathscr{T}_{\alpha,r}$ over a pair $(a,b)$ is identified with the finite set $\mathscr{C}_{a,b,r}\simeq \Z_{e_1}\oplus\Z_{e_2}$, whereas the fibre of $\mathscr{A}_{\alpha,r}$ over  pair $(a,b)$ is $\mathscr{C}_{a,b,r}\times (\C\big/\frac{b\wedge a}{r}\Z)\simeq (\Z_{e_1}\oplus\Z_{e_2})\times\C^*$
 \\

The next step is to introduce a natural $\GL(2,\Z)$-action on  the new parameter space $\mathscr{A}_{\alpha,r}\simeq {\cal A}_{\alpha,r}/{\Z^3}$. 
We start with the following elementary lemma: 
\begin{lm}\label{explic-C-formula}
 Let $(a,b,c) \in {\cal T}_{\alpha,r}$  and  $K\in\GL(2,\Z)$.  Put $g_i\edf g_i(a,b,c)$ for $1\leq i\leq 2$, $g_3=g_3(a,b,r)$ and
 \begin{equation}\label{K-action-on-generators}
G_1\edf g_1^{k_{11}}g_2^{k_{12}},\ G_2\edf g_1^{k_{21}}g_2^{k_{22}},\ G_3\edf g_3^{\det(K)}.	
\end{equation}
Then $G_i=g_i(A,B,C)$ and $G_3=g_3(A,B,r)$ (in the sense of formulae (\ref{Def-gi})), where 
  \begin{equation}\label{ABC} 
\begin{split}
\hspace{-2mm}A= Ka,\ B=Kb,\   
C= \frac{1}{2}\bpm A_1B_1\\A_2B_2 \epm+K\bigg(c-\frac{1}{2}\begin{pmatrix}
 a_1b_1\\
a_2b_2	
\end{pmatrix} \bigg)  +\frac{b\wedge a}{2}\bpm k_{11}k_{12}\\ k_{21}k_{22}\epm.
\end{split}   	
\end{equation}

\end{lm}

\begin{pr} Let $(a,b,c)\in {\cal T}_{\alpha,r}$  and  $K\in\GL(2,\Z)$. Then
\begin{enumerate}
	\item The formulae (\ref{ABC})
define an element $(A,B,C)\in {\cal T}_{\alpha,r}$ which will be denoted $K\cdot (a,b,c)$.
\item The tripe  $(A,B,[C]_{AB})\in \mathscr{T}_{\alpha,r}$ depends only on the matrix $K$ and the  triple $(a,b,[c]_{ab})\in \mathscr{T}_{\alpha,r}$. Moreover, the formula
$$
K\cdot(a,b,[c]_{ab})\edf (A,B,[C]_{AB})
$$
defines an action of $\GL(2,\Z)$ on the space $\mathscr{T}_{\alpha,r}$.
\end{enumerate} 
\end{pr}

The complicated formula for $C$ is justified by the equivalence explained in Lemma \ref{explic-C-formula} stated above. 
 Note that formula $(K,(a,b,c))\mapsto K\cdot(a,b,c)$ does {\it not} define a $\GL(2,\Z)$-action on ${\cal T}_{\alpha,r}$. Therefore factorising the space of triples ${\cal T}_{\alpha,r}$ by $ \Z^2$ is an important step in our constructions.

\begin{proof} (1)
It's easy to see that $(A,B)$ remains an $\alpha$-compatible pair. The corresponding matrix is
$$
N(\alpha,A,B)=KN(\alpha,a,b)K^{-1}.
$$
It remains to check that $C$ is $(A,B,r)$ compatible, i.e.  	that, putting 
$$\Ng=(\ng_{ij})\edf N(\alpha,A,B),$$
we have
\begin{equation}\label{to-check}
(I_2-\Ng)\bigg(C-\frac{1}{2}\bpm A_1B_1\\ A_2B_2\epm\bigg)-\frac{B\wedge A}{2}\bpm \ng_{11}\ng_{12}\\ \ng_{21}\ng _{22}\epm\in \frac{B\wedge A}{r}\Z^2.	
\end{equation}

Putting $p\edf p(a,b,c,r)$, the first term on the left in (\ref{to-check}) is
$$
(I_2-\Ng)\bigg(C-\frac{1}{2}\bpm A_1B_1\\ A_2B_2\epm\bigg)=K(I_2-N)\bigg(c-\frac{1}{2}\begin{pmatrix}
 a_1b_1\\
a_2b_2	
\end{pmatrix} \bigg)+\frac{b\wedge a}{2}(I_2-\Ng)\bpm k_{11}k_{12}\\ k_{21}k_{22}\epm=$$
$$
=b\wedge a\bigg[K\bigg(\frac{1}{2}\begin{pmatrix} n_{11}n_{12}\\
n_{21}n_{22}
\end{pmatrix}+\frac{1}{r}p\bigg)+\frac{1}{2}(I_2-\Ng)\bpm k_{11}k_{12}\\ k_{21}k_{22}\epm\bigg].
$$
Put $\varepsilon_K\edf\det(K)\in\{\pm1\}$. The element $P$ (a priori in $\R^2$) associated with $(A,B,C,r)$ via formula (\ref{c->p-formula}) is
\begin{equation}\label{P-formula}
P
%
=\varepsilon_K \bigg[ Kp+\frac{r}{2}\bigg( K
\bpm n_{11}n_{12}\\
n_{21}n_{22}\epm+ (I_2-\Ng)\bpm k_{11}k_{12}\\ k_{21}k_{22}\epm-\varepsilon_K\bpm \ng_{11}\ng_{12}\\ \ng_{21}\ng _{22}\epm\bigg)\bigg]
\end{equation}

It suffices to prove that 
$$
K\bpm n_{11}n_{12}\\
n_{21}n_{22}\epm+ \bpm k_{11}k_{12}\\ k_{21}k_{22}\epm -\Ng \bpm k_{11}k_{12}\\ k_{21}k_{22}\epm-\varepsilon_K\bpm \ng_{11}\ng_{12}\\ \ng_{21}\ng _{22}\epm\in2\Z^2.
$$
Applying the elementary Lemma \ref{mod2congruence-lemma} stated below to the pairs $(K,N)$, $(\Ng, K)$, we obtain 
$$
K\bpm n_{11}n_{12}\\
n_{21}n_{22}\epm+ \bpm k_{11}k_{12}\\ k_{21}k_{22}\epm-\begin{pmatrix}
(KN)_{11}(KN)_{12}\\
(KN)_{21}(KN)_{22}
\end{pmatrix}\in 2\Z^2,
$$
$$
\Ng \bpm k_{11}k_{12}\\ k_{21}k_{22}\epm+\varepsilon_K\bpm \ng_{11}\ng_{12}\\ \ng_{21}\ng _{22}\epm-\begin{pmatrix}
(\Ng K)_{11}(\Ng K)_{12}\\
(\Ng K)_{21}(\Ng K)_{22}
\end{pmatrix}\in 2\Z^2.
$$
It suffices to note that $\Ng K=KNK^{-1}K=KN$.
\vspace{3mm}\\
(2) For the first claim of (2) note that replacing  $(a,b,c)\in {\cal T}_{\alpha,r}$ by $(a,b,c')$ where  $c'=\frac{b\wedge a}{r}\bpm s_1\\s_2\epm $ in the formula for $K\cdot(a,b,c)$, gives the triple $(A,B,C')$, where
$$
C'=C+ \frac{b\wedge a}{r} K\bpm s_1\\s_2\epm =C+\frac{B\wedge A}{r} \det(K)^{-1}K\bpm s_1\\s_2\epm\in C+\frac{B\wedge A}{r}\Z^2. 
$$

For the second claim of (2) we have to prove that for any $(a,b,c)\in {\cal T}_{\alpha,r}$ and $K$, $L\in\GL(2,\Z)$, the triples $L\cdot(K\cdot(a,b,[c]_{ab}))$, $(LK)\cdot(a,b,[c]_{ab})$ coincide in $\mathscr{T}_{\alpha,r}$.

Put
$$
\G_1\edf G_1^{l_{11}}G_2^{l_{12}}=(g_1^{k_{11}}g_2^{k_{12}})^{l_{11}}(g_1^{k_{21}}g_2^{k_{22}})^{l_{12}},$$
$$ \G_2\edf G_1^{l_{21}}G_2^{l_{22}}=(g_1^{k_{11}}g_2^{k_{12}})^{l_{21}}(g_1^{k_{21}}g_2^{k_{22}})^{l_{22}},$$
$$
 \G_3\edf G_3^{\det(L)}=g_3^{\det(LK)}.
$$
By Lemma \ref{explic-C-formula} applied to $(A,B,C)\in {\cal T}_{\alpha,r}$ and $L\in\GL(2,\Z)$, we have 
$\G_i=g_i(\Ag,\Bg,\Cg)$, where
$$
(\Ag,\Bg,\Cg)=L\cdot(A,B,C)=L\cdot(K\cdot(a,b,c)).
$$

Put $R=LK$. Using the commutation formulae $[g_1,g_2]=g_3^r$ we obtain 
$$
\G_1=g_1^{r_{11}}g_2^{r_{12}}g_3^{s_1(L,K)}, \ \G_2= g_1^{r_{21}}g_2^{r_{22}}g_3^{s_2(L,K)}
$$
where $s_i(L,K)\in\Z$. Put $(\Ag',\Bg',\Cg')\edf R(a,b,c)$, $\G'_i\edf g_i(\Ag',\Bg',\Cg')$. We have 
$$\Ag'=(LK)a=LA=\Ag,\ \Bg'=(LK)b=LB=\Bg.
$$

By Lemma \ref{explic-C-formula},   we have 
$$
\G'_1=g_1^{r_{11}}g_2^{r_{12}},\ \G'_2=g_1^{r_{21}}g_2^{r_{22}}, 
$$
in other words for $1\leq i\leq 2$ we have
$$
g_i(\Ag,\Bg,\Cg)=g_i(\Ag,\Bg,\Cg')g_3(a,b,r)^{s_i(L,K)}=g_i(\Ag,\Bg,\Cg')g_3(\Ag,\Bg,r)^{\det(R)s_i(L,K)}. $$
Taking into account the equivalence (\ref{s(c)}), this proves that 
$$\Cg=\Cg'+ \frac{\Bg\wedge \Ag}{r}\bpm\det(R)s_1(L,K)\\\det(R)s_2(L,K) \epm,$$
 which proves the claim.

 \end{proof}

\begin{lm}\label{mod2congruence-lemma}
Let $K$, $L\in \GL(2,\Z)$. Then
\begin{equation}\label{mod2congruence}
\begin{split}
L\begin{pmatrix}
k_{11}k_{12}\\
k_{21}k_{22}	
\end{pmatrix}+\det(K)&\begin{pmatrix}
l_{11}l_{12}\\
l_{21}l_{22}	
\end{pmatrix}-\begin{pmatrix}
(LK)_{11}(LK)_{12}\\
(LK)_{21}(LK)_{22}
\end{pmatrix}=\\
=&2\bpm\frac{l_{11}-l_{11}^2}{2}k_{11}k_{12}+\frac{l_{12}-l_{12}^2}{2}k_{21}k_{22}\\\frac{l_{21}-l_{21}^2}{2}k_{11}k_{12}+\frac{l_{22}-l_{22}^2}{2}k_{21}k_{22} \epm-2k_{12} k_{21}\begin{pmatrix}
l_{11}l_{12}\\
l_{21}l_{22}	
\end{pmatrix},
\end{split}
\end{equation}
which obviously belongs to $2\Z^2$ if $K$, $L\in \GL(2,\Z)$.
\end{lm}

\begin{proof} Direct computation.
	
\end{proof}

For a 4-tuple $((a,b,c),t)\in {\cal A}_{\alpha,r}$ we put 
$$
K\cdot(a,b,c,t)=(A,B,C,t)
$$

We need: the class of $K\cdot(a,b,c,t)=(A,B,C,t)$ mod $\frac{B\wedge A}{r}\Z^3$ depends only on $K$ and the class $[a,b,c,t]$.  This follows from the formula:
\begin{align*}
	K\cdot\big (a,b,c+&\frac{b\wedge a}{r}\bpm s_1\\s_2\epm ,t+\frac{b\wedge a}{r} s_0\big )\\
	&=\big(A,B,C+\frac{B\wedge A}{r}\det(K)K\bpm s_1\\s_2\epm,t+\frac{B\wedge A}{r}\det(K) s_0\big).
\end{align*}
Note also that the formula 
$$
K\cdot (a,b,[c]_{ab},[t]_{ab})\edf (A,B,[C]_{AB},[t]_{AB})
$$
defines a $\GL(2,\Z)$-action on $\mathscr{A}_{\alpha,r}$.

\begin{thry} \label{Quotient-approx} Let  $(a,b,c,t), \ (a',b',c',t')\in {\cal A}_{\alpha,r}$. The following conditions are equivalent:
\begin{enumerate}
	\item $(a,b,c,t)\approx (a',b',c',t')$.
	\item $(a',b',[c']_{a'b'},[t']_{a'b'})\in \GL(2,\Z)(a,b,[c]_{ab},[t]_{ab})$.
\end{enumerate} 
Therefore the quotient set $	{\cal A}_{\alpha,r}/\approx$ is identified with the orbit space $\qmod{\mathscr{A}_{\alpha,r}}{\GL(2,\Z)}$.
\end{thry}

\begin{proof} $(1)\Rightarrow(2)$:
Suppose that $(a,b,c,t)\approx (a',b',c',t')$, in other words we have $G(\alpha,a,b,c,r,t)=G(\alpha,a',b',c',r,t')$, in particular 
$$
g_0(\alpha,t'),\ g_3(a',b',r),\  g_i(a',b',c')  \  (1\leq i\leq 2)$$
is a  system of generators of $G(\alpha,a,b,c,r,t)$.

Put  $G=G(\alpha,a,b,c,r,t)=G(\alpha,a',b',c',r,t')$. By Remark \ref{RelationsSubgroups} we know that 
$$(g_1(a,b,c),g_2(a,b,c),g_3(a,b,r)),\ (g_1(a',b',c'),g_2(a',b',c'),g_3(a',b',r))
$$
are both system of generators of the group 
$$G\cap \Aff^1_1(U)=G(a,b,c,r)=G(a',b',c',r)$$
 and $g_3(a,b,r)$, $g_3(a',b',r)$ generate the same normal subgroup 
$$Z\edf Z(a,b,r)=Z(a',b',r)=\langle g_3(a,b,r)\rangle=\langle g_3(a',b',r)\rangle.$$
The two pairs   
 $$
([(g_1(a,b,c)]_Z,[(g_2(a,b,c)]_Z), \ ([(g_1(a',b',c')]_Z,[(g_2(a',b',c')]_Z)
 $$ 
of  congruence classes with respect to $Z$  are both systems of  	generators   of the quotient group  
$(G\cap \Aff^1_1(U))/Z\simeq\Z^2$. It follows that there exists $K=(k_{ij})\in\GL(2,\Z)$ such that
$$
[g_1(a',b',c')]_Z=[g_1(a,b,c)]^{k_{11}}_Z[g_2(a,b,c)]^{k_{12}}_Z,$$ 
$$[g_2(a',b',c')]_Z=[g_1(a,b,c)]^{k_{21}}_Z[g_2(a,b,c)]^{k_{22}}_Z.
$$
Therefore
\begin{equation}\label{gi(a',b',c')}
\begin{split}	
g_1(a',b',c')=&g_1(a,b,c)^{k_{11}}g_2(a,b,c)^{k_{12}}g_3(a,b,r)^{s_1},\\
 g_2(a',b',c')=&g_1(a,b,c)^{k_{21}}g_2(a,b,c)^{k_{22}}g_3(a,b,r)^{s_2}
 \end{split}
 \end{equation}
with $s_1$, $s_2\in\Z$.   Put
\begin{equation}\label{introd-G1G2}
G_1\edf g_1(a,b,c)^{k_{11}}g_2(a,b,c)^{k_{12}},\ G_2\edf g_1(a,b,c)^{k_{21}}g_2(a,b,c)^{k_{22}}. 
\end{equation}
By (\ref{gi(a',b',c')}) it follows that
$$
[G_1,G_2]=[g_1(a',b',c'),g_2(a',b',c')]=g_3(a',b',r)^r.
$$
On the other hand, taking into account that $[g_1(a,b,c),g_2(a,b,c)]=g_3(a,b,r)^r$ and the definitions of $G_1$, $G_2$, direct computations give $[G_1,G_2]=g_3(a,b,r)^{\det(K)r}$.
Therefore
\begin{equation}\label{g3'g3}
G_3\edf g_3(a',b',r)=g_3(a,b,r)^{\det(K)}.
\end{equation}

For comparing $g_0(\alpha,t)$ with $g_0(\alpha,t')$ recall that, by Remark \ref{RelationsSubgroups},  their classes with respect to the subgroup $G\cap \Aff^1_1(U)=G(a,b,c,r)=G(a',b',c',r)$ coincide, so
$$
g_0(\alpha,t')=g_0(\alpha,t)g_1(a,b,c)^{k_1}g_2(a,b,c)^{k_2}g_3(a,b,r)^{s_0}
$$
with $k_1$, $k_2$, $s_0\in\Z$. Putting $k=(k_1,k_2)$, $ka=k_1a_1+k_2a_2$, $kb=k_1b_1+k_2b_2$, direct computations give
\begin{align*}
&g_0(\alpha,t)g_1(a,b,c)^{k_1}g_2(a,b,c)^{k_2}g_3(a,b,r)^{s_0}
\bpm
w\\
z	
\epm=\\
&\begin{pmatrix}
\alpha(w+ka)\\
z+(kb) w+\frac{k_{1}(k_{1}-1)}{2}a_1b_1+k_{1}k_{2}b_1a_2+\frac{k_{2}(k_{2}-1)}{2}a_2b_2+kc+t+\frac{s_0}{r}b\wedge a	
\end{pmatrix}.
\end{align*}

Taking into account the explicit formula for $g_0(\alpha,t')$, we obtain $ka=kb=0$, so $k=0$. Therefore we have
\begin{equation}\label{g0->g0'}
g_0(\alpha,t')=	g_0(\alpha,t)g_3(a,b,r)^{s_0}.
\end{equation}
By Lemma \ref{explic-C-formula} we see that  formulae (\ref{introd-G1G2}), (\ref{g3'g3}) give
$$
G_i=g_i(A,B,C) \hbox{ for }1\leq i\leq 2,\ G_3=g_3(A,B,r)
$$
where $(A,B,C)=K\cdot (a,b,c)$. On the other hand (\ref{gi(a',b',c')}) can be written as
\begin{equation}\label{Gi->gi(a',b',c')}
\begin{split}	
g_1(a',b',c')=&g_1(A,B,C)g_3(A,B,r)^{\det(K)s_1},\\
 g_2(a',b',c')=&g_2(A,B,C) g_3(A,B,r)^{\det(K) s_2},
 \end{split}
 \end{equation}
which, by Remark \ref{equiv-for-c'-t'}, gives $c'=C+  \frac{B\wedge A}{r} \det(K) s\in [C]_{AB}$. Finally, formula (\ref{g0->g0'}) shows that $t'=t+ \frac{B\wedge A}{r} \det(K) s_0\in [t]_{AB}$.

Therefore,  $a'=A=Ka$, $b'=B=Kb$ and $[c']_{AB}=[C]_{AB}$, $[t']_{AB}=[t]_{AB}$ which, taking into account that $K\cdot (a,b,c,t)=(A,B,C,t)$,  proves the claim. 
\vspace{2mm}\\
$(2)\Rightarrow (1)$:\\
It suffices to note that, if $(a',b',[c']_{a'b'},[t']_{a'b'})=K\cdot (a,b,[c]_{ab},[t]_{ab})$ with $K\in\GL(2,\Z)$, then $G(\alpha,a',b',c',r,t')\subset G(\alpha,a,b,c,r,t)$. 	
\end{proof}

Therefore, by Theorem \ref{Quotient-approx}, the quotient set ${\cal A}_{\alpha,r}/\approx$ (which parameterises  the subgroups $G\subset \Aff(U)$ defining  type II Inoue surfaces) is identified with the orbit space $ {\mathscr{A}_{\alpha,r}}/{\GL(2,\Z)}$. This orbit space comes with a natural surjective map
$$
{\mathscr{A}_{\alpha,r}}/{\GL(2,\Z)}\to {\cal P}_\alpha/\GL(2,\Z)\to   {\cal N}_\alpha/\GL(2,\Z)
$$
onto the set of $\GL(2,\Z)$-similarity classes of matrices $N\in{\cal N}_\alpha$ (see section \ref{Intro-typeII}).

\subsubsection{A classification theorem for type II Inoue surfaces}

Let $(\alpha,r)\in]1,+\infty[\times \N^*$ with $\alpha$ $S^+$-admissible. We now come back to our initial problem, which is:  
describe explicitly the quotient of ${\cal A}_{\alpha,r}\edf{\cal T}_{\alpha,r}	\times\C$ by the equivalence relation
$$
(a,b,c,t)\sim (a',b',c',t') \hbox{ if $G(\alpha,a,b,c,r,t)$, $G(\alpha,a',b',c',r,t')$ are conj. in $\Aff(U)$}.$$

 Elementary computations give the following:
\begin{lm}\label{tau(gi)tau(-1)}
Let $(a,b,c,t)$, $(a',b',c',t')\in {\cal A}_{\alpha,r}$ and put 
$$
\left\{
\begin{array}{lll}
g_0=g_0(\alpha,t),& g_i=g_i(a,b,c),&g_3=g_3(a,b,r),\\
g_0'=g_0(\alpha,t'),& g_i'=g_i(a',b',c'),& g_3'=g_3(a',b',r).
\end{array}\right.$$
 Let $k=(k_1,k_2)\in\Z^2$, $\bpm s_0\\s_1\\s_2\epm \in\Z^3$ and $\tau\in \Aff^1_1(U)$,   
$
\begin{pmatrix}
w\\
z 	
\end{pmatrix}\textmap{\tau} \begin{pmatrix}
1  &0\\
\lambda	&1
\end{pmatrix}\begin{pmatrix}
w\\
z 	
\end{pmatrix}+\begin{pmatrix}
u\\ \zeta	
\end{pmatrix}$. 
The following conditions are equivalent:\vspace{2mm}
\begin{enumerate}[(i)]
	\item 
$
\tau g'_0\tau^{-1}= g_0g_1^{k_1}g_2^{k_2}g_3^{s_0},\ \tau\circ  g_i'\circ \tau^{-1}=g_ig_3^{s_i} \hbox{ for }1\leq i\leq 2,\ \tau\circ  g_3'\circ \tau^{-1}=g_3$. \vspace{3mm}
\item We have: 
$$
\left\{\begin{array}{ccl}
\lambda&=&\frac{kb}{\alpha-1}.\\
u&=&\frac{\alpha (k a)}{1-\alpha}.\\
a'&=&a.\\
b'&=&b.\\
c'&=&c-\frac{\alpha(ka)}{\alpha-1}b-\frac{ kb }{\alpha-1}a+\frac{b\wedge a}{r}s,\hbox{ where }s\edf \bpm s_1\\s_2\epm.\\
t'&=&t+\frac{\alpha(ka)(kb)}{1-\alpha}	+ kc +  \frac{k_{1}(k_{1}-1)a_1b_1+k_{2}(k_{2}-1)a_2b_2}{2}+k_{1}k_{2}b_1a_2 +  \frac{s_0}{r}b\wedge a.
\end{array}
 \right.
$$
\end{enumerate}

\end{lm}

  \begin{lm}\label{k.c-k.t}  \begin{enumerate}
 \item  Putting  
\begin{align*}
 k\cdot c\edf & c-\frac{\alpha(ka)}{\alpha-1}b-\frac{ kb }{\alpha-1}a,\\	
 k \cdot_c t\edf & t+\frac{\alpha(ka)(kb)}{1-\alpha}	+ kc +  \frac{k_{1}(k_{1}-1)a_1b_1+k_{2}(k_{2}-1)a_2b_2}{2}+k_{1}k_{2}b_1a_2,
\end{align*}
we have the identities
\begin{align}
 k\cdot c= & c+ b\wedge a (I_2-N)^{-1}\bpm -k_2\\ k_1\epm, \label{k.c}\\	
 k\cdot_c t= & t +b\wedge a\bigg(k (I_2-N)^{-1}\bigg(\frac{1}{2}\bpm n_{11}n_{12}-k_2\\n_{21}n_{22}+k_1\epm+\frac{1}{r}\pi_{a,b,r}(c) \bigg)  +\frac{k_1k_2}{2}\bigg) \label{k.t}.
\end{align}

\item  Put  $\delta\edf  \det(I_2-N) =2-\theta$. If $k\in \delta\Z^2$, then 
\begin{align}
 k\ (I_2-N)^{-1} \bpm n_{11}n_{12}-k_2\\n_{21}n_{22}+k_1\epm + k_1k_2 &\in 2\Z,\label{in 2Z}\\	
 k (I_2-N)^{-1}\bigg(\frac{1}{2}\bpm n_{11}n_{12}-k_2\\n_{21}n_{22}+k_1\epm+\frac{1}{r}\Z^2 \bigg)  +\frac{k_1k_2}{2}&\subset \frac{1}{r}\Z. \label{in 1/r Z}
\end{align}

\end{enumerate}
\end{lm}

\begin{proof} (1) For the first formula it suffices to apply $I_2-N$ to $k\cdot c$ taking into account that $(I_2-N)a=(1-\alpha)a$, $(I_2-N)b=(1-\alpha^{-1})b$.

For the second formula in (1), we have
$$\frac{\alpha(ka)(kb)}{1-\alpha}	+ kc +  \frac{k_{1}(k_{1}-1)a_1b_1+k_{2}(k_{2}-1)a_2b_2}{2}+k_{1}k_{2}b_1a_2=
$$
$$=\frac{k_{1}(k_{1}-1)}{2}a_1b_1+k_{1}k_{2}b_1a_2+\frac{k_{2}(k_{2}-1)}{2}a_2b_2+kc-\frac{\alpha}{\alpha-1}(ka)(kb)=
$$
$$
=\frac{1}{2}(k_1^2a_1b_1+ k_2^2a_2b_2+k_1k_2(a_1b_2+a_2b_1)-k_1k_2(a_1b_2+a_2b_1)-k_1 a_1b_1-k_2a_2b_2)+
$$
$$
+\frac{1}{2} 2k_{1}k_{2}b_1a_2+kc-\frac{\alpha}{\alpha-1}(ka)(kb)=
$$
$$
=\frac{1}{2}\big((ka)(kb)+k_1k_2 b\wedge a\big)+ kc -\frac{\alpha}{\alpha-1}(ka)(kb)-k_1 a_1b_1-k_2a_2b_2=
$$
$$
=kc+\frac{1}{2}(k_1k_2 b\wedge a -k_1 a_1b_1-k_2a_2b_2) -\frac{\alpha+1}{2(\alpha-1)}(ka)(kb)=
$$
$$
k\bigg( c-\frac{1}{2} \bpm a_1b_1\\ a_2 b_2\epm\bigg) + \frac{1}{2}k_1k_2 b\wedge a -\frac{\alpha+1}{2(\alpha-1)}(ka)(kb).
$$

On the other hand,  applying Lemma \ref{Lemma-Cramer} to $(I_2-N)^{-1}$ and noting that
$$
(I_2-N)^{-1}(a)=\frac{1}{1-\alpha}a,\ (I_2-N)^{-1}(b)=\frac{\alpha}{\alpha-1}b,
$$
we obtain:
$$ -\frac{\alpha+1}{\alpha-1}(ka) (kb)=(b\wedge a) \bigg(  k (I_2-N)^{-1}\bpm -k_2\\k_1 \epm\bigg).$$
Therefore
$$\frac{\alpha(ka)(kb)}{1-\alpha}	+ kc +  \frac{k_{1}(k_{1}-1)a_1b_1+k_{2}(k_{2}-1)a_2b_2}{2}+k_{1}k_{2}b_1a_2=
$$
$$
=k\bigg( c-\frac{1}{2} \bpm a_1b_1\\ a_2 b_2\epm\bigg) + \frac{k_1k_2}{2} b\wedge a +\frac{b\wedge a}{2}  \bigg(  k (I_2-N)^{-1}\bpm -k_2\\k_1 \epm\bigg). $$

It suffices to recall that:
$$
 c-\frac{1}{2} \bpm a_1b_1\\ a_2 b_2\epm=(I_2-N)^{-1}\bigg (\frac{b\wedge a}{2}\bpm n_{11}n_{12}\\ n_{21}n_{22}\epm+ \frac{b\wedge a}{r}\pi_{a,b,r}(c)\bigg)
 $$
 by the definition of $p(a,b,c,r)=\pi_{a,b,r}(c)$.
\\ \\
(2) Note first that (\ref{in 1/r Z}) follows easily from (\ref{in 2Z}) by noting that $\delta (I_2-N)^{-1}$ is a matrix with integer entries. Therefore it's enough to prove (\ref{in 2Z}). Let $l\in \Z^2$ such that $k=\delta l$. The claim follows by elementary computations using the formulae:
$$
k (I_2-N)^{-1}=\delta l(I_2-N)^{-1}=l\, \trcof(I_2-N)= l\bpm 1-n_{22} & n_{12}\\ n_{21} & 1-n_{11}\epm, 
$$
$$
 l_1l_2 (n_{22}-n_{11}+\delta)=l_1l_2 (n_{22}-n_{11}+2-(n_{11}+n_{22}))=2l_1l_2(1-n_{11})\in 2\Z,
$$
$$
l_i^2\equiv l_i \  {\rm mod}\  2, \ \delta\equiv\theta\equiv n_{11}+n_{22} \  {\rm mod}\  2,
$$
$$
 n_{12}n_{22}(1+n_{21}-n_{11}),\ n_{21}n_{11}(-1+n_{12}-n_{22})\in 2\Z. 
$$
The latter formula follows using $\det(N)=1$.

\end{proof}

\begin{lm}\label{Lemma-Cramer}
Let $M\in M_2(\C)$ and $a=\bpm a_1\\ a_2\epm $, $b=\bpm b_1\\ b_2\epm$ be linearly independent in $\C^2$ such that $Ma=\alpha a$, $Mb=\beta b$ and $x=(x_1,x_2)\in\C^2$. Then
$$
(\beta-\alpha)( x a) (x b)=(a\wedge b)  x M \bpm -x_2\\x_1\epm. 
$$
	
\end{lm}

\begin{proof}
We apply the Cramer rule to the system
$$
\left\{\begin{array}{ccc}
a_1 x_1+a_2 x_2&=& xa\\
b_1 x_1+b_2 x_2&=& x b 	
\end{array}\right..
$$
in order to express $x_1$, $x_2$ in terms of $xa$ and $xb$.
\end{proof}

\begin{pr} \label{Z2-action} Let $(a,b)\in {\cal P}_{\alpha}$. 
\begin{enumerate}
\item 	The formula $
(k,c)  \mapsto k\cdot c$ defines an action of $\Z^2$ on ${\cal C}_{a,b,r}$ satisfying the identity
\begin{equation}\label{pi(kc)}
\pi_{a,b,r}( k\cdot c)=\pi_{a,b,r}(c)+r \bpm -k_2\\ k_1\epm.
\end{equation}
\item The formula $
 \bigg (k,\bpm  [c]_{ab}\\ [t]_{ab}\epm\bigg )  \mapsto \bpm [k\cdot  c ]_{ab},	\\
  \big[k \cdot_c t \big ]_{ab}\epm$
defines an action  of $\Z^2$ on the product $\mathscr{C}_{a,b,r}\times \big(\C/\frac{b\wedge a}{r}\Z\big)$. This action descends to an action of the finite group $\Z^2/ d\Z^2=(\Z_d)^2$, where $d\edf |\det(I_2-N)|=\theta-2$.

\end{enumerate}

\end{pr}
\begin{proof}
(1) This follows from formula (\ref{k.c}).%
\\  \
(2) Note first that the pair $([k\cdot  c ]_{ab},	 
  \big[k \cdot_c t \big ]_{ab})$ depends only on the classes $[c]_{ab}$, $ [t]_{ab}$ and $k$, so the right hand term is well defined.  To prove that the given formula defines an action, put:
$$
 \cg\edf k\cdot c,\ \tg\edf  k\cdot_ct,\ \Cg\edf l\cdot \cg,\  \Tg \edf  l\cdot_\cg \tg,\ \Cg'\edf (k+l) \cdot c,\ \Tg' \edf  (k+l)\cdot_c  t.
$$
By (1) we have  $\Cg'=\Cg$, and, using the definition of $k\cdot_c t$, elementary computations give
\begin{equation}\label{Tg'}
\Tg'=\Tg + k_2l_1 b\wedge a\in \Tg+  \frac{b\wedge a}{r}\Z,
\end{equation}
which proves the claim.	For the second claim in (2), note that,  by formulae (\ref{k.c}), (\ref{k.t}),  (\ref{in 1/r Z}) stated in Lemma \ref{k.c-k.t},    we have 
$$k\cdot c-c\in   (b\wedge a) \Z^2\subset\frac{b\wedge a}{r}\Z^2 ,\   k\cdot_c t-t\in \frac{b\wedge a}{r}\Z$$
  for any $k\in d \Z^2$, so $d \Z^2$ acts trivially on $\mathscr{C}_{a,b,r}\times \big(\C/\frac{b\wedge a}{r}\Z\big)$.
\end{proof}
For a triple $(a,b,c)\in {\cal T}_{\alpha,r}$ and $t\in\C$ we will also use the notation
$$
k\cdot(a,b,c,t)=(a,b,\cg,\tg) \hbox{ where }(\cg,\tg)\edf k\cdot (c,t).
$$
Formula (\ref{Tg'}) shows that the formulae 
$$(k, (c , t ))\mapsto  k\cdot (c, t),\ (k, (a,b,c , t ))=k\cdot(a,b,c,t)$$
 {\it do  not} define   $\Z^2$-actions on the respective spaces. However, Proposition \ref{Z2-action} shows that the formulae
 $$
 (k, ([c]_{ab} , [t]_{ab} ))\mapsto k\cdot  ([c]_{ab} , [t]_{ab} )\edf   ([\cg]_{ab},[\tg]_{ab}),$$
 $$ (k, (a,b,[c]_{ab} , [t]_{ab} ))\mapsto k\cdot(a,b,[c]_{ab} , [t]_{ab} )\edf (a,b,[\cg]_{ab},[\tg]_{ab})
 $$
 define $\Z^2$-actions which descend to $\Z_d^2$-actions on the spaces $\mathscr{C}_{a,b,r}\times \big(\C/\frac{b\wedge a}{r}\Z\big)$, $\mathscr{A}_{\alpha,r}$ respectively.

On the latter space $\mathscr{A}_{\alpha,r}$ we have now defined a  $\GL(2,\Z)$-action and a $\Z^2$-action (which induces a $\Z_d^2$-action). These actions {\it do not commute}. We have
\begin{re}\label{K(k)versus k'(K)}
Let $(a,b,c)\in {\cal T}_{\alpha,r}$, $t\in \C$, $k=(k_1,k_2)\in\Z^2$, $K=(k_{ij})\in\GL(2,\Z)$. 	Then the 4-tuples $K\cdot(k\cdot(a,b,c,t))$, $(kK^{-1})\cdot (K\cdot (a,b,c,t))$ define the same elements in $\mathscr{A}_{\alpha,r}$. More precisely
\begin{enumerate}
\item 	We have the identity
$$
K\cdot (k\cdot c)=(k K^{-1})\cdot (K\cdot c).$$
\item Putting 
$$(\cg,\tg)\edf k\cdot (c,t), \ (A,B,C)=K\cdot (a,b,c),\ (A,B,\Cg)=K\cdot (a,b,\cg)
$$
we have
$$
([\Cg]_{AB},[\tg]_{AB})=(kK^{-1})\cdot ([C]_{AB},[t]_{AB}).
$$

\end{enumerate}
\end{re}
\begin{proof}
(1) Follows by elementary computations. 
\vspace{2mm}\\  
(2) Put $k'\edf k K^{-1}$.  Note that, by (1), the first component of $k'\cdot (C,t)$ is $\Cg$.
Define $\tg'\in\C$ by   $k'\cdot (C,t)=(\Cg,\tg')$. We have to prove that $\tg$, $\tg'$ are congruent mod $\frac{B\wedge A}{r}$. Elementary computations give:  
\begin{equation}
\begin{split}
\tg'-\tg=&\frac{b\wedge a}{2}\left(k'\bpm k_{11}k_{12}\\ k_{21}k_{22}\epm+ k'_1k'_2\det K-k_1k_2\right)\\
=& \frac{b\wedge a}{2}\big( (k_1'-k_1'^2)k_{11}k_{12}+ (k_2'-k_2'^2)k_{21}k_{22}-2 k'_2k'_1k_{12}k_{21}\big)\\  \in&  \frac{B\wedge A}{r} r \Z \subset \frac{B\wedge A}{r}   \Z  . 
\end{split} 	 
\end{equation}
\end{proof}
Proposition \ref {Z2-action} and  Remark \ref{K(k)versus k'(K)} show  that:
\begin{pr}  Let $\varphi:\GL(2,\Z)\to\Aut(\Z^2)$, $\psi:\GL(2,\Z)\to\Aut(\Z^2_d)$ be the group morphisms defined by
$
\varphi(K)(k)=kK^{-1}$, $\psi(K)([k])=[k]K^{-1}$.
The formula 
$$
(K,k)\cdot (a,b,[c]_{ab},[t]_{ab})\edf K\cdot(k\cdot(a,b,[c]_{ab},[t]_{ab})),
$$	
defines an action of the semi-direct product\footnote{In general, for groups $N$, $H$ and a group morphism $\varphi:H\to \Aut(N)$,  the semi-direct product $H\ltimes_\varphi N$ is   $H\times N$ endowed with the composition law $(h,n)*(h',n')\edf (hh',\varphi(h'^{-1})(n)n')$.} 
$
\GL(2,\Z)\ltimes_\varphi \Z^2
$
on the space $\mathscr{A}_{\alpha,r}$ which descends to an action of $\GL(2,\Z)\ltimes_\psi  \Z_d^2$. \end{pr}

Now note that the action  
$$
(\mu,\rho)\cdot (a,b,c,t)\edf (\mu a,\rho b,\mu\rho c, \mu\rho t)
$$
of the product $\R^*_+\times\R^*$ on ${\cal A}_{\alpha,r}$ induces an $\R^*_+\times\R^*$-action on $\mathscr{A}_{\alpha,r}$ which commutes with the $(\GL(2,\Z)\ltimes_\varphi \Z^2)$-action defined above. Therefore we obtain an action of the product group
$$
(\R^*_+\times\R^*)\times (\GL(2,\Z)\ltimes_\varphi \Z^2)
$$
on $\mathscr{A}_{\alpha,r}$  given explicitly by
$$
((\mu,\rho),(K,k))\cdot (a,b,[c]_{ab},[t]_{ab})\edf (K,k)\cdot (\mu a,\rho b,[\mu\rho c]_{ab},[\mu \rho t]_{ab}),
$$
which descends to an action of $(\R^*_+\times\R^*)\times (\GL(2,\Z)\ltimes_\psi \Z^2_d$.

Our main theorem is
\begin{thry}\label{main} Let $(a,b,c,t)$, $(a',b',c',t')\in {\cal A}_{\alpha,r}$. 
The groups $G(\alpha,a,b,c,r,t)$, $G(\alpha,a',b',c',r,t')$ are conjugate in $\Aff(U)$ if and only if 
\begin{equation}
\begin{split}
(a',b',[c']_{a'b'},[t']_{a'b'})\in & \big((\R^*_+\times\R^*)\times (\GL(2,\Z)\ltimes_\psi \Z_d^2)\big)\cdot (a,b,[c]_{ab},[t]_{ab}).	\\
=& \big((\R^*_+\times\R^*)\times (\GL(2,\Z)\ltimes_\varphi \Z^2)\big)\cdot (a,b,[c]_{ab},[t]_{ab}).
\end{split}	
\end{equation}
\end{thry}
\begin{proof}

Put $G\edf G(\alpha,a,b,c,r,t)$, $G'\edf G(\alpha,a',b',c',r,t')$,
$$
\left\{
\begin{array}{lll}
g_0=g_0(\alpha,t),& g_i=g_i(a,b,c),&g_3=g_3(a,b,r),\\
g_0'=g_0(\alpha,t'),& g_i'=g_i(a',b',c'),& g_3'=g_3(a',b',r).
\end{array}\right. $$
First, suppose $G$, $G'$ are conjugate in $\Aff(U)$, and let $\theta\in\Aff(U)$,
$$
\theta\begin{pmatrix}
w\\z	
\end{pmatrix}=\begin{pmatrix}
\mu &0\\
\lambda&\nu	
\end{pmatrix}\begin{pmatrix}
w\\z 	
\end{pmatrix}+ \begin{pmatrix}
u\\ \zeta 	
\end{pmatrix}
$$
(with  $\mu\in \R^*_+$, $\nu\in\C^*$, $u\in\R$, $\lambda$, $\zeta\in\C$) be such that 
$
\theta G' \theta^{-1}=G$.

It follows that, putting, $\gamma_i\edf \theta g'_i\theta^{-1}$ for $0\leq i\leq 3$, $(\gamma_0,\gamma_1,\gamma_2,\gamma_3)$ is a system of generators of $G$. Note that this system has the properties:
\begin{enumerate}[(C1)]
\item $\gamma_3$ is a generator of the cyclic group	$G\cap\T(U)$, in other words $\gamma_3\in\{g_3,g_3^{-1}\}$.
\item $(\gamma_1,\gamma_2,\gamma_3)$ is a system of generators of 	 	$G(a,b,c,r)=G \cap\Aff_1^1(U)$ with the property $[\gamma_1,\gamma_2]=\gamma_3^r$. 
\item The first diagonal element of the linear part of $\gamma_0$ is $\alpha$. \end{enumerate}

By Lemma \ref{SetOfCanonical} proved below, it follows that there exists $K\in\GL(2,\Z)$,  $k=(k_1,k_2)\in\Z^2$ and $\bpm s_0\\s_1\\s_2 \epm \in\Z^3$ such that
\begin{equation}\label{gammai-gi}
\gamma_0=g_0g_1^{k_1}g_2^{k_2} g_3^{s_0},\ \gamma_1=g_1^{k_{11}}g_2^{k_{12}}g_3^{s_1},\ \gamma_2=g_1^{k_{21}}g_2^{k_{22}}g_3^{s_2},\ \gamma_3=g_3^{\det(K)}.	
\end{equation}

Identifying the (2,1)-entries of the matrices associated with the linear parts of the two sides in the equalities $\theta g_i'\theta^{-1}= g_1^{k_{i1}}g_2^{k_{i2}}g_3^{s_i}$ for $1\leq i\leq 2$, we obtain $\nu\mu^{-1} b'_i\in\R$, so  $\nu\mu^{-1}\in\R$, i.e. $\nu\in\C^*\cap\R=\R^*$.

  Put 
$$G_0\edf g_0,\ G_1\edf g_1^{k_{11}}g_2^{k_{12}},\ G_2\edf g_1^{k_{21}}g_2^{k_{22}},\ G_3\edf g_3^{\det(K)}.$$
By Lemma \ref{SetOfCanonical}  $(G_0,G_1,G_2,G_3)$ is a system of generators of $G$ and, by Lemma \ref{explic-C-formula}, we have
$$
G_0=g_0(\alpha,t),\ G_i= g_i(A,B,C),\ G_3=g_3(A,B,r),
$$
where $(A,B,C)=K\cdot (a,b,c)$.  

On the other hand we can write $\theta=\tau\circ \delta$ where
$$
\tau\begin{pmatrix}
w\\z	
\end{pmatrix}=\begin{pmatrix}
1 &0\\
\lambda\mu^{-1}& 1	
\end{pmatrix}\begin{pmatrix}
w\\z 	
\end{pmatrix}+ \begin{pmatrix}
u\\ \zeta 	
\end{pmatrix},\ \delta\begin{pmatrix}
w\\z	
\end{pmatrix}=\begin{pmatrix}
\mu w\\ \nu z
\end{pmatrix}.
$$
Note that $\tau\in \Aff^1_1(U)$. We have  
\begin{equation} 
\gamma_i=\tau g_i''\tau^{-1}	 \hbox{ for }0\leq i\leq 3
\end{equation}
where $g_i''\edf \delta g'_i\delta^{-1}$. Direct computations show that  
$$
g_0''=g_0(\alpha,t''), \ g_i''=g_i(a'',b'',c'')  \hbox{ for }1\leq i\leq 2,\ g_3''=g_3(a'',b'',r),
$$
where $a''=\mu a'$, $b''=\nu\mu^{-1} b'$, $c''=\nu c'$, $t''=\nu t'$. Since $\nu\mu^{-1}\in\R^*$,   we can write
\begin{equation}\label{a''b''c''t''}
(a'',b'',c'',t'')=(\mu,\nu\mu^{-1})\cdot (a',b',c',t').	
\end{equation}
Taking into account (\ref{gammai-gi}) we obtain
\begin{equation}\label{tau-gi-tau{-1}}
\begin{split}
\tau g_0''\tau^{-1}&=G_0g_1^{k_1}g_2^{k_2} g_3^{s_0},\ \tau g_1''\tau^{-1}=G_1G_3^{s_1\det(K)},\\ 
\tau g_2''\tau^{-1}&=G_2G_3^{s_2\det(K)},\ \tau g_3''\tau^{-1}=G_3.	
\end{split}	
\end{equation}
Since $g_1$ and $g_2$ commute modulo $\langle g_3^r\rangle$ (see Remark \ref{com-rel-II}) it follows that, putting $L\edf K^{-1}$, $l=(l_1,l_2)\edf k L$, we have
$$
g_1=G_1^{l_{11}}G_2^{l_{12}}g_3^{rm_1},\ g_2=G_1^{l_{21}}G_2^{l_{22}}g_3^{rm_2},\ g_1^{k_1}g_2^{k_2}=G_1^{l_1}G_2^{l_2} g_3^{rm} \hbox{ with }m_1,\ m_2,\ m\in\Z,
$$
and the formula $\tau g_0''\tau^{-1}=G_0g_1^{k_1}g_2^{k_2} g_3^{s_0}$ can be written as 
$$\tau g_0''\tau^{-1}=G_0G_1^{l_1}G_2^{l_2} g_3^{s_0+rm}.$$
 Therefore (\ref{tau-gi-tau{-1}}) becomes
\begin{equation}
\begin{split}
\tau g_0''\tau^{-1}&=G_0G_1^{l_1}G_2^{l_2} G_3^{(s_0+rm)\det(K)},\ \tau g_1''\tau^{-1}=G_1G_3^{s_1\det(K)},\\ \tau g_2''\tau^{-1}&=G_2G_3^{s_2\det(K)},\ \tau g_3''\tau^{-1}=G_3.	
\end{split}
\end{equation}
By Lemma \ref{tau(gi)tau(-1)}, it follows that
$$
a''=A,\ b''=B,\ ([c'']_{a''b''},\ [t'']_{a''b''})=l\cdot([C]_{AB},t_{AB})
$$
so, by  (\ref{a''b''c''t''}), 
$$
(\mu,\nu\mu^{-1})\cdot (a',b',[c']_{a'b'},[t']_{a'b'})=l\cdot (A,B,[C]_{AB},[t]_{AB}),
$$
i.e.
$$
(\mu,\nu\mu^{-1})\cdot (a',b',[c']_{a'b'},[t']_{a'b'})=l\cdot(K\cdot(a,b,[c]_{ab},[t]_{ab}),
$$
which completes the proof of the first implication.\\

For the converse, note first that, if
$$
(a',b',[c']_{a'b'},[t']_{a'b'})\in \big((\R^*_+\times\R^*)\times (\GL(2,\Z)\ltimes_\varphi \Z^2)\big)\cdot (a,b,[c]_{ab},[t]_{ab}),
$$
then there exists $(\mu,\rho)\in\R^*_+\times\R^*$, and $(K,k)\in(\GL(2,\Z)\ltimes_\varphi \Z^2)$  such that  
\begin{equation}\label{(mu,rho).(a',b',[c'],[t'])}
(\mu,\rho)\cdot(a',b',[c']_{a'b'},[t']_{a'b'}) =(K,k)\cdot (a,b,[c]_{ab},[t]_{ab}).
\end{equation}
We have to show that there exists $\theta\in \Aff(U)$ such that 
$$
\theta G(\alpha,a',b',c',r,t')\theta^{-1}= G(\alpha,a,b,c,r,t).$$
Put  $g_0''\edf g_0(\alpha,t'')$,  $g_i''\edf g_i(a'',b'',c'')$ for  $1\leq i\leq 2$,  $g_3''\edf g_3(a'',b'',r)$, where 
$$(a'',b'',c'',t'')\edf (\mu,\rho)\cdot(a',b',c',t').$$
An easy computation shows that   
\begin{equation}\label{g_i''-g_i'}
g_i''=\delta g_i'\delta^{-1} \hbox{ for }0\leq i\leq 3,
 \end{equation}
 where $\delta\in\Aff(U)$ is given by $\delta(w,z)=(\mu w, \mu\rho z)$.

On the other hand, by (\ref{(mu,rho).(a',b',[c'],[t'])}) we have   
\begin{equation}\label{a''b''}
\begin{split}
(a'',b'',[c'']_{a''b''},[t'']_{a''b''})=&   kK^{-1}\cdot(K\cdot(a,b,[c]_{ab},[t]_{ab})= \\  = &(K,k)\cdot (a,b,[c]_{ab},[t]_{ab})= l\cdot(A,B,[C]_{AB},[t]_{AB}),
\end{split}
\end{equation}
where $l\edf kK^{-1}=(l_1,l_2)$. 
Formula (\ref{a''b''}) shows that there exists $s_0\in\Z $, $s= \bpm s_1\\s_2\epm\in\Z^2$ such that
$$
\left\{\begin{array}{ccl}
a''&=&A.\\
b''&=&B.\\
c''&=&C-\frac{\alpha(l A)}{\alpha-1}B-\frac{ lB }{\alpha-1}A+\frac{A\wedge B}{r}s.\\
t''&=&t+\frac{\alpha(l A)(l B)}{1-\alpha}	+ l C +  \frac{l_{1}(l_{1}-1)A_1B_1+l_{2}(l_{2}-1)A_2B_2}{2}+l_{1}l_{2}B_1A_2 +  \frac{s_0}{r}B\wedge A.
\end{array}
 \right.
$$
Let $\tau\in \Aff^1_1(U)$ be given by
$$
\begin{pmatrix}
w\\
z 	
\end{pmatrix}\textmap{\tau} \begin{pmatrix}
1  &0\\
\lambda	&1
\end{pmatrix}\begin{pmatrix}
w\\
z 	
\end{pmatrix}+\begin{pmatrix}
u\\ 0	
\end{pmatrix},\ \lambda\edf \frac{lB}{\alpha-1},\ u\edf \frac{\alpha (l A)}{1-\alpha}.$$
Using the implication $ii)\Rightarrow i)$ of  Lemma \ref{tau(gi)tau(-1)}, it follows that:
\begin{equation}\label{Gi->tau-g''i-tau{-1}}
\begin{split}
\tau g_0''\tau^{-1}=&g_0(\alpha,t)g_1(A,B,C)^{l_1}g_2(A,B,C)^{l_2}g_3(A,B,r)^{s_0},\\	
\tau g_i''\tau^{-1}=&g_i(A,B,C)g_3(A,B,r)^{s_i} \hbox{ for } 1\leq i\leq 2,\\
\tau g_3''\tau^{-1}=&g_3(A,B,r).
\end{split}	
\end{equation}
Combining (\ref{g_i''-g_i'}) and (\ref{Gi->tau-g''i-tau{-1}})  and putting $\theta=\tau \circ \delta$, we obtain
\begin{equation*} 
\begin{split} 	
\theta g_0'\theta^{-1}=&g_0(\alpha,t)g_1(A,B,C)^{l_1}g_2(A,B,C)^{l_2}g_3(A,B,r)^{s_0},\\
\theta g_i'\theta^{-1}=&g_i(A,B,C)g_3(A,B,r)^{s_i}, 1\leq i\leq 2,\\
\theta g_3'\theta^{-1}=&g_3(A,B,r).
\end{split} 
\end{equation*}
This shows that $\theta G(\alpha, a',b',c',r,t')\theta^{-1}= G(\alpha, A,B,C,r,t)$. On the other hand, since $(A,B,[C]_{AB},[t]_{AB})=K\cdot(a,b,[c]_{ab},[t]_{ab})$, it follows by Theorem \ref{Quotient-approx} that    $G(\alpha, A,B,C,r,t)=G(\alpha,a,b,c,r,t)$. 

Therefore $\theta G(\alpha, a',b',c',r,t')\theta^{-1}=G(\alpha,a,b,c,r,t)$, which completes the proof.
 
\end{proof}

\begin{lm}\label{SetOfCanonical}
The set of system of generators $(\gamma_0,\gamma_1,\gamma_2,\gamma_3)$ of $G$ satisfying the properties (C1), (C2), (C3) stated in the proof of Theorem \ref{main} is
$$
\bigg\{\big(g_0g_1^{k_1}g_2^{k_2} g_3^{s_0}, g_1^{k_{11}}g_2^{k_{12}}g_3^{s_1},g_1^{k_{21}}g_2^{k_{22}}g_3^{s_2},g_3^{\det(K)}\big)|\, K\in\GL(2,\Z),\, k\in\Z^2,\, \bpm s_0\\s_1\\s_2 \epm \in\Z^3\bigg\}.  	
$$
\end{lm}

\begin{proof}
This follows using Remark \ref{gen-form-g-in-G(a,b,c,r-t)} by elementary computations.	
\end{proof}

Taking into account Remark \ref{ClassRem(alpha,r)}, Theorem \ref{main} gives the following geometric interpretation of the set of biholomorphism classes of type II Inoue surfaces associated with a pair $(\alpha,r)$:
\begin{thry}\label{Classif-Theorem-II}
Let $(\alpha,r)\in]1,+\infty[\times \N^*$ with $\alpha$ $S^+$-admissible. The set of biholomorphism 	classes of type II Inoue surfaces associated with $(\alpha,r)$ is naturally  identified with the quotient space
$$
\mathscr{Q}_{\alpha,r}\edf \qmod{\mathscr{A}_{\alpha,r}}{(\R^*_+\times\R^*)\times (\GL(2,\Z)\ltimes_\psi \Z^2_d)}.
$$
\end{thry}

\subsubsection{The fibre over a similarity class of $\SL(2,\Z)$-matrices } \label{fibre-over-simclass}

In section \ref{Intro-typeII} we have put ${\cal N}_\alpha\edf \{N\in \SL(2,\Z)| \ \alpha\in \Spec(N)\}$. Consider the {\it surjective} maps
$$
\begin{array}{r|l}
\begin{tikzcd}
\mathscr{A}_{\alpha,r}\ar[rr, "\psi_{\alpha,r}" ]\ar[rd, "\phi_{\alpha,r}"']&&{\cal N}_\alpha 	\\ 
&{\cal P}_\alpha \ar[ru, "\eta_\alpha"'] & 
\end{tikzcd}
& 
\begin{array}{lll}
\phi_{\alpha,r}(a,b,[c]_{ab},[t]_{ab})&\edf&(a,b),	\\
\eta_{\alpha}(a,b)&\edf&N(\alpha,a,b),\\
\psi_{\alpha,r}&\edf &  \eta_\alpha\circ \phi_{\alpha,r},
\end{array}
\end{array}
$$
and let the group
$${\cal G}\edf (\R^*_+\times\R^*)\times (\GL(2,\Z)\ltimes_\psi \Z_d^2) $$
act on the spaces ${\cal P}_\alpha$, ${\cal N}_\alpha$ by 
$$((\mu,\rho),(K,[k]))\cdot(a,b)\edf (\mu Ka,\rho Kb),\ ((\mu,\rho),(K,[k]))\cdot N\edf KNK^{-1}.$$
The maps $\phi_{\alpha,r}$, $\eta_{\alpha}$, $\psi_{\alpha,r}$ are obviously ${\cal G}$-equivariant, so the they induce maps
$$
\begin{tikzcd}
\mathscr{Q}_{\alpha,r}\edf \mathscr{A}_{\alpha,r}/{\cal G}\ar[rr, "\Psi_{\alpha,r}" ]\ar[rd, "\Phi_{\alpha,r}"']&&{\cal N}_\alpha/{\cal G}\ar[r, equal]&[-16pt]{\cal N}_\alpha/\GL(2,\Z)={\cal N}_\alpha/\sim	\\ 
&{\cal P}_\alpha/{\cal G} \ar[ru, "\simeq", "H_\alpha"']\ar[rr, equal] &&{\cal P}_\alpha/(\R^*_+\times\R^*)\times\GL(2,\Z)& 
\end{tikzcd}
$$
between the respective ${\cal G}$-quotients. We are interested in the fibres of $\Psi_{\alpha,r}$.

Using the same method as is the proof of the implication $(2)\Rightarrow(1)$ of Proposition \ref{orbits-similitude}, we obtain easily:
\begin{re}
The map $H_\alpha:{\cal P}_\alpha/(\R^*_+\times\R^*)\times\GL(2,\Z)\to {\cal N}_\alpha/\sim$ induced by $\eta_\alpha$ is bijective, in particular for any $(a,b)\in {\cal P}_\alpha$ we have   
$$
\Psi_{\alpha,r}^{-1}([N(\alpha,a,b)])=\Phi_{\alpha,r}^{-1}([(a,b)]).
$$
\end{re}

We will describe the fibres $\Phi_{\alpha,r}^{-1}([(a,b)])$ of $\Phi_{\alpha,r}$   using the general

\begin{re}\label{nice-remark}
Let $G$ be a group, $X$, $B$ be topological spaces endowed with  $G$-actions by homeomorphisms, $\phi:X\to B$ a $G$-equivariant map, and $\Phi: X/G\to B/G$ the induced map between the quotients. Let $b\in B$. The map 
$$\iota_b: \phi^{-1}(b)/G_b\to    \Phi^{-1}([b]_G),\ \iota_b([x]_{G_b})\edf [x]_G$$
is a continuous bijection. If the map $\phi^{-1}(b)\to \phi^{-1}(Gb)/G$ induced by the inclusion $\phi^{-1}(b)\hookrightarrow \phi^{-1}(Gb)$ is open, then $\iota_b$ is a homeomorphism.
\end{re}

It is easy to see that the openness condition in Remark \ref{nice-remark} is satisfied, so:

\begin{pr}\label{fibre-description} Let  ${\cal G}_{(a,b)}$ be the stabiliser  of $(a,b)$ in ${\cal G}$. The obvious map 
$$\bigg(\{(a,b)\}\times \mathscr{C}_{a,b,r}\times\big(\C/ \frac{b\wedge a}{r}\Z\big)\bigg)\big/{\cal G}_{(a,b)}\to \Phi_{\alpha,r}^{-1}([(a,b)])=\Psi_{\alpha,r}^{-1}([N(\alpha,a,b)])
$$
is a homeomorphism.
\end{pr}

Now note that $\{(a,b)\}\times\mathscr{C}_{a,b,r}\times\big(\C/ \frac{b\wedge a}{r}\Z\big)$ can be obviously identified with the product $\mathscr{C}_{a,b,r}\times\big(\C/ \frac{b\wedge a}{r}\Z\big)$; we let ${\cal G}_{(a,b)}$ act on the product $\mathscr{C}_{a,b,r}\times\big(\C/ \frac{b\wedge a}{r}\Z\big)$ via this identification and we endow its first factor $\mathscr{C}_{a,b,r}$ with  the   ${\cal G}_{(a,b)}$-action which makes the first projection
$$
p_1: \mathscr{C}_{a,b,r}\times\big(\C/ \frac{b\wedge a}{r}\Z\big)\to \mathscr{C}_{a,b,r}
$$
${\cal G}_{(a,b)}$-equivariant.  More precisely, for $((\mu,\rho),(K,[k]))\in {\cal G}_{a,b}$ and $([c]_{ab},[t]_{ab})\in \mathscr{C}_{a,b,r}\times\big(\C/ \frac{b\wedge a}{r}\Z\big)$ we  have
\begin{align}
((\mu,\rho),(K,[k]))\cdot ([c]_{ab},[t]_{ab})&=([\mu \rho\Cg]_{ab},[\mu\rho\tg ]_{ab}),\label{action-Gab-on-product}\\ 
((\mu,\rho),(K,[k]))\cdot ([c]_{ab})&=[\mu\rho\Cg]_{ab}, \label{action-Gab-on-Cabr}
\end{align}
where the pair $(\Cg,\tg)$ is defined by the equality  $(Ka, Kb, \Cg,\tg)=K\cdot(k\cdot  (a,b,c,t))$.

Applying Remark \ref{nice-remark} this time to the ${\cal G}_{(a,b)}$-equivariant map $p_1$, and taking into account Proposition \ref{fibre-description}, we obtain:
\begin{pr}\label{the-map-Pi}  Let $(a,b)$ be an $\alpha$-compatible pair. The fibre 	$\Phi_{\alpha,r}^{-1}([(a,b)])$ comes with a canonical surjective map
$$
\Pi:\Phi_{\alpha,r}^{-1}([(a,b)])=\bigg(\mathscr{C}_{a,b,r}\times\big(\C/ \frac{b\wedge a}{r}\Z\big)\bigg)\big/{\cal G}_{(a,b)} \to  \mathscr{C}_{a,b,r}/{\cal G}_{(a,b)} 
$$
induced by $p_1$. Its fibre over an orbit   ${\cal G}_{(a,b)} [c]_{ab}$ is canonically identified with the quotient 
$$
\big(\C/ \frac{b\wedge a}{r}\Z\big)\big/{\cal G}_{(a,b),[c]_{ab}}
$$
of $\C/ \frac{b\wedge a}{r}\Z$ by the stabiliser ${\cal G}_{(a,b),[c]_{ab}}$ of $[c]_{ab}$ in ${\cal G}_{(a,b)}$.

\end{pr}

Therefore, since the quotient topology on the finite set $\mathscr{C}_{a,b,r}/{\cal G}_{(a,b)}$ is discrete and any quotient of $\C/ \frac{b\wedge a}{r}\Z\simeq\C^*$  is obviously connected, it follows that
\begin{co} Let $(a,b)$ be an $\alpha$-compatible pair. The connected components of 	the fibre 	$\Phi_{\alpha,r}^{-1}([(a,b)])$ are parameterised by the finite set $\mathscr{C}_{a,b,r}/{\cal G}_{(a,b)}$. The connected component corresponding to an orbit ${\cal G}_{(a,b)} [c]_{ab}$ is canonically identified with   
$\big(\C/ \frac{b\wedge a}{r}\Z\big)\big/{\cal G}_{(a,b),[c]_{ab}}$.

\end{co}

Our next goal is to describe explicitly the stabiliser  ${\cal G}_{(a,b)}$ and the quotient $\mathscr{C}_{a,b,r}/{\cal G}_{(a,b)}$ which parameterises the set of connected components of the fibre $\Phi_{\alpha,r}^{-1}([(a,b)])$.

We obviously have:
\begin{equation}\label{stab(a,b)}
{\cal G}_{(a,b)}=\{((\mu,\rho),(K,[k]))\in {\cal G}|\ Ka =\mu^{-1}a,\ Kb=\rho^{-1}b \}.
\end{equation}

Put $N\edf N(\alpha,a,b)$. The formulae $Ka =\mu^{-1}a$, $Kb=\rho^{-1}b$ with $(\mu,\rho)\in \R^*_+\times\R^*$ show that $K$ belongs to the subgroup
\begin{equation}\label{PosCentrDef}
Z^+_{\GL(2,\Z)}(N)\edf \{K\in\GL(2,\Z)| KN=NK,\ Ka\in\R^*_+a\}
\end{equation}
of the centraliser 
\begin{equation*}
\begin{split}
Z_{\GL(2,\Z)}(N)\edf& \{K\in\GL(2,\Z)| KN=NK\}\\
=&\{K\in\GL(2,\Z)|\ a,\ b \hbox{ are eigenvectors for }K\}	
\end{split}
\end{equation*}
of $N$ in the group $\GL(2,\Z)$. Note that $Z^+_{\GL(2,\Z)}(N)$ is infinite cyclic, see section \ref{Z+-section}.

 $Z^+_{\GL(2,\Z)}(N)$ comes with an obvious group morphism $\vartheta :Z^+_{\GL(2,\Z)}(N)\to \R^*_+$ defined by the condition
$$
Ka=\vartheta(K)a  
$$
(i.e. $\vartheta(K)$ is the eigenvalue of $K$ corresponding to the eigenvector $a$). Note that the  eigenvalue of $K$ corresponding to the eigenvector $b$ will be $\varepsilon_K \vartheta(K)^{-1}$, where $\varepsilon_K\coloneq\det(K)\in\{\pm1\}$. Therefore formula (\ref{stab(a,b)}) becomes
$$
{\cal G}_{(a,b)}=\{(\vartheta(K)^{-1}, \varepsilon_K\vartheta(K)),(K,[k]))|\  K\in Z^+_{\GL(2,\Z)}(N),\ [k]\in\Z^2_d\}.
$$
Note that, by (\ref{action-Gab-on-product}) we have 
\begin{re}\label{Isom-for-Gab} The map
$$
Z^+_{\GL(2,\Z)}(N)\ltimes_\psi\Z^2_d\to {\cal G}_{(a,b)},\ (K,[k])\mapsto ((\vartheta(K)^{-1}, \varepsilon_K\vartheta(K)),(K,[k]))
$$  
is a group isomorphism, so ${\cal G}_{(a,b)}$ is naturally isomorphic to the semidirect product $Z^+_{\GL(2,\Z)}(N)\ltimes_\psi\Z^2_d$.  Via this isomorphism, $Z^+_{\GL(2,\Z)}(N)\ltimes_\psi\Z^2_d$ acts on the spaces $\mathscr{C}_{a,b,r}$, $\mathscr{C}_{a,b,r}\times\big(\C/ \frac{b\wedge a}{r}\Z\big)$ by the formulae
\begin{align}
(K,[k])* ([c]_{ab})&=[\varepsilon_K\Cg]_{ab}, \label{action-semid-prod-on-Cabr}\\
(K,[k])* ([c]_{ab},[t]_{ab})&=([\varepsilon_K\Cg]_{ab},[\varepsilon_K\tg ]_{ab}),\label{action-semid-prod-on-product}
\end{align}
where the pair $(\Cg,\tg)$ is defined by the equality  $(Ka, Kb, \Cg,\tg)=K\cdot(k\cdot  (a,b,c,t))$.
\end{re}

\begin{pr}\label{action-of-Z+on-Z2}
Via the bijection $\mathscr{C}_{a,b,r}\textmap{\simeq} \Z^2/(I_2-N)\Z^2$ induced by $\pi_{a,b,r}$ the subgroup $Z^+_{\GL(2,\Z)}(N)$ of the semi-direct product $Z^+_{\GL(2,\Z)}(N)\ltimes_\psi\Z^2_d$ acts on $\Z^2/(I_2-N)\Z^2$ by
\begin{equation}\label{K*[p]}
K* [p]=\frac{r}{2}\bigg( (\varepsilon_KK -I_2)\bpm n_{11}n_{12}\\ n_{21}n_{22}\epm+\varepsilon_K(I_2-N)\bpm k_{11}k_{12}\\ k_{21}k_{22}\epm\bigg)+\varepsilon_K K[p].
\end{equation} 
The first term on the right belongs to $r\Z^2$.
\end{pr}
\begin{proof}
By Remark \ref{Isom-for-Gab}	 we have
$
K\cdot [c]_{ab}=\left[\varepsilon_K C^K\right ]_{ab}$, 
where 
\begin{equation}
\begin{split}
C^K &= \frac{1}{2}\bpm (Ka)_1(Kb)_1\\(Ka)_2(Kb)_2 \epm+ K\bigg(c-\frac{1}{2}\begin{pmatrix}
 a_1b_1\\
a_2b_2	
\end{pmatrix} \bigg)  +\frac{b\wedge a}{2}\bpm k_{11}k_{12}\\ k_{21}k_{22}\epm\\
&=\frac{\varepsilon_K}{2}\bpm a_1b_1\\a_2b_2 \epm+ K\bigg(c-\frac{1}{2}\begin{pmatrix}
 a_1b_1\\
a_2b_2	
\end{pmatrix} \bigg)  +\frac{b\wedge a}{2}\bpm k_{11}k_{12}\\ k_{21}k_{22}\epm.	
\end{split}
\end{equation}
Put $P\edf \pi_{a,b,r}(\varepsilon_K C^K)$. Using  Remark \ref{c->p-remark}, we obtain:
\begin{equation*}
\begin{split}
\frac{b\wedge a}{r}&(P-p)=(I_2-N)\bigg( \varepsilon_K C^K-c\bigg)=
\\
=& (I_2-N)\bigg( \frac{1}{2}\bpm a_1b_1\\a_2b_2 \epm+\varepsilon_KK\bigg(c-\frac{1}{2}\begin{pmatrix}
 a_1b_1\\
a_2b_2	
\end{pmatrix} \bigg)  +\varepsilon_K\frac{b\wedge a}{2}\bpm k_{11}k_{12}\\ k_{21}k_{22}\epm-c\bigg)
\\
=&(I_2-N)\bigg(  (\varepsilon_KK -I_2)\bigg(c- \frac{1}{2}\bpm a_1b_1\\a_2b_2 \epm\bigg) +\varepsilon_K\frac{b\wedge a}{2}\bpm k_{11}k_{12}\\ k_{21}k_{22}\epm\bigg)
\\
=&(I_2-N)(\varepsilon_KK -I_2)\bigg(c- \frac{1}{2}\bpm a_1b_1\\a_2b_2 \epm\bigg)+\varepsilon_K\frac{b\wedge a}{2}(I_2-N)\bpm k_{11}k_{12}\\ k_{21}k_{22}\epm
\\
=&(\varepsilon_KK -I_2)(I_2-N)\bigg(c- \frac{1}{2}\bpm a_1b_1\\a_2b_2 \epm\bigg)+\varepsilon_K\frac{b\wedge a}{2}(I_2-N)\bpm k_{11}k_{12}\\ k_{21}k_{22}\epm
\\
=&b\wedge a\bigg( (\varepsilon_KK -I_2)\bigg(\frac{1}{2}\bpm n_{11}n_{12}\\ n_{21}n_{22}\epm+ \frac{1}{r}p\bigg) + \frac{\varepsilon_K}{2}  (I_2-N)\bpm k_{11}k_{12}\\ k_{21}k_{22}\epm \bigg),
\end{split}
\end{equation*}
which proves the first claim.

For the second claim, we have to prove that
\begin{equation*}
\begin{split}
& (\varepsilon_KK -I_2)\bpm n_{11}n_{12}\\ n_{21}n_{22}\epm+\varepsilon_K(I_2-N)\bpm k_{11}k_{12}\\ k_{21}k_{22}\epm =\\
&=\varepsilon_K \bigg(K\bpm n_{11}n_{12}\\ n_{21}n_{22}\epm+ \varepsilon_N\bpm k_{11}k_{12}\\ k_{21}k_{22}\epm\bigg)-\varepsilon_K\bigg(N\bpm k_{11}k_{12}\\ k_{21}k_{22}\epm+\varepsilon_K\bpm n_{11}n_{12}\\ n_{21}n_{22}\epm\bigg) \in 2\Z^2.
 \end{split}
\end{equation*}

By Lemma \ref{mod2congruence-lemma} we know that 
$$
K\bpm n_{11}n_{12}\\ n_{21}n_{22}\epm+ \varepsilon_N\bpm k_{11}k_{12}\\ k_{21}k_{22}\epm\equiv\bpm (KN)_{11}(KN)_{12}\\ (KN)_{21}(KN)_{22} \epm \hbox{ mod }2\Z^2, $$
$$ \ N\bpm k_{11}k_{12}\\ k_{21}k_{22}\epm+\varepsilon_K\bpm n_{11}n_{12}\\ n_{21}n_{22}\epm \equiv \bpm (NK)_{11}(NK)_{12}\\ (NK)_{21}(NK)_{22} \epm  \hbox{ mod }2\Z^2.
$$
Since $KN=NK$, this proves the claim.
\end{proof}
\begin{thry}\label{ConnCompOfFibre}  Let $\alpha$ be $S^+$-admissible and $(a,b)$ be an $\alpha$-compatible pair and $r\in\N^*$.
The quotient 	$\mathscr{C}_{a,b,r}/{\cal G}_{(a,b)}$ which parameterises  the connected components of the fibre 
$\Phi_{\alpha,r}^{-1}([(a,b)])=\Psi_{\alpha,r}^{-1}([N(\alpha,a,b)])$
 can be naturally identified with the quotient of 
$$\Z_{N,r}\edf \Z^2/(I_2-N)\Z^2+r\Z^2$$
 by the  group  $Z^+_{\GL(2,\Z)}(N)$ acting on $\Z_{N,r}$ by $K*[p]\edf [\varepsilon_K Kp]$.
\end{thry}

\begin{proof}
We have identified 	${\cal G}_{(a,b)}$ with the semi-direct product $Z^+_{\GL(2,\Z)}(N)\ltimes_\psi\Z^2_d$, which obviously fits in the short exact sequence
$$
0\to \Z_d^2\to Z^+_{\GL(2,\Z)}(N)\ltimes_\psi\Z^2_d\to Z^+_{\GL(2,\Z)}(N)\to \{1\}. 
$$
It folows that the quotient $\mathscr{C}_{a,b,r}/{\cal G}_{(a,b)}$ can be identified with the quotient of $\mathscr{C}_{a,b,r}/\Z_d^2$ by $Z^+_{\GL(2,\Z)}(N)$. Taking into account Proposition \ref{Z2-action} (2), we see that, via the bijection  $\mathscr{C}_{a,b,r}\textmap{\simeq}\Z^2/(I_2-N)\Z^2$ induced by $\pi_{a,b,r}$, the quotient $\mathscr{C}_{a,b,r}/\Z_d^2$ is identified with $\Z^2/(I_2-N)\Z^2+r\Z^2$. 
On the other hand, since the first term in (\ref{K*[p]}) belongs to $r\Z^2$, Proposition \ref{action-of-Z+on-Z2} shows that the induced action of $Z^+_{\GL(2,\Z)}(N)$ on $\mathscr{C}_{a,b,r}/\Z_d^2$ is given by $K* [p]=[\varepsilon_K Kp]$, as claimed.
\end{proof}
\begin{re}\label{N-acts-trivially}
The cyclic subgroup $\langle N\rangle \subset Z^+_{\GL(2,\Z)}(N)$ generated by $N$ acts trivially on $\Z_{N,r}$ via $*$. 	
\end{re}

For the description of the connected components of the fibre $\Phi_{\alpha,r}^{-1}([(a,b)])$: By Proposition \ref{the-map-Pi}, the connected component corresponding to the ${\cal G}_{(a,b)}$-orbit of  $[c]_{ab}\in \mathscr{C}_{a,b,r}$ is naturally identified with the quotient $\big(\C/ \frac{b\wedge a}{r}\Z\big)\big/{\cal G}_{(a,b),[c]_{ab}}$. Although is difficult to compute the stabiliser ${\cal G}_{(a,b),[c]_{ab}}$ explicitly in the general case, we have a good control of its image in the automorphism group of the complex manifold $\C/ \frac{b\wedge a}{r}\Z$. The subgroup $T\subset \Aut_h(\C/ \frac{b\wedge a}{r}\Z)$ generated by translations associated with torsion elements and the inversion automorphism $\iota$ can be identified with the semi-direct product $\mu_2\ltimes \Tors$, where 
$$\Tors= \frac{b\wedge a}{r}\Q\,\big/\, \frac{b\wedge a}{r}\Z\simeq \Q/\Z\simeq \{e^{2\pi i q}|\ q\in\Q\}\subset S^1$$
 is the torsion subgroup of  $\C/ \frac{b\wedge a}{r}\Z$, and $\mu_2\edf \{\pm 1\}$ acts on $\Tors$ by $\varepsilon\cdot \tau\edf \varepsilon\tau$. Therefore $T$ fits in the short exact sequence
 $$
 0\to \Tors\to T\textmap{\sigma} \mu_2\to\{1\},
 $$
 where $\sigma(\tau)=1$ ($-1$) if and only if $\tau$ preserves (respectively interchanges) the two ends of $\C/ \frac{b\wedge a}{r}\Z\simeq \C^*$.
 
\begin{re}\label{SubgrOfT} Let $H\subset T$ be a finite subgroup. 
 \begin{enumerate}
 \item If $\sigma(H)=\{1\}$, $H$ is a finite (hence cyclic) subgroup of $\Tors$. 
 \item  If $\sigma(H)=\mu_2$, $H$ a semi-direct product $\mu_2\times C$, where $C\subset \Tors$ is a finite cyclic subgroup. 
 \end{enumerate}
 \end{re} 

\begin{pr}\label{TheConnComp}
 Let $(a,b)$ be an $\alpha$-compatible pair, $[c]_{ab}\in \mathscr{C}_{a,b,r}$. and $H$ the image  of the stabiliser ${\cal G}_{(a,b),[c]_{ab}}$ in $\Aut_h(\C/ \frac{b\wedge a}{r}\Z)$. Then
 \begin{enumerate}
 	\item $H$ is a finite subgroup of $T$.
 	\item  Put $p\edf \pi_{a,b,r}(c)$. We have $\sigma(H)=\mu_2$ if and only if there exists $L\in Z^+_{\GL(2,\Z)}(N)$ with $\varepsilon_L=-1$ such that
 	$$
 	\exists L\in Z^+_{\GL(2,\Z)}(N) \hbox{ such that }\varepsilon_L=-1 \hbox{ and }  (I_2-\varepsilon_L L)p\in r\Z^2+(I_2-N)\Z^2. \eqno{(C)}
 	$$ 

 \end{enumerate} 
 \end{pr} 
 
 \begin{proof} (1) Formula (\ref{action-semid-prod-on-product}) shows that any element of ${\cal G}_{(a,b),[c]_{ab}}$ acts on $\C/ \frac{b\wedge a}{r}\Z$ by a formula of the form $[t]_{ab}\mapsto \pm [k]\cdot_c [t]_{ab}$, where $[k]$ belongs to the finite set $\Z_d^2$.\\
 
 (2) Taking into account formula (\ref{action-semid-prod-on-product}), we see that $\sigma(H)=\mu_2$ if and only if there exists $([k],K)\in Z^+_{\GL(2,\Z)}(N)\ltimes_\psi \Z_d^2$ leaving $[c]_{ab}$ invariant such that $\varepsilon_K=-1$.
 
 An element $([k],K)\in Z^+_{\GL(2,\Z)}(N)\ltimes_\psi \Z_d^2$  leaves $[c]_{ab}$ invariant, if and only if, putting $L\edf K^{-1}$, we have
  $$
 	[k]\cdot [c]_{ab}= L * [c]_{ab},  
 	$$
 i.e, using the bijection induced by $\pi_{a,b,r}$, Remark \ref{c->p-remark}, and formulae	(\ref{k.c}),(\ref{K*[p]}), 
 $$
 p+ r\bpm -k_2\\ k_1\epm \equiv \frac{r}{2}\bigg( (\varepsilon_LL -I_2)\bpm n_{11}n_{12}\\ n_{21}n_{22}\epm+\varepsilon_L(I_2-N)\bpm l_{11}l_{12}\\ l_{21}l_{22}\epm\bigg)+\varepsilon_L Lp
 $$
 mod $(I_2-N)\Z^2$. It suffices to recall that the first term on the right belongs to $r\Z^2$ by Proposition \ref{action-of-Z+on-Z2}.
 \end{proof}
 
 Now note that the quotient of $\C^*$ by the group of translations associated with 
a cyclic subgroup $C\subset \C^*$ is biholomorphic with $\C^*$, whereas the quotient of $\C^*$ by $\langle \iota\rangle$ is biholomorphic with $\C$.

Therefore
\begin{co}\label{fibre-C-C*}
Let $(a,b)$ be an $\alpha$-compatible pair and $r\in\N^*$. The connected components of   the fibre $\Phi_{\alpha,r}^{-1}([(a,b)])$ are biholomorphic to either $\C^*$ or $\C$. The connected component associated with an orbit ${\cal G}_{(a,b)}[c]_{ab}$, where $[c]_{ab}\in \mathscr{C}_{a,b,r}$, is biholomorphic with $\C$ if and only if, putting $p\edf \pi_{a,b,r}(c)$, condition $(C)$ stated in  Proposition \ref{TheConnComp} is satisfied.
\end{co}

\begin{ex}
The case $\theta=3$, i.e. $\alpha=\frac{3+\sqrt{5}}{2}$. 

As explained in the proof of Theorem \ref{Z+Calcul-eng} (3), using \cite[Section 4]{BVdM}	it follows that any $N\in\SL(2,\Z)$ with $\tr(N)=3$ is similar to the matrix $N_0:=\begin{pmatrix}
1 & 1 \\
1 &2	
\end{pmatrix}$. Therefore the set ${\cal N}_\alpha/\sim$ of similarity classes of such matrices reduces to the singleton $\{\Ng_0\}=\{[N_0]\}$.

In this case we have    $\det(I_2-N)=-1$, so the quotient $\Z_{N,r}=\Z^2/(I_2-N)\Z^2+r\Z^2$ is a singleton for any $r\in\N^*$. By Theorem \ref{ConnCompOfFibre} it follows that, for any fixed $r\in \N^*$, there exists a unique deformation class of type II Inoue surfaces $S_{a,b,c,t}^{\alpha,r}$ with $\tr(N(\alpha,a,b))=3$.

By  Theorem \ref{Z+Calcul-eng} we know that $Z^+_{\GL(2,\Z)}(N_0)$ is generated by the matrix $N'_0:=\begin{pmatrix}
0& 1 \\
1 &1	
\end{pmatrix}$ which satisfies $N_0'^2=N_0$ with $\det(N_0')=-1$. 

By Corollary \ref{fibre-C-C*}, it follows that for any fixed $r\in \N^*$, the unique  connected component of the space of isomorphism classes of type II Inoue surfaces of the form  $S_{a,b,c,t}^{\alpha,r}$ is naturally biholomporphic to $\C$.

\end{ex}

\begin{ex}
The case $\theta=4$, i.e. $\alpha=2+\sqrt{3}$.

There exists exactly two matrices with determinant 1 and  trace 4 which are reduced in the sense of  \cite[Definition 4.2]{BVdM}, namely
$$
N'=\bpm 1 & 2 \\ 1 & 3\epm,\ N''=\bpm 1 & 1 \\ 2 & 3\epm.
$$
On the other hand the reduction operator $P$ defined on  \cite[p. 10]{BVdM} maps $N'$ to $N''$. Therefore there exists a unique cycle of reduced matrices with determinant 1 and  trace 4, so, by \cite[Theorem 4.3]{BVdM},   a unique similarity class of such matrices. 

 Therefore the set ${\cal N}_\alpha/\sim$ of similarity classes of such matrices reduces to the singleton $\{\Ng'\}=\{[N']\}$.   The elementary divisors of the matrix 
$$
(I_2-N'\,|\,rI_2)=\begin{pmatrix}
0 &-2 &r&0\\
-1&-2&0&r	
\end{pmatrix}\in M_{2,4}(\Z)
$$
are $\varepsilon_1=1$, $\varepsilon_2=\gcd(2,r)$, so the quotient $\Z_{N',r}=\Z^2/(I_2-N')\Z^2+r\Z^2$ intervening in Theorem \ref{ConnCompOfFibre} is isomorphic as a $\Z$-module to $\Z_{\gcd(2,r)}$. 

On the other hand by  Theorem \ref{Z+Calcul-eng} the positive centraliser $Z^+_{\GL(2,\Z)}(N')$ is generated by $N'$, so it acts trivially on $\Z_{N',r}$ by Remark \ref{N-acts-trivially}. Therefore, by Theorem \ref{ConnCompOfFibre},   for any $r\in\N^*$ there are exactly $\gcd(2,r)$ deformation classes of type II Inoue surfaces $S_{a,b,c,t}^{\alpha,r}$, and, by Corollary \ref{fibre-C-C*}, both connected components of the space of isomorphism classes of such type II Inoue  surfaces is naturally biholomporphic to $\C^*$.

\end{ex}

\subsection{The classification of type III Inoue surfaces} \label{ClassTypeIII-section}

The classification of type III Inoue surfaces can be obtained using the same method as for type II surfaces: 
For a triple $(a,b,c)\in {\cal T}_{\alpha,r}$ put $[c]_{ab}\edf c+\frac{b\wedge a}{r}\Z^2$ and we define
\begin{align*}
\mathscr{C}_{a,b,r}&\edf {\cal C}_{a,b,r}\bigg/\frac{b\wedge a}{r}\Z^2,  \\
\mathscr{T}_{\alpha,r}&\edf \left\{(a,b,[c]_{ab})\big|\,\   (a,b)\in {\cal P}_\alpha,\ [c]_{ab}\in \mathscr{C}_{a,b,r}\right\} \simeq  {\mathcal{T}_{\alpha,r}}/\Z^2 .
\end{align*}

We let $\GL(2,\Z)$ act on ${\cal T}_{\alpha,r}$ via the same  formula (\ref{ABC}) as for type II surfaces, and with this definition the analogue of Theorem \ref{Quotient-approx} for type III surfaces remains true. Therefore the set of subgroups $G\subset\Aff(U)$ defining type III Inoue surfaces is identified with the quotient ${\cal T}_{\alpha,r}/\GL(2,\Z)$.

On the other hand we have to use a different formula for defining a $\Z^2$-action on ${\cal T}_{\alpha,r}$; this new formula is provided by the following analogue of Lemma \ref{tau(gi)tau(-1)}:
\begin{lm}\label{tau(gi)tau(-1)-}
Let $(a,b,c)$, $(a',b',c')\in {\cal T}_{\alpha,r}$ and put 
$$
\left\{
\begin{array}{lll}
g_0=g_0(\alpha),& g_i=g_i(a,b,c),&g_3=g_3(a,b,r),\\
g_0'=g_0(\alpha),& g_i'=g_i(a',b',c'),& g_3'=g_3(a',b',r)
\end{array}\right..$$
 Let $k=(k_1,k_2)\in\Z^2$, $\bpm s_0\\s_1\\s_2\epm \in\Z^3$ and $\tau\in \Aff^1_1(U)$,   
$
\bpm 
w\\
z 	
\epm \textmap{\tau} 
\bpm 
1  &0\\
\lambda	&1
\epm 
\bpm 
w\\
z 	
\epm +\bpm 
u\\ \zeta	
\epm$. 
The following conditions are equivalent:
\begin{enumerate}[(i)]
	\item 
$
\tau g'_0\tau^{-1}= g_0g_1^{k_1}g_2^{k_2}g_3^{s_0},\ \tau\circ  g_i'\circ \tau^{-1}=g_ig_3^{s_i} \hbox{ for }1\leq i\leq 2,\ \tau\circ  g_3'\circ \tau^{-1}=g_3$.
\item We have: 
$$
\left\{\begin{array}{ccl}
\lambda&=&\frac{kb}{\alpha-1}.\\
u&=&\frac{\alpha (k a)}{1-\alpha}.\\
2\zeta&=&\lambda u(1+\alpha)	-\frac{k_{1}(k_{1}-1)}{2}a_1b_1-\frac{k_{2}(k_{2}-1)}{2}a_2b_2\\&&-k_{1}k_{2}b_1a_2-kc-\frac{s_0}{r}b\wedge a.\\
a'&=&a.\\
b'&=&b.\\
c'&=&c+\frac{kb}{1+\alpha}a+\frac{\alpha(ka)}{1-\alpha}b+ \frac{b\wedge a}{r} s,\hbox{ where }s\edf \bpm s_1\\s_2\epm.\end{array}
 \right.
$$
\end{enumerate}

\end{lm}
Note  that, whereas for type II surfaces, both conditions in Lemma  \ref{tau(gi)tau(-1)} were independent of the coefficient $\zeta$  of $\tau$, this is no longer true for type III surfaces. In  Lemma \ref {tau(gi)tau(-1)-} both conditions do depend on $\zeta$.

We obtain an action
\begin{equation}
\label{k-action-on-c-and-classes-}
(k,  c )   \mapsto  k\cdot    c\edf  c+\frac{kb}{1+\alpha}a+\frac{\alpha(ka)}{1-\alpha}b= c+ b\wedge a(I_2+N)^{-1} \bpm -k_2\\ k_1\epm \end{equation}
on ${\cal C}_{a,b,r}$ and an induced action $(k,  [c]_{ab})\mapsto  k\cdot    [c]_{ab}\edf [k\cdot c]_{ab} $ on $\mathscr{C}_{a,b,r}$.

Noting that for type III surfaces we have $\det(I_2+N)=\theta\in\N^*$, we see that the obtained $\Z^2$ action on ${\cal T}_{\alpha,r}$ descends to a $\Z_\theta^2$-action.  In the same way as for type II surfaces we obtain  actions of the semi-direct product     $\GL(2,\Z)\ltimes_\psi \Z^2_\theta$ and of the product 
$$
{\cal G}\edf (\R^*_+\times\R^*)\times (\GL(2,\Z)\ltimes_\psi \Z^2_\theta)
$$
on ${\cal T}_{\alpha,r}$, and the following classification theorem:

\begin{thry}\label{main-} Let $\alpha$ be $S^-$-admissible and $(a,b,c)$, $(a',b',c')\in {\cal T}_{\alpha,r}$. 
The groups $G(\alpha,a,b,c,r)$, $G(\alpha,a',b',c',r)$ are conjugate in $\Aff(U)$ if and only if 
\begin{equation*}
\begin{split}
(a',b',[c']_{a'b'})\in & \big((\R^*_+\times\R^*)\times (\GL(2,\Z)\ltimes_\psi \Z_\theta^2)\big)\cdot (a,b,[c]_{ab}).	\\
=& \big((\R^*_+\times\R^*)\times (\GL(2,\Z)\ltimes_\varphi \Z^2)\big)\cdot (a,b,[c]_{ab}).
\end{split}	
\end{equation*}
\end{thry}
Let now
$$
\Psi_{\alpha,r}:\mathscr{Q}_{\alpha,r}^-\edf {\cal T}_{\alpha,r}/{\cal G}\to {\cal N}_\alpha/{\cal G}={\cal N}_\alpha/\GL(2,\Z)={\cal N}_\alpha/\sim,
$$
$$
\Phi_{\alpha,r}:\mathscr{Q}_{\alpha,r}^-\to {\cal P}_{\alpha,r}^-/{\cal G}={\cal P}_{\alpha,r}^-/(\R^*_+\times\R^*)\times \GL(2,\Z)
$$
be the obvious maps. Using now the  methods of section \ref{fibre-over-simclass}, we obtain

\begin{thry}\label{FibreOverSimilClass-} Let $\alpha$ be $S^-$-admissible, $(a,b)$ be an $\alpha$-compatible pair and $r\in\N^*$.
The quotient 	$\mathscr{C}_{a,b,r}/{\cal G}_{(a,b)}$ which parameterises  the connected components of the fibre 
$$\Phi_{\alpha,r}^{-1}([(a,b)])=\Psi_{\alpha,r}^{-1}([N(\alpha,a,b)])$$
 can be naturally identified with the quotient of 
$$\Z_{N,r}\edf \Z^2/(I_2+N)\Z^2+r\Z^2$$
 by the  group  $Z^+_{\GL(2,\Z)}(N)$ acting on $\Z_{N,r}$ by $K*[p]\edf [\varepsilon_K Kp]$.
\end{thry}
\section{Appendix: The positive centraliser  \texorpdfstring{$Z^+_{\GL(2,\Z)}(N)$}{Z}}\label{Z+-section}

 The ``positive centraliser"  $Z^+_{\GL(2,\Z)}(N)$ associated with a matrix $N\in {\cal N}_\alpha$ ($N\in {\cal N}_\alpha^-$), where $\alpha$ is $S^+$- (respectively $S^-$-) admissible plays a crucial role in  our results.  This group is always infinite cyclic. This follows from:

\begin{pr}\label{Z+(N)}
Let $N\in \GL(2,\Z)$ with  $\Spec_\R(N)\cap]1,+\infty[\ne\emptyset $. Fix $\alpha\in\Spec(N)\cap]1,+\infty[$ and an eigenvector $a\in \R^2\setminus\{0\}$ for the eigenvalue $\alpha$. 
\begin{enumerate}
	\item The group $Z_{\SL(2,\Z)}^+(N)\edf \{K\in\SL(2,\Z)|\ KN=NK,\ Ka\in \R^*_+a\}$ is infinite cyclic.
	\item Suppose $\alpha\not\in\Q$. Then $Z_{\GL(2,\Z)}^+(N)\edf \{K\in\GL(2,\Z)|\ KN=NK,\ Ka\in \R^*_+a\}$ is  also infinite cyclic.
\end{enumerate}
\end{pr}
\begin{proof}
(1) Note first that the group morphism 
$$
\vartheta|_{Z_{\SL(2,\Z)}^+(N)}: Z_{\SL(2,\Z)}^+(N)\to \R^*_+
$$	
is injective, so it induces an isomorphism $Z_{\SL(2,\Z)}^+(N)\textmap{\simeq} \vartheta(Z_{\SL(2,\Z)}^+(N))\subset \R^*_+ $. Recall  that any {\it closed} proper subgroup of $\R$ is either trivial or infinite cyclic. Since the Lie groups $(\R^*_+,\cdot)$, $(\R,+)$ are isomorphic, the same will hold for the subgroups of $(\R^*_+,\cdot)$.

Since the subgroup $\vartheta(Z_{\SL(2,\Z)}^+(N))$ is non-trivial (because it contains $\alpha>1$) and is proper (because is countable),  it suffices to prove that $\vartheta(Z_{\SL(2,\Z)}^+(N))$ is closed in $\R^*_+$. 

Let $b$ be an eigenvector for the second eigenvalue $\beta=\det(N)\alpha^{-1}$ of $N$. For   any $u\in \R^*_+$ let
$$
K_{(a,b)}^u\edf \begin{pmatrix}
a_1 &b_1\\
a_2& b_2	
\end{pmatrix}\begin{pmatrix}
u &0\\
0&u^{-1}	
\end{pmatrix}\begin{pmatrix}
a_1 &b_1\\
a_2& b_2	
\end{pmatrix}^{-1}\in M_2(\R)
$$
be the (unique) real (2,2)-matrix which admits $a$, $b$ as eigenvectors with eigenvalues $u$, $u^{-1}$ respectively. 

We obviously have
$$
\vartheta(Z_{\SL(2,\Z)}^+(N))= \{u\in\R^*_+| \  K_{(a,b)}^u \in M_2(\Z) \}
$$
which shows that $\vartheta(Z_{\SL(2,\Z)}^+(N))$ is closed in $\R^*_+$, because the map $u\mapsto K_{(a,b)}^u$ is continuous and $M_2(\Z)$ is closed in $M_2(\R)$. 
\vspace{2mm}\\
(2)  We claim that, under our  assumptions, the group morphism 
$$
\vartheta:Z_{\GL(2,\Z)}^+(N)\to \R^*_+
$$
is still injective. Indeed, we have
$$
\ker(\vartheta)\subset  \{I_2,   L \},
$$
where $L \in M_2(\R)$ is the matrix of the endomorphism of $\R^2$ which admits $a$, $b$ as eigenvectors with eigenvalues 1, -1 respectively. We will show that  $L\not\in M_2(\Q)$ so this matrix cannot be an element of $Z_{\GL(2,\Z)}^+(N)$. Indeed, if  $L$ belonged to  $M_2(\Q)$, its eigenspace  $\R a$ associated with the (rational) eigenvalue 1 would admit a rational generator $ q=\bpm q_1\\q_2\epm \in \Q^2\setminus\{0\}$. Since  $\R a$ is also the eigenspace of $N$ associated with the eigenvalue $\alpha$, it would follow $Nq=\alpha q$  so, since  $N\in M_2(\Q)$, it would follow $\alpha\in\Q$ which contradicts our assumption. This proves that $\vartheta$ is injective as claimed.

Since $\vartheta$ is injective, $Z_{\GL(2,\Z)}^+(N)$ is isomorphic to a subgroup of $(\R^*_+,\cdot)\simeq(\R,+)$, so it is   torsion free. On the other hand, this abelian group fits in the short exact sequence
$$
\{1\}\to Z_{\SL(2,\Z)}^+(N)\to Z_{\GL(2,\Z)}^+(N)\textmap{\det}\{\pm 1\}\to 1,
$$
which, taking into account (1), shows that it is a finitely generated abelian group of rank 1. Since it is torsion free, it is  infinite cyclic.
\end{proof}
Note that 
\begin{re}\label{ImageOfZ+inPSL-PGL}
The natural group morphisms 
$$Z_{\SL(2,\Z)}^+(N)\to  \PSL(2,\Z),\ Z_{\GL(2,\Z)}^+(N)\to \PGL(2,\Z)$$
 are injective and their images coïncide with the centralisers 
 $$Z_{\PSL(2,\Z)}([N]),\ Z_{\PGL(2,\Z)}([N])$$
  of $[N]$ in $\PSL(2,\Z)$, respectively $\PGL(2,\Z)$. 	
\end{re}
Therefore  Lemma \ref{Z+(N)} gives:

\begin{co}
In the conditions of Lemma \ref{Z+(N)}, the centraliser $Z_{\PSL(2,\Z)}(N)$ of $[N]$ in $\PSL(2,\Z)$
  is  infinite cyclic. If $\alpha\not\in\Q$, the centraliser $Z_{\PGL(2,\Z)}([N])$ of $N$ in $\PGL(2,\Z)$ is also infinite cyclic.
\end{co}
This result answers a question discussed by experts on the Mathoverflow forum, see the article  ``Centralizers  of Elements in $\SL(2,\Z)$" \cite{MathF}.

A natural problem: given a matrix $N$ satisfying the assumptions of Proposition \ref{Z+(N)}, specify a generator of the cyclic group $Z^+_{\GL(2,\Z)}(N)$. The following result answer this question in a particular case:
 \begin{thry} \label{Z+Calcul-eng}
  Let $\N\in\SL(2,\Z)$ with $\theta\edf \tr(N)\geq 3$. Suppose that 
$$\mathrm{gcd}(n_{12},n_{21},n_{22}-n_{11})=1.$$
Then:
\begin{enumerate}
\item $Z^+_{\SL(2,Z)}(N)=\{N^k|\ k\in\Z\}$, so $Z^+_{\SL(2,Z)}(N)$ is the  cyclic subgroup of $\SL(2,\Z)$ generated by $N$.
	\item If  $\theta>3$, then $Z^+_{\GL(2,Z)}(N)=Z^+_{\SL(2,Z)}(N)$.
	\item  If $\theta=3$, then $Z^+_{\GL(2,Z)}(N)$ is generated by a matrix $N'\in \GL(2,\Z)$ with $N'^2=N$, $\det(N')=-1$.
\end{enumerate} 
\end{thry}
\begin{proof}
(1) Since the characteristic polynomial of $N$ has no multiple roots in $\C$, it follows that by a classical theorem in matrix theory \cite[Theorem 4, p. 27]{We} that
$$
Z_{M_{2,2}(\C)}(N)=\{P(N)|\ P(X)\in\C[X]\}=\langle I_2,N\rangle_\C.
$$
Therefore
$$
Z_{M_{2,2}(\Q)}(N)=Z_{M_{2,2}(\C)}(N)\cap M_{2,2}(\Q)=\langle I_2,N\rangle_\C\cap M_{2,2}(\Q)=\langle I_2,N\rangle_\Q,
$$
so
\begin{align}\label{Z{M{2,2}(Z)}(N)}
Z_{M_{2,2}(\Z)}(N)=&Z_{M_{2,2}(\Q)}(N)\cap M_{2,2}(\Z)=\nonumber\\
=&\{x N+yI_2|\ (x,y)\in \Q^2,\  x N+yI_2\in M_{2,2}(\Z)\}.	
\end{align}

Let $x=\frac{m}{n}$, $y=\frac{p}{q}$ with $\gcd(m,n)=\gcd(p,q)=1$ be such that $L\edf xN+y I_2$ is an integer matrix. We have
$$
\frac{m}{n}n_{12}= l_{12}\in\Z,\ \frac{m}{n}n_{21}= l_{21}\in\Z,\ \frac{m}{n}(n_{22}-n_{11})=l_{22}-l_{11},
$$
so $n| m n_{12}$, $n| m n_{21}$, $n| m (n_{22}-n_{11})$. But $\mathrm{gcd}(n,m)=1$, so $n| n_{12}$, $n| n_{21}$ and $n| (n_{22}-n_{11})$, so 
$n| \mathrm{gcd}(n_{12},n_{21}, (n_{22}-n_{11}))$.
We assumed $\mathrm{gcd}(n_{12},n_{21}, (n_{22}-n_{11}))=1$,  so  $n=\pm 1$, so $x\in\Z$, so $yI_2=L-xN\in M_{2,2}(\Z)$, so $y\in\Z$. Therefore, by (\ref{Z{M{2,2}(Z)}(N)}), we obtain
\begin{equation}\label{Z{M{2,2}(Z)}(N)new}
Z_{M_{2,2}(\Z)}(N)=\{xN+y I_2|\ (x,y)\in\Z^2\}=\langle I_2,N\rangle_\Z.	
\end{equation}
Putting $\alpha\edf \frac{\theta+\sqrt{\theta^3-4}}{2}$, we obtain
$
\det(xN+y I_2)=(x\alpha+y)(x\alpha^{-1}+y)=x^2+\theta xy+y^2,
$
so, by (\ref{Z{M{2,2}(Z)}(N)new}),
\begin{equation}\label{Z_{GL(2,Z)}(N)new}
\begin{split}
Z_{\SL(2,\Z)}(N)&=\{xN+y I_2|\ (x,y)\in\Z^2,\ x^2+\theta xy+y^2= 1\}, \\
Z_{\GL(2,\Z)}(N)&=\{xN+y I_2|\ (x,y)\in\Z^2,\ x^2+\theta xy+y^2=\pm 1\}. 
\end{split}
\end{equation}

For $k\in \Z$ let  $(a_k,b_k)\in\Z^2$ be defined by the condition
$
N^k=a_kN+b_k I_2$, and let $(u_n)_{n\in\N}$ be the integer sequence defined by the order 2 recurrence
$$
u_0=0,\ u_1=1,\ u_{k+2}=\theta u_{k+1}-u_k.
$$
Using the formulae $N^2=\theta N-I_2$, $(N^{-1})^2=\theta N^{-1}-I_2$, it is easy to prove that
$$
(a_k,b_k)=\left\{
\begin{array}{ccc}
(u_k, -u_{k-1})&\rm for &k\geq 1\phantom{.}\\
(-u_{-k},u_{-k+1} ) &\rm for &k\leq 0.
\end{array}
\right.
$$
By \cite[Theorem 5.6.1, p. 130]{AA} and \cite[section 6.3]{AA}, it follows that, for $\theta\geq 3$, the set of solutions of the diophantine equation $x^2+\theta xy+y^2= 1$ is precisely 
$$
\{ \pm (-u_n,u_{n+1})|\ n\in\N\}\cup \{\pm (u_{n+1},-u_n)|\ n\in \N\}=\{\pm(a_k,b_k)|\ k\in\Z\}.
$$ 
Therefore, the first formula in (\ref{Z_{GL(2,Z)}(N)new}) combined with $N^k=a_kN+b_k I_2$ gives
$$
Z_{\SL(2,\Z)}(N)=\{\pm N^k|\ k\in\Z\},
$$
which implies that $Z^+_{\SL(2,\Z)}(N)=\{N^k|\ k\in\Z\}$. This proves (1).  
\\ \\
(2) follows by second formula in (\ref{Z_{GL(2,Z)}(N)new}) taking into account that,  by \cite[Theorem 6.3.1, p. 150]{AA}, for $\theta>3$, the diophantine equation $x^2+\theta xy+y^2=-1$ has no solution.  
\\ \\
(3) Using  \cite[Theorem 4.3]{BVdM} it follows that any matrix $N\in \SL(2,\Z)$ with $\tr(N)=3$ is similar to  $N_0:=\begin{pmatrix}
1 & 1 \\
1 &2	
\end{pmatrix}$. Indeed,  there exists a unique   integer matrix  with trace 3 and determinant 1 which is reduced in the sense of \cite[Definition 4.2]{BVdM}. Note that the matrix $N'_0:=\begin{pmatrix}
0& 1 \\
1 &1	
\end{pmatrix}$ is a square root of $N_0$,  belongs to $Z^+_{\GL(2,\Z)}(N_0)$ and  has $\det(N_0')=-1$. Let $K$ be  a generator of the cyclic group $Z^+_{\GL(2,\Z)}(N_0)$, and let $s\in\Z$ such that $N_0'=K^s$. Since $K^2\in Z^+_{\SL(2,\Z)}(N_0)$ (which is generated by $N_0$), there exists $t\in \Z$ such that $K^2=N_0^t$. It follow
$$
N_0=N_0'^2=(K^s)^2=K^{2s}=(K^2)^s=N_0^{ts},
$$
so $ts=1$, so $s=\pm 1$, which proves that $N_0'=K^s$ is also a generator of $Z^+_{\GL(2,\Z)}(N_0)$.

\end{proof}

\end{document}